\pdfoutput=1
\documentclass[11pt,namelimits,sumlimits]{amsart}

\usepackage{cite}

\usepackage{amssymb,amsmath}
\usepackage[mathscr]{eucal}
\usepackage{slashed}

\usepackage[latin1]{inputenc}
\usepackage[T1]{fontenc}
\usepackage{amsfonts}
\usepackage{fancyhdr}
\usepackage{graphicx}
\usepackage{amsthm}
\usepackage{amsmath,amscd}
\usepackage{latexsym}
\usepackage{cite}
\usepackage{microtype}
\usepackage{amssymb,amsmath}
\usepackage[colorinlistoftodos]{todonotes}

\usepackage{comment}
\usepackage[mathscr]{eucal}
\usepackage{times,psfrag}
\usepackage{mathtools}
\usepackage[pdftitle={Cohomogeneity one solitons in Laplacian flow}, hidelinks]{hyperref}
\usepackage{multirow}
\usepackage{setspace}
\usepackage{longtable}
\usepackage{bm}
\usepackage{fancyhdr}
\usepackage{float}
\usepackage{esint}
\usepackage{layout}
\usepackage{enumitem}
\usepackage{lmodern}
\usepackage{xcolor}
\usepackage{empheq}
\usepackage{watermark}

\usepackage[abbrev,msc-links]{amsrefs}

\numberwithin{equation}{section}
\newtheorem{theorem}[equation]{Theorem}
\newtheorem{lemma}[equation]{Lemma}
\newtheorem{prop}[equation]{Proposition}
\newtheorem{corollary}[equation]{Corollary}

\theoremstyle{definition}

\newtheorem{example}[equation]{Example}

\theoremstyle{remark}
\newtheorem{remark}[equation]{Remark}
\newtheorem*{remark*}{Remark}

\newcommand{\abs}[1]{\lvert#1\rvert}
\newcommand\norm[1]{\left\lVert#1\right\rVert}

\newcommand{\ie}{\emph{i.e.} }
\newcommand{\eg}{\emph{e.g.} }
\newcommand{\cf}{\emph{cf.} }

\newcommand{\beq}{\begin{equation}}
\newcommand{\eeq}{\end{equation}}
\newcommand{\bea}{\begin{eqnarray}}
\newcommand{\eea}{\end{eqnarray}}

\newcommand{\C}{\mathbb{C}}

\newcommand{\R}{\mathbb{R}}

\newcommand{\Z}{\mathbb{Z}}
\newcommand{\N}{\mathbb{N}}

\newcommand{\T}{\mathbb{T}}

\newcommand{\CP}{\mathbb{CP}}

\newcommand{\Sph}{\mathbb{S}}
\newcommand{\ra}{\rightarrow}

\newcommand{\vol}{\operatorname{Vol}}

\newcommand{\Real}{\operatorname{Re}}
\newcommand{\Imag}{\operatorname{Im}}

\newcommand{\Ric}{\operatorname{Ric}}

\newcommand{\hot}{{\text{h.o.t.}}}

\newcommand{\sech}{\operatorname{sech}}

\newcommand{\Aut}{\text{Aut}}
\newcommand{\aut}{\mathfrak{aut}}
\newcommand{\killing}{\mathcal{K}}

\newcommand{\tu}[1]{\textup{#1}}

\newcommand{\gtwo}{\ensuremath{\textup{G}_2}}
\newcommand{\texorpdfgtwo}{\texorpdfstring{\ensuremath{\textup{G}_2}}{G\_2}}

\newcommand{\gtstr}{\gtwo--structure}
\newcommand{\texorpdfgtstr}{\texorpdfstring{\gtwo}{G\_2}--structure}

\newcommand{\gtmfd}{\gtwo--manifold}
\newcommand{\gtmetric}{\gtwo--metric}
\newcommand{\gthol}{\gtwo--holonomy\ }

\newcommand{\hk}{hyperK\"ahler}

\newcommand{\wt}[1]{\widetilde #1}

\newcommand{\unitary}[1]{\textup{U$(#1)$}}

\newcommand{\sunitary}[1]{\textup{SU$(#1)$}}
\newcommand{\texorpdfsunitary}[1]{\texorpdfstring{\textup{SU$(#1)$}}{SU(#1)}}

\newcommand{\suthreestr}{\sunitary{3}--structure}

\newcommand{\sorth}[1]{\textup{SO$(#1)$}}
\newcommand{\Sp}[1]{\textup{Sp$(#1)$}}
\newcommand{\texorpdfSp}[1]{\texorpdfstring{\textup{Sp$(#1)$}}{Sp(#1)}}
\newcommand{\orth}[1]{\textup{O$(#1)$}}

\newcommand{\flag}{\mathbb{F}_{1,2}}
\newcommand{\Rrep}[1]{[\![#1]\!]}

\newcommand{{\isomgtc}}{\ensuremath{\sunitary{2}^3 \rtimes S_3}}

\newcommand{\tbar}{\overline{\mathsf{T}}}
\newcommand{\tpi}{\mathsf{T}_\pi}

\newcommand{\tb}{\overline{\tau}}
\newcommand{\fb}{\overline{f^2}}
\newcommand{\fbf}{\overline{f^4}}

\newcommand{\sit}{\mathscr{T}}
\newcommand{\sis}{\mathscr{S}}

\newcommand{\sif}{\mathscr{F}}
\newcommand{\sie}{\mathscr{E}}

\def\co{\colon\thinspace}

\newcounter{mtheorem}

\newtheoremstyle{mystyle}%
  {}%
  {}%
  {\itshape}%
  {}%
  {\bfseries}%
  {.}%
  { }%
  {}%

\theoremstyle{mystyle}
\newtheorem{mtheorem}[mtheorem]{Theorem}

\begin{document}

\title[Cohomogeneity-one Laplacian solitons]{Cohomogeneity-one solitons in Laplacian flow: \\local, smoothly-closing and steady solitons}

\author{
Mark Haskins
\and
Johannes Nordstr\"om}

\date{}

\maketitle

\begin{abstract}
\vspace{-1.2\baselineskip}
We initiate a systematic study of cohomogeneity-one solitons in Bryant's Laplacian flow of closed \gtstr s on a $7$-manifold, motivated by the problem of understanding finite-time singularities of that flow. %
Here we focus on solitons with symmetry groups \Sp{2} and \sunitary{3}; 
in both cases we prove the existence of continuous families of local cohomogeneity-one gradient Laplacian solitons 
and characterise which of these local solutions extend smoothly 
over their unique singular orbits. The main questions are then to determine which of these smoothly-closing solutions extend
to complete solitons and furthermore to understand the asymptotic geometry of these complete solitons.

We provide complete answers to both questions in the case of steady solitons.
Up to the actions of scaling and discrete symmetries, we show that the set of all
smoothly-closing~\sunitary{3}-invariant steady Laplacian solitons defined on a neighbourhood of the zero-section of  $\Lambda^2_-\CP^2$
is parametrised by~$\R_{\ge 0}$, the set of nonnegative reals. 
We then determine precisely which of these solutions extend to a complete soliton 
defined on the whole of $\Lambda^2_-\CP^2$.
An open interval $I=(0,c_*) \subset \R_{\ge 0}$ corresponds to complete nontrivial gradient solitons that are asymptotic to the unique~\sunitary{3}-invariant 
torsion-free~\gtwo-cone. The point $0 \in \partial I$ corresponds to the well-known Bryant--Salamon asymptotically conical torsion-free structure on $\Lambda^2_-\CP^2$ viewed as a trivial steady soliton, while the other point
$c_* \in \partial I$ corresponds to an explicit complete gradient steady soliton with exponential volume 
growth and novel asymptotic geometry. The open interval $(c_*,\infty)$ consists entirely 
of incomplete solutions. 

In addition, we find an explicit complete gradient shrinking soliton on 
 $\Lambda^2_-\Sph^4$ and $\Lambda^2_-\CP^2$. Both these shrinkers are
asymptotic to closed but non-torsion-free~\gtwo-cones. Like the
nontrivial AC gradient steady solitons on $\Lambda^2_-\CP^2$,
these shrinkers appear to be potential singularity models for finite-time
singularities of Laplacian flow.
We also compare the behaviour of the Laplacian solitons we construct to solitons in Ricci flow.
\end{abstract}

\maketitle

\tableofcontents

\section{Introduction and main results}

\subsection{Laplacian flow and torsion-free~\texorpdfgtstr s}
Bryant's Laplacian flow is  a weakly parabolic geometric flow of closed positive \mbox{$3$-forms}
(closed \mbox{\gtstr s}) on a $7$-manifold which evolves a closed $3$-form in the direction of its Hodge-Laplacian. Laplacian flow arises as an upward gradient flow for Hitchin's volume functional: its stationary points
are torsion-free~\gtstr s and these are necessarily local maxima of the volume functional.
Some further basic geometric and analytic features of Laplacian flow and its solitons will be reviewed in Section~\ref{S:LF:LS}.

Any torsion-free~\mbox{\gtstr}~induces a Ricci-flat Riemannian metric whose holonomy group reduces to a subgroup of 
the compact exceptional simple Lie group~\gtwo. Currently the only known source of Ricci-flat metrics on compact simply 
connected odd-dimensional manifolds are such~\mbox{\gtwo-holonomy} metrics. 
Several methods are now known for constructing compact~\gtwo-holonomy
manifolds~\cites{Joyce:holonomy:book,Kovalev:TCS,CHNP:II,Joyce:Karigiannis}, 
all based on methods of nonlinear elliptic PDEs and, more specifically, on variations on 
the degeneration/gluing/perturbation method pioneered by Joyce in his construction of the first compact ~\gtwo-holonomy manifolds~\cite{Joyce:G2:JDG}. Although several obstructions to the existence of~\gtwo-holonomy metrics are known, 
currently we know rather little about which compact orientable spin $7$-manifolds can admit such metrics.

A long-term goal in Laplacian flow would be to use the flow to give a parabolic approach to construct
torsion-free~\gtstr s on compact $7$-manifolds, potentially shedding some light on which manifolds can admit such structures. 
To produce torsion-free~\gtstr s via Laplacian flow one needs to establish long-time existence and convergence results for the flow. 
A major obstruction to proving long-time existence results for the flow is that in general finite-time singularities are expected to develop. However, the study of singularity formation in Laplacian flow is still in its infancy, especially
compared to the high level of development now achieved in understanding singularities in various better-known geometric flows, \eg Ricci flow and K\"ahler--Ricci flow, mean curvature flow (especially in the codimension one and Lagrangian settings), Yamabe flow, Yang--Mills flow and harmonic map heat flow. 

\subsection{Laplacian solitons}
Given a smooth $7$-manifold $M$ (compact or noncompact), a triple $(\varphi,X,\lambda)$ consisting of a \gtstr~$\varphi$, 
a vector field $X$ and a constant $\lambda \in \R$ is called a \emph{Laplacian soliton} if the triple satisfies the 
following system of partial differential equations
\[\left\{
  \begin{aligned}
  \label{eq:LS}
    d \varphi &= 0,\\
\Delta_\varphi \varphi&= \lambda \varphi + \mathcal{L}_X \varphi.
      \end{aligned}\right.\tag{LSE}
\]
The second equation is nonlinear in $\varphi$ because $\Delta_\varphi$ depends on the Hodge-star operator 
of the metric $g_\varphi$ induced by $\varphi$ and $g_\varphi$ depends nonlinearly on $\varphi$. 
Any Laplacian soliton $(\varphi,X,\lambda)$ gives rise to a self-similar solution to the Laplacian flow.
We call the soliton shrinking if $\lambda<0$, steady if $\lambda=0$ and expanding if $\lambda>0$; 
we call $\lambda$ the dilation constant of the soliton.

The Laplacian soliton equations constitute a diffeomorphism-invariant overdetermined system of nonlinear PDEs for the pair $(\varphi,X)$.
Given its overdetermined nature, a priori, it is unclear how plentiful solutions of~\eqref{eq:LS} are even locally. 
However, Bryant~\cite{Bryant:LS:local:generality} has used the methods of overdetermined PDEs to prove that locally there are in fact many solutions of~\eqref{eq:LS}
(though questions remain about the local generality of solutions when the vector field $X$ is assumed to be gradient).
The question therefore is how to produce global solutions to~\eqref{eq:LS}, by which we mean solitons on compact manifolds or 
on complete noncompact manifolds
(in the compact case the only nontrivial solitons could be expanders, though no such examples are currently known). 

There are no general analytic methods currently available that can produce global solutions to~\eqref{eq:LS}, 
a problem that one also faces in the study of Ricci solitons (except in the K\"ahler setting where recently 
complex Monge--Amp\`ere methods have enabled general analytic constructions of K\"ahler--Ricci solitons~\cites{Conlon:Deruelle:KRexpanders,Conlon:Deruelle:Sun:Kahler:solitons,conlon2020steady}). 
Therefore we are led to study solutions of~\eqref{eq:LS} on which we impose additional geometric structures that makes their 
construction more tractable. For many nonlinear geometric PDEs a natural approach is to impose a continuous group of symmetries 
on the problem. In particular, in Ricci flow most of the known solitons arise by considering cohomogeneity-one group actions, \ie 
where the generic orbit of the symmetry group has codimension one. 
Imposing such a group of symmetries reduces the system of nonlinear PDEs governing solitons 
to a system of nonlinear ODEs. 
Some of the best-known examples include: 
Hamilton's cigar soliton \cite{Hamilton:cigar}; Bryant's rotationally-invariant
steady soliton and his $1$-parameter family of expanders
\cite{Bryant:expanders}; Cao's $\unitary{n}$-invariant K\"ahler expanders and
steady solitons on $\C^n$ and the Feldman--Ilmanen-Knopf
$\unitary{n}$-invariant K\"ahler shrinkers~\cite{FIK};
the noncollapsed steady non-K\"ahler solitons found recently by Appleton~\cite{Appleton:steady:Ricci:soliton}, and the infinite discrete 
family of  complete gradient asymptotically conical shrinkers and expanders on $\R^{p} \times S^{q-1}$ (with $p,q \ge 3$ and $p+q \le 10$) 
found recently by Angenent--Knopf \cite{Angenent:Knopf:AC:RS}. 
General features of cohomogeneity-one Ricci solitons were studied in a whole series of papers by Dancer and Wang (and collaborators)~\cite{Dancer:Wang:Ricci:solitons:coh1,Dancer:Hall:Wang:cohom1:Ricci:shrinkers,Dancer:Wang:cohom:1:Ricci:solitons:nonKahler,Dancer:Wang:cohom:1:Ricci:expanders:nonKahler,BDGW:steady,BDW:Ricci:soliton:Hamiltonian}.

\subsection{Cohomogeneity-one Laplacian solitons}
The problem of understanding cohomogeneity-one Laplacian solitons can naturally be divided up into two distinct steps, 
each of which has a different character. 
The first step is to understand which Lie groups can possibly act with cohomogeneity one on closed~\gtwo-manifolds: 
for compact groups this problem has already been solved by  Cleyton--Swann~\cite{Cleyton:Swann} (see also Cleyton's thesis~\cite{Cleyton:thesis}). The methods needed here 
are (not surprisingly) mainly of a Lie-theoretic nature. The second step is 
to derive the equations for $G$-invariant Laplacian solitons  for each possible group action $G$ and then to study the resulting nonlinear system of ODEs. 

The second step is much more involved than the first.  The fundamental difficulty is that for a typical group $G$ the general 
solution of the nonlinear ODEs governing $G$-invariant Laplacian solitons does not correspond to a complete soliton 
and therefore one has to find a way to recognise which (if any) of the local solutions represent complete solutions. 
The difficulties are further compounded by two factors: 
in most cases the general solutions to the ODE systems are not explicitly available; also
we do not know in advance what types of asymptotic geometry complete noncompact solitons may exhibit. 
As a result of these issues there are no general systematic approaches to determining which of the local solutions extend to complete ones. 

Such completeness issues have also been faced in the context of previous work on cohomogeneity-one Einstein metrics or special holonomy metrics.
For compact cohomogeneity-one Einstein metrics we mention B\"ohm's pioneering work on inhomogeneous Einstein metrics on spheres~\cite{Bohm:cohom1:Einstein} 
and the work of Foscolo--Haskins~\cite{Foscolo:Haskins:NK} that constructed the first complete inhomogeneous nearly K\"ahler 6-manifolds; 
both of these papers introduced important ideas to deal with complete metrics on compact spaces. 
Closer to the issues faced in the current paper is previous work on complete noncompact cohomogeneity-one Ricci-flat or special holonomy metrics, 
especially B\"ohm's work~\cite{Bohm:BSMF} and previous work of the current authors on cohomogeneity-one~\gthol~metrics~\cite{FHN:JEMS}.

In this paper we begin a detailed study of the geometry of 
cohomogeneity-one Laplacian solitons with principal orbit type
$\CP^3$ and $\flag:= \sunitary{3}/\T^2$; these have symmetry groups \Sp{2} and \sunitary{3} respectively.
We have several reasons for singling out these two among the seven types of simply connected principal orbit that can occur for cohomogeneity-one~\gtstr s~\cite[Theorem 3.1]{Cleyton:Swann}. We explain those reasons in detail 
in our review of the basics of cohomogeneity-one~\gtstr s in Section~\ref{ss:closed:coh1:G2}.
However, the ultimate justification for our choice of principal orbit types is that they do indeed lead
to the existence of interesting complete shrinking, steady and expanding gradient Laplacian solitons. 

After developing the basic local theory of the ODE systems governing
\sunitary{3} and \Sp{2}-invariant cohomogeneity one $G_2$ solitons
(the latter can in fact be viewed as a special case of the former,
with an extra $\Z_2$ symmetry), the main focus of this paper is 
on complete~\sunitary{3}-invariant steady gradient Laplacian solitons
on~$\Lambda^2_-\CP^2$.
The sequel~\cite{HJN} focusses mainly on the \Sp{2}-invariant case,
in particular finding a $1$-parameter family of complete~\Sp{2}-invariant
expanding gradient Laplacian solitons on $\Lambda^2_-\Sph^4$
(and a related family of $\sunitary{3} \times \Z_2$-invariant expanders
on $\Lambda^2_- \CP^2$).

\subsection{Main results of the paper}
We now describe the main results of this paper. 

\subsubsection*{ODE systems and local solutions} 
On the face of it, the nonlinear ODE system governing \sunitary{3}-invariant
Laplacian solitons consists of five mixed-order differential equations  for
four unknown functions
$(f_1,f_2,f_3,u)$ defined on some interval $I \subset \R$. The triple $f=(f_1,f_2,f_3)$ determines an~\sunitary{3}-invariant~\gtstr~while $u$ determines 
the invariant vector field $X= u\, \partial_t$. 
At a given value of $t \in I$, the triple $f$ determines an $\sunitary{3}$-invariant Riemannian metric $g_f$ on the flag manifold $\sunitary{3}/\T^2$.
There are three distinct homogeneous $\Sph^2$-fibrations of $\sunitary{3}/\T^2$  
\begin{equation}
\label{eq:fibrations}
\Sph^2_j = \unitary{2}_j/\T^2 \longrightarrow \sunitary{3}/\T^2 \longrightarrow \sunitary{3}/\unitary{2}_j = \CP^2_j
\end{equation}
that arise from three different $\unitary{2}_j$ subgroups of $\sunitary{3}$ all of which contain the diagonal subgroup $\T^2 \subset \sunitary{3}$; 
the three fibres $\Sph^2_1$, $\Sph^2_2$ and $\Sph^2_3$ are mutually orthogonal with respect to any homogeneous metric $g$ on $\sunitary{3}/\T^2$,
and the restriction of $g$ to the $j$th fibre determines a homogeneous metric $g_j$ on $\Sph^2$. Moreover, 
$g$ is determined by the size of these three fibres at a single point, \ie by three positive parameters.
With appropriate normalisations each of the invariant metrics $g_j$  is just the standard round metric of sectional curvature $1$.
In other words, the geometric interpretation of the components of the triple $f$ is that the $j$th component $f_j$ determines the ``size'' of the $j$th spherical fibre $\Sph^2_j$ with respect to the homogeneous metric $g_f$.

For most purposes it is more convenient not to work directly with the mixed-order system~\eqref{eq:ODEs:SU3}
but instead to work with an equivalent real-analytic first-order system.
The variables in this first-order reformulation are the previous triple of functions  $f=(f_1, f_2, f_3)$ along with a further triple of functions $(\tau_1, \tau_2, \tau_3)$ 
that describes   the torsion $2$-form $\tau$ of the~\sunitary{3}-invariant closed\mbox{~\gtstr}~determined by the triple $f$.
The algebraic constraints on the torsion of a closed~\gtstr, \ie that the torsion $\tau$ is a $2$-form of type $14$ \eqref{eq:tau_14}, 
imply a single algebraic constraint on $(f,\tau)$. 

An \sunitary{3}-invariant Laplacian soliton 
can therefore be interpreted as an integral curve of an explicit vector field (depending on the dilation constant $\lambda$) 
on a $5$-dimensional smooth noncompact phase space $\mathcal{P}^5 \subset \R^3_> \times \R^3$.
By Remark \ref{rmk:involution},
an \Sp{2}-invariant Laplacian soliton can be interpreted as an integral curve of the restriction of this vector field 
to the $3$-dimensional invariant submanifold $\mathcal{P}^3 \subset \mathcal{P}^5$ obtained by imposing the conditions 
$f_2=f_3$ and $\tau_2=\tau_3$. 
These first-order reformulations of the ODE system for $G$-invariant Laplacian solitons immediately imply the following result about the space of all local $G$-invariant Laplacian solitons.

\begin{mtheorem}[Local $G$-invariant Laplacian solitons]
\label{mthm:local:solitons}
Consider the $G$-invariant soliton ODE system with a fixed
dilation constant $\lambda \in \R$.
Then the space of solutions to the ODE system that exist on $I \times P$ for
some interval $I$ has
\begin{enumerate}[left=0em]
\item dimension 2 for $G = \Sp{2}$ and $P = \CP^3$.
\item dimension 4 for $G = \sunitary{3}$ and $P = \sunitary{3}/T^2$;
\end{enumerate}
\end{mtheorem}

\begin{remark}
\label{rmk:scale_param_count}
Often it is more relevant to ask how many $G$-invariant solitons there are
\emph{up to scale}. Because rescaling a soliton also scales the dilation
constant (see Remark \ref{rmk:scaling}), Theorem \ref{mthm:local:solitons}(ii)
means that the spaces of $\sunitary{3}$-invariant local expanders and local
shrinkers up to scale have dimension 4, while (because the space of solitons
with fixed dilation constant 0 is already scale-invariant) the space of
$\sunitary{3}$-invariant local steady solitons up to scale has dimension 3.
Similarly, there is only a 1-parameter family of $\Sp{2}$-invariant local
steady solitons up to scale.
\end{remark}

\smallskip
Next we want to understand which of these local $G$-invariant Laplacian solitons extends to a complete Laplacian soliton. 
One aspect of the completeness issue can be studied systematically: the problem
of extending a $G$-invariant Laplacian soliton smoothly over a so-called
singular orbit, to obtain a local solution defined near the zero section of a vector bundle.
For the actions we are considering, Cleyton--Swann's work on the structure of cohomogeneity-one~\gtstr s~\cite{Cleyton:Swann} implies that 
any complete $G$-invariant Laplacian soliton must
possess a unique singular orbit, \ie a nongeneric $G$-orbit,  which in our cases is necessarily of lower dimension.
 Moreover, the structure of this lower-dimensional singular orbit is determined by $G$: 
$\CP^2 \cong \sunitary{3}/\unitary{2}$ for $G=\sunitary{3}$ and $\Sph^4 \cong \Sp{2}/\Sp{1} \times \Sp{1}$ for $G= \Sp{2}$.
To understand which of the local $G$-invariant Laplacian solitons extend to complete Laplacian solitons
(which would necessarily be defined on $\Lambda^2_-\CP^2$ or on $\Lambda^2_-\Sph^4$ respectively) 
our first task is therefore to understand which of them extend smoothly over the (unique) singular orbit. 
There is a relatively systematic way to understand such so-called smooth closure problems and applying these methods leads to the following result.

\begin{mtheorem}[Smoothly-closing $G$-invariant Laplacian solitons] Fix any $\lambda \in \R$. 
\hfill
\label{mthm:smooth:closure}
\begin{enumerate}[left=0.25em]
\item
For $G=\Sp{2}$,  among the $2$-dimensional space of local $G$-invariant Laplacian solitons with dilation constant $\lambda$ defined on 
$I \times \CP^3$ for some interval $I \subset \R$, there is a $1$-parameter family of distinct smoothly-closing $G$-invariant Laplacian solitons, \ie a $G$-invariant Laplacian soliton that extends smoothly over the zero-section $\Sph^4 \subset \Lambda^2_-\Sph^4$.
\item
For $G=\sunitary{3}$, among the $4$-dimensional space of local $G$-invariant Laplacian solitons 
with dilation constant $\lambda$ defined on 
$I \times \sunitary{3}/\T^2 $ for some interval $I \subset \R$, there is a $2$-parameter family of distinct smoothly-closing $G$-invariant Laplacian solitons,  \ie a $G$-invariant Laplacian soliton that extends smoothly over the zero-section $\CP^2 \subset \Lambda^2_-\CP^2$.
\end{enumerate}
\end{mtheorem}

\begin{remark}
In terms of the functions $f_1, f_2, f_3$, the family in (i) is simply the subfamily of the solutions in (ii) with $f_2 = f_3$. Alternatively, these solutions can be thought of as corresponding to the subset of local
\sunitary{3}-invariant solitons on $\Lambda^2_- \CP^2$ that are anti-invariant under the involution defined by multiplying fibres by $-1$, see
Remark \ref{rmk:involution_asd}.
(In fact, solutions to these equations yield solitons on the
anti-self-dual bundle of any self-dual positive Einstein 4-manifold,
see Remark~\ref{rmk:involution}. However, it is well-known that
$\Sph^4$ and $\CP^2$ are the only closed examples, though many such $4$-orbifolds exist.)
\end{remark}

Theorem~\ref{mthm:smooth:closure} holds for shrinkers, expanders and steady solitons (and the proof 
turns out to be insensitive to the sign of the dilation constant $\lambda$). 
However, because of the scaling invariance of steady solitons discussed
in Remark \ref{rmk:scale_param_count}, Theorem~\ref{mthm:smooth:closure}
implies: \\
[0.2em] \emph{Up to rescaling there is a unique smoothly-closing \Sp{2}-invariant steady soliton
and a $1$-parameter family of distinct smoothly-closing \sunitary{3}-invariant steady solitons. }

\smallskip

\begin{remark}
In the \Sp{2}-invariant case the Bryant--Salamon asymptotically conical~\gthol metric on $\Lambda^2_-\Sph^4$ 
already provides a smoothly-closing (trivial) steady soliton. It follows that any 
smoothly-closing $\Sp{2}$-invariant steady soliton must be trivial, \ie torsion-free with vanishing soliton vector field.
Therefore to find complete nontrivial steady solitons we must look at the~\sunitary{3}-invariant setting.
\end{remark}

\smallskip

The next step in searching for complete solitons is to understand which of
the local smoothly-closing solutions constructed in
Theorem~\ref{mthm:smooth:closure} are forward complete. While we do settle
this question fully in the steady case below, it is a harder problem to
tackle systematically. The existence of a large family of forward-incomplete
solutions, established by solving a singular initial value problem of a similar
class to that in Theorem \ref{mthm:smooth:closure}, demonstrates that the
forward-completeness of a given solution is certainly not automatic.

\begin{mtheorem}
\label{mthm:extinction}
Fix any $\lambda \in \R$.

\begin{enumerate}
\item There is a 4-parameter family of %
\sunitary{3}-invariant Laplacian
solitons with dilation constant $\lambda$ defined on
$(-\epsilon, 0) \times \sunitary{3}/\T^2$ such that $f_k \to 0$ as $t \to 0$
with ${f_k = O(\sqrt{-t})}$, while $\frac{f_i}{f_j}$ and $f_1f_2f_3$ converge to
positive limits as $t \to 0$.

\item
A 2-parameter subfamily has $f_2 = f_3$ and thus defines \Sp{2}-invariant
Laplacian solitons with dilation constant $\lambda$ on
$(-\epsilon, 0) \times \CP^3$.
\end{enumerate}
\end{mtheorem}

These are again local solutions, in this case defined near the extinction time,
and it is not easy to decide how far back in $t$ they can be extended or
whether they close smoothly on a singular orbit. 
Note that there are enough free parameters for the flow lines to fill an open
region in the phase space, so we should expect this type of
forward-incompleteness to be a stable property.

\subsubsection*{Asymptotically conical solitons}

One type of complete infinite ends that appears are asymptotically conical (AC)
ones.
In the shrinker case we can identify a particular solution among the smoothly-closing ones from Theorem \ref{mthm:smooth:closure} that gives rise to
an explicit complete $G$-invariant shrinking soliton of this kind.

\begin{mtheorem}
\label{mthm:shrinkers}
There exists an explicit complete noncompact \Sp{2}-invariant gradient shrinking soliton on $\Lambda^2_-\Sph^4$ with principal 
orbit $\CP^3$ and an explicit complete noncompact \sunitary{3}-invariant gradient shrinking soliton on $\Lambda^2_-\CP^2$ with principal orbit $\sunitary{3}/\T^2$. Both shrinking solitons are asymptotic to closed but non-torsion-free $G$-invariant~\gtwo-cones.
\end{mtheorem}
In terms of the triple $f$ of functions described in~\eqref{eq:fibrations} above (see also Lemmas \ref{lem:SU3str:invt} and \ref{lem:Sp2:invt:SU3:strs}
for more precise definitions) and the function $u$ determining the soliton vector field $X=u\, \partial_t$, the precise statement is that for any $b>0$ the quadruple $(f_1,f_2,f_3,u)$
\[
f_1 = t, \quad f_2 = f_3 = \sqrt{b^2 + \tfrac{1}{4}t^2}, \quad u = \frac{3t}{4b^2} + \frac{4t}{4b^2+t^2},
\]
is a complete AC shrinker with dilation constant $\lambda = -\frac{9}{4b^2}$.
The AC end behaviour is encoded by each of the $f_i$ being asymptotically
linear and the asymptotic cone is determined by the limiting values of the ratios $\frac{f_i}{f_j}$. 

Shrinking solitons are usually the rarest type of soliton, reflecting the hope/expectation that 
finite-time singularities  of a `nice' geometric flow beginning with smooth initial data on a compact manifold cannot be arbitrarily bad. 

There also turn out to exist complete $G$-invariant AC
Laplacian solitons of the other two types, \ie expanders and
steady solitons. 
However, there are important differences in the study of AC ends between the
three classes of soliton. 

\bigskip

In the steady case, one can decouple an overall scale from the soliton ODEs
to obtain an autonomous first-order system in 4 scale-normalised variables.
In this simpler system, the torsion-free cone over $\sunitary{3}/\T^2$ is the unique fixed point.
The stability of this unique fixed point, established in Lemma \ref{lem:steady:crit:pt},
has some immediate consequences.

\begin{mtheorem} For~\sunitary{3}-invariant steady solitons
\label{mthm:steady_ac_end}
\begin{enumerate}[left=0em]
\item The only possible limit cone of an AC end is the torsion-free cone.
The asymptotic rate is $-1$, except for static solutions on torsion-free
AC ends with rate $-4$.
\item There exists a 3-parameter family of such AC steady ends up to scale and time
translation (and a 1-parameter subfamily of static solutions).
\item
Any small perturbation of an initial condition that leads to an AC end
still gives an AC end.
\end{enumerate}
\end{mtheorem}

For both expanders and shrinkers, ends asymptotic to a fixed closed cone can
be expressed as solutions to an irregular singular initial value problem.
In both cases a solution always exists (so \emph{any} closed cone arises
as the limit of an AC end solution), but the sign of $\lambda$ affects
the qualitative behaviour, and in particular, the number of free
parameters in the general solution \cite[\S9.2]{HJN}.
In the shrinker case AC ends are rigid, in the sense that two AC ends with the
same limit cone must coincide;
this AC shrinker rigidity does not even rely on assuming cohomogeneity one
by the main result of Haskins--Khan--Payne \cite{HKP}.
In the expander case, however, cohomogeneity-one AC ends turn out to be
stable, analogously to what
Theorem \ref{mthm:steady_ac_end}(iii) asserts in the steady case
(see \mbox{\cite[Theorem E]{HJN}} for the \Sp{2}-invariant case).

\subsubsection*{Steady solitons}
In the steady case we can understand all possible behaviours under
forward-evolution from any initial condition. Above we have seen two types of
behaviour: incomplete ends from Theorem \ref{mthm:extinction} and AC ends from
Theorem \ref{mthm:steady_ac_end}(ii). 
There turns out to be only one more possibility, namely a certain type of
forward-complete end with exponential growth.

\begin{mtheorem}
\label{mthm:trichotomy}
Any~\sunitary{3}-invariant steady soliton satisfies exactly one of the following:
\begin{enumerate}[left=0em]
\item
It is AC (generically with rate $-1$) with asymptotic cone the torsion-free cone over the standard~\sunitary{3}-invariant nearly K\"ahler structure on $\sunitary{3}/\T^2$.
\item
It has infinite forward existence time (and therefore it gives rise to
metrically-complete end), and as $t \to \infty$ one variable $f_k$ has a finite positive limit while the other two grow exponentially, with ratio
$f_i/f_j \to 1$. 
\item
It has a finite forward maximal existence time $t_*$ (and therefore it is metrically incomplete). 
As $t \to t_*$, its smallest variable $f_k \to 0$ with $f_k = O(\sqrt{t_*-t})$.
\end{enumerate}
\end{mtheorem}

In fact, Corollary \ref{cor:exist_exp_end} gives a 2-parameter family of end solutions
of type (ii) up to scale, whose flow lines therefore form a codimension one set in the phase space. We therefore expect the solutions of type (ii) to form a ``wall'' between the two stable types (i) and (iii).

The asymptotic geometry in case (ii) can be described as an $\Sph^2$-fibration with fibres of constant size over the sinh-cone of $\CP^2$.
Recall that the sinh-cone over a compact Einstein manifold $E$
with positive Einstein constant  is the product $\R_{\ge 0} \times E$ endowed with the Riemannian metric $g= dr^2 + \sinh^2{r} \,g_E$. 
(The sinh-cone over the round $n-1$ sphere of radius $1$ yields the standard warped product description of the hyperbolic metric on $\R^n$.)
In general, the sinh-cone over $E$ is a mildly-singular Einstein space with negative Einstein constant:
it has a single isolated conical singularity at $r=0$ modelled on the Ricci-flat cone $g_C=dr^2 + r^2 g_E$, and a complete end 
with exponential volume growth as $r \to \infty$.

Now recall from \eqref{eq:fibrations} that
the flag variety $\flag=\sunitary{3}/\T^2$ can be viewed as a homogeneous
$2$-sphere fibration over $\CP^2 = \sunitary{3}/\unitary{2}$, where the base
metric
on $\CP^2$ is controlled by the coefficients $f_1$ and $f_2$ of the triple $f$,
whereas $f_3$ controls the scale of the round $2$-sphere fibre.
For the corresponding fibration of $\R_{>0} \times \flag \to \R_{>0} \times \CP^2$ of the forward-complete end metrics in Theorem \ref{mthm:trichotomy}(ii),
the base metric is well approximated by the complete end of the
sinh-cone over $\CP^2$, because 
$f_1 \simeq f_2 \simeq \sinh{t}$ for $t$ sufficiently large.

\bigskip
Theorem~\ref{mthm:smooth:closure} yields a $2$-parameter family $\mathcal{S}_{b,c}$ of smoothly-closing steady solitons
(where the parameters $b >0$, $c \in \R$ are described in
Theorem \ref{thm:SU3:smooth:closure}).
The elements of the family are pairwise non-isomorphic (although there is
an orientation-reversing isometry between $\mathcal{S}_{b,c}$ and
$\mathcal{S}_{b,-c}$, see Remark \ref{rmk:non_isom}), but 
by rescaling we may normalise so that $b=1$. 
Because $c = 0$ corresponds to the static soliton on the Bryant--Salamon
AC \gtmfd, the stability of AC steady ends asserted in
Theorem \ref{mthm:steady_ac_end}(iii) immediately implies
that $\mathcal{S}_{1,c}$ is AC for $c$ sufficiently small.
However, we can do much better than that, and decide which of the cases in the
trichotomy from Theorem \ref{mthm:trichotomy} occurs for each value of $c$.

\begin{mtheorem}
\label{mthm:complete:steady}
Among the $1$-parameter family $\mathcal{S}_{1,c}$ of smoothly-closing $\sunitary{3}$-invariant steady gradient solitons 
the $1$-parameter subfamily with $c^2 \le \tfrac{9}{2}$ consists of complete solitons all defined on $\Lambda^2_-\CP^2$, while
the $1$-parameter subfamily with $c^2>\tfrac{9}{2}$ consists entirely of
incomplete solitons (case \textup{(iii)} in Theorem \ref{mthm:trichotomy}). 
Moreover we have the following additional properties.
\begin{enumerate}[left=0.25em]
\item
Any complete $\sunitary{3}$-invariant steady soliton with principal orbit $\sunitary{3}/\T^2$ 
belongs to this family (up to scaling and discrete symmetries).
\item
$\mathcal{S}_{1,0}$ is 
the trivial steady soliton on 
the Bryant--Salamon~\gthol metric on $\Lambda^2_-\CP^2$. It is asymptotic with rate $-4$  to the cone $C_{\tu{tf}}$, the unique $\sunitary{3}$-invariant 
torsion-free~\gtwo-cone over $\sunitary{3}/\T^2$.
\item
For $0<c^2<\tfrac{9}{2}$, $\mathcal{S}_{1,c}$ is a nontrivial steady soliton on $\Lambda^2_-\CP^2$
asymptotic with rate $-1$ to the cone $C_{\tu{tf}}$ (case \textup{(i)} in Theorem
\ref{mthm:trichotomy}).
\item
For $c^2=\tfrac{9}{2}$, $\mathcal{S}_{1,c}$ is
a complete nontrivial steady soliton on $\Lambda^2_-\CP^2$ with
exponential volume growth, and asymptotically constant negative scalar
curvature (case \textup{(ii)} in Theorem \ref{mthm:trichotomy}).
\end{enumerate}
\end{mtheorem}
For $c=3/\sqrt{2}$ the solution in (iv) has the explicit expression
\begin{equation}
\label{eq:exp_steady}
f_1 = 2 \sinh{\tfrac{t}{2}}, \;f_2 = \sqrt{1 + e^t}, \;f_3  = \sqrt{1 + e^{-t}}, \quad u=\tanh{\tfrac{t}{2}}.
\end{equation}
(The Bryant--Salamon torsion-free~\gtstr~in (ii) is described in Example
\ref{ex:BS}.)

Steady solitons, being eternal solutions of the flow have features of both ancient solutions (like shrinking solitons) 
and of immortal solutions (like expanding solitons).
Compared to the well-known steady solitons in Ricci flow and K\"ahler--Ricci flow, \eg Hamilton's cigar soliton \cite{Hamilton:cigar},
Bryant's unique rotationally-invariant steady soliton in each dimension $n\ge 3$~\cite{Bryant:expanders} and
Cao's $\unitary{n}$-invariant steady K\"ahler solitons on $\C^n$~\cite{Cao:GKRS}, 
the existence of such asymptotically conical steady solitons is a distinctive feature of Laplacian flow.  
In fact, it is impossible to have a nontrivial steady soliton in Ricci flow 
that is asymptotic to a Ricci-flat cone whose cross-section is smooth
(there are of course various well-known shrinking and expanding Ricci solitons 
asymptotic to regular cones). 

The asymptotic geometry of the explicit steady soliton \eqref{eq:exp_steady},
with its exponential volume growth and 
asymptotically constant negative scalar curvature, is further removed yet from the asymptotic behaviour of steady Ricci solitons. 

\pagebreak[3]

\subsection{Organisation of the paper} 
Here we describe in some detail the structure of the paper, explain some of its overall logic and roughly how the proofs of our main theorems 
proceed.
\\[0.5em]
Section~\ref{S:LF:LS} gives a rapid account of Laplacian flow and some of its basic features; it describes some results that have been proven about it and also other aspects of Laplacian flow that remain open. This material is intended, in part, for readers familiar
with other geometric flows, like Ricci flow, but not with Laplacian flow or with solitons in Laplacian flow, 
and also in part to motivate our study of solitons in Laplacian flow. 
\\[0.5em]
Sections~\ref{s:closed:coh1:G2} and~\ref{sec:su3sp2} summarise the facts that we need about 
cohomogeneity-one closed~\gtwo-structures. 
This material is used mainly to derive the fundamental ODEs governing $G$-invariant Laplacian solitons
in Section~\ref{s:ODE:cohom1:LS}.
Much of this material appears in some 
form already in Cleyton's thesis~\cite{Cleyton:thesis} and in Cleyton--Swann~\cite{Cleyton:Swann}, but the perspective adopted in Section~\ref{ss:G2:evolve:closed} gives us a slightly different viewpoint. The main point is that
we have chosen to make systematic use of the description of a cohomogeneity-one closed~\gtwo-structure 
in terms of a $1$-parameter family of \sunitary{3}-structures (with some constraints on their intrinsic torsion). 
We have also chosen to describe in detail some aspects related to the discrete symmetries of these structures:
these play an important role in a couple of places later in the paper. 
However, our main reason for treating this material in detail, rather than simply quoting more extensively from Cleyton--Swann, 
is our desire to make the paper more accessible (and self-contained)
for those familiar with Ricci solitons, but who are perhaps (much) less familiar with~\gtwo-geometry.
\\[0.5em]
Section~\ref{s:ODE:cohom1:LS} derives the (mixed order) systems of nonlinear ODEs~\eqref{eq:ODEs:SU3} and ~\eqref{eq:ODEs:Sp2} governing \sunitary{3}-invariant and \Sp{2}-invariant Laplacian solitons. However, 
as previously mentioned, rather than work directly with these mixed-order systems, we instead prefer to work with equivalent real-analytic first-order systems that we
derive in Section~\ref{ss:order1:su3}. The variables in this reformulation are a pair of triples $f=(f_i)$ and $\tau = (\tau_i)$ defined on some 
interval $I \subset \R$. 
Any such triple $f$ (subject to satisfying one scalar differential equation) determines a closed~\sunitary{3}-invariant~\gtstr~and $\tau$ represents its torsion $2$-form. In this first-order reformulation the soliton vector field $X=u \,\partial_t$ is no longer explicit in the ODE problem, but is recovered algebraically 
from $f$, $\tau$ and the dilation constant $\lambda$. Theorem~\ref{mthm:local:solitons} follows easily from our first-order reformulation 
of the soliton ODEs.

Section~\ref{s:smooth:closure:solns} deals with the so-called smooth closure problem: understanding which of the local $G$-invariant 
solitons extend smoothly over a singular orbit.
Using the aforementioned first-order reformulations of the soliton ODEs  we show that
the initial conditions guaranteeing that a $G$-invariant Laplacian soliton close smoothly
constitute a singular initial value problem.
Its singularities  are of so-called regular type. This implies that  any formal power series solution has a nonzero 
radius of convergence and hence defines a real-analytic solution.
Analysing this initial value problem leads to the statements claimed in Theorem~\ref{mthm:smooth:closure}. 
The power series solutions for the smoothly-closing $G$-invariant Laplacian solitons can in theory be computed algorithmically in terms of 
the admissible initial data and the dilation constant $\lambda$.
The first several terms of these power series solutions are detailed in Appendix~\ref{app:sc:power:series}.
The specific forms of the first several nonzero terms of these power series solutions turn out to play important roles 
at a couple of key points in this paper. The same techniques also allow us to prove the existence of 
many incomplete $G$-invariant Laplacian solitons with prescribed singular behaviour as the solution 
approaches its maximal existence time.

In Sections~\ref{sec:scale_eqs}--\ref{S:steady} we specialise to~\sunitary{3}-invariant steady solitons.  In the steady case we 
can determine precisely which smoothly-closing steady solitons extend to complete solutions, and for every complete steady soliton
we determine its precise asymptotic geometry. Within these three sections we encounter four reformulations of the steady ODE system, 
each of which leads to a clear understanding of one or more aspects of the behaviour of steady solitons. 

Section~\ref{sec:scale_eqs} explores the consequences of the fact that in the steady case one can decouple an overall scale 
to obtain an autonomous first-order system~\eqref{eq:steady:normal} in four scale-normalised variables.
In this version of the steady system, the torsion-free cone over $\sunitary{3}/\T^2$ is the unique fixed point.
The stability of this fixed point, established in Lemma \ref{lem:steady:crit:pt}, leads readily to the proof of Theorem~\ref{mthm:steady_ac_end}, 
on the existence and qualitative properties of AC steady ends. 

The main goal of Section~\ref{s:forward:steady} is to establish Theorem~\ref{mthm:trichotomy}, which proves
that the eventual forward-time behaviour of any steady soliton must be one of three (disjoint) types.
Its proof uses a different reformulation of the steady ODE system  exploiting the fact that
the steady system possesses three conserved quantities (not present in either the shrinker or expander systems).
We use this fact to reduce the steady ODE system to a first-order ODE system~\eqref{eq:fi:steady} 
involving only the triple $f$ and a constant zero-sum triple $c_1,c_2$ and $c_3$ 
(recall that in general our soliton ODE system is a first-order system involving both the triple $f$ and its torsion $\tau$, and the 
latter involves first derivatives of $f$). 

One of the potential solution classes in Theorem~\ref{mthm:trichotomy} involves forward-complete solutions that are not AC; 
for these non-AC forward-complete solutions, instead of linear growth of all $f_i$ as in the AC case, 
the largest two coefficients $f_i$ and $f_j$ grow exponentially at the same rate, while the smallest coefficient $f_k$ is asymptotically a positive constant. 
To establish the existence and qualitative properties of these exponentially-growing steady ends, we find a  variable change
that transforms the rational first-order ODE system~\eqref{eq:fi:steady} into a polynomial (in fact, cubic) one~\eqref{eq:steady:poly},
whose coefficients depend on the constant triple $c$. In these new variables, an exponentially-growing end is transformed into 
a trajectory that is asymptotic to a (boundary) critical point of the cubic ODE system. 
Since these critical points turn out to be hyperbolic, the Stable Manifold Theorem, allows us to prove 
the existence and understand the local generality of such exponentially-growing ends. 
In particular, the hyperbolic nature of the critical point also implies that, unlike AC steady ends and forward-incomplete steady solitons, such exponentially-growing ends
form a hypersurface inside phase space, that therefore can potentially act as a wall between the previous two stable solution types.

Section~\ref{S:steady} applies the results proven in Sections~\ref{s:smooth:closure:solns},~\ref{sec:scale_eqs} and~\ref{s:forward:steady}
to prove Theorem~\ref{mthm:complete:steady} which determines the eventual forward-behaviour of every member of the $1$-parameter (up to scale) 
family of smoothly-closing~\sunitary{3}-invariant steady solitons produced by Theorem~\ref{mthm:smooth:closure}. We prove that there is a critical value of the parameter: all solutions below this value
are AC; the solution at the critical value is exponentially-growing; all solutions beyond this value are forward-incomplete. 
This vitiates our previous expectation that exponentially-growing ends can form walls between AC solutions and incomplete ones. 

The proof itself rests on an explicit determination of the critical value (informed by numerical work) and then finding an explicit 
smoothly-closing exponentially-growing steady soliton at this critical value (whose form we found by contemplation of the small-$t$ power series solutions). 
In the polynomial form of the steady ODE system this explicit solution takes a particularly simple form 
and this description motivates the final variant of the steady ODE system~\eqref{eq:steady:poly:v2}. 
One of the variables in this system, $\Lambda>0$, turns out to be central to our analysis and on the explicit solution $\Lambda$ is identically equal to $1$.  
It follows from the steady trichotomy Theorem~\ref{mthm:trichotomy} 
that only three end-time behaviours of $\Lambda$ are possible: either it tends
to~$0$ (then the solution has an AC end); it tends to~$1$
(then the solution has an exponentially-growing end); or, it tends to infinity (in which case the solution is forward-incomplete). 
The ODE system~\eqref{eq:steady:poly:v2} shows that $\Lambda$ is strictly decreasing when $\Lambda<1$ 
and smoothly-closing solutions initially have $\Lambda<1$ precisely when the parameter is below the critical value. 
By the steady trichotomy this is enough to conclude that all such solutions are AC. 
To prove incompleteness for solutions above the critical value we instead show that in our setting when $\Lambda>1$ 
then $\Lambda$ is strictly increasing and hence by the steady trichotomy $\Lambda \to \infty$ and such solutions are all forward-incomplete. 

Finally in Section~\ref{S:comparison} we compare some of the Laplacian solitons we have constructed with analogous known Ricci solitons, 
describing some key differences. 

\smallskip
\noindent
\textbf{Acknowledgements.}
Haskins's work was partially supported by Simons Collaboration grant 488620. Nordstr\"om's work was partially supported by Simons Collaboration grant 488631. 
The authors would like to thank Gavin Ball and Anna Fino for helpful discussions related to this paper.

\section{The Laplacian flow and Laplacian solitons}
\label{S:LF:LS}

\subsection{Bryant's closed Laplacian flow}
For a parabolic approach to the problem of finding torsion-free~\gtstr s on a compact oriented spin $7$-manifold it is natural to seek a geometric flow 
on positive $3$-forms. Although a number of different flows on positive $3$-forms have been considered, in this paper we discuss only what is widely considered to be the 
most promising of these flows with the nicest geometric and analytic features: Bryant's closed~\gtwo-Laplacian flow
\cite{Bryant:remarks:g2}. A $1$-parameter family of closed~\gtstr s $\varphi(t)$ evolves according to the Laplacian flow if it satisfies
\begin{equation}
\label{eq:LF}
\frac{\partial \varphi}{\partial t} = \Delta_{\varphi} \varphi
\end{equation}
where $\Delta_\varphi$ is the Hodge Laplacian on $3$-forms determined by the evolving metric $g_{\varphi(t)}$. 
Clearly any torsion-free~\gtstr~gives rise to a fixed point of Laplacian flow and on a compact manifold integration by parts shows 
that these are the only fixed points. 
For any closed~\gtstr~$\varphi$ there is a unique $2$-form $\tau$ of type $14$ with the property that $d(\ast \varphi) = \tau \wedge \varphi$. 
$\tau$ is called the torsion of $\varphi$ and it encodes all the first-order local invariants of a closed~\gtstr. Using the algebraic properties of 
$2$-forms of type $14$ it is readily seen that $\Delta_{\varphi} \varphi = d \tau$ and so in particular under Laplacian flow the cohomology class 
of $\varphi(t)$ remains constant. 
The flow of $\varphi(t)$ induces a flow of metrics $g_t:=g_{\varphi(t)}$ which has the form\footnote{The coefficient $\tfrac{1}{6}$ of the second-last term
disagrees with that in \cite[(6.15)]{Bryant:remarks:g2}, which Bryant informs us is an error.}
\begin{equation}
\label{eq:g:evolve:LF}
\frac{\partial g}{\partial t} = - 2 \Ric(g) + \frac{1}{6} \abs{\tau}^2 g + \frac{1}{4} j_\varphi (\ast_\varphi(\tau \wedge \tau)),
\end{equation}
where the map $j_\varphi: \Omega^3(M) \to S^2(T^*M)$ sends a $3$-form $\alpha$ to the symmetric covariant 2-tensor defined by $j_\varphi(V,W)(\alpha) = \ast (V \lrcorner \varphi \wedge W \lrcorner \varphi \wedge \alpha)$.
In particular, the metric $g_t$ evolves by Ricci flow with the addition of two quadratic correction terms involving the torsion $\tau$.
Since the Ricci curvature of any closed~\gtstr~is determined by $\varphi, \tau$ and $d \tau$ \cite[(4.37)]{Bryant:remarks:g2}
one can view these additional terms quadratic in $\tau$ as `lower-order' corrections. 

However, these `correction 
terms' have a profound impact on certain geometric features of the flow. For instance, a standard computation \cite[(6.14)]{Bryant:remarks:g2}
shows that the induced volume form $\vol_\varphi$ evolves via
\[
\frac{\partial}{\partial t} \vol_\varphi= \frac{1}{3} \abs{\tau_\varphi}^2 \vol_\varphi,
\]
\ie the induced volume form is pointwise increasing in $t$. This implies that for a non-torsion-free\mbox{~\gtstr} on a compact manifold the total volume $\vol_{\varphi}(M)$ 
is increasing in $t$. This is clearly very different from Ricci flow where for instance every compact Einstein manifold with positive scalar curvature
shrinks homothetically to a point in finite time. 
Hitchin provided a more geometric understanding of why the total volume should be increasing by exhibiting a gradient structure for Laplacian flow. 
More specifically,  Hitchin considered the volume functional
\[
\mathcal{V}(\varphi) := \frac{1}{7} \int_{M}{\varphi \wedge \ast \varphi} = \int_{M}{\vol_\varphi} = \vol(M,g_\varphi)
\]
and proved that for an appropriate Riemannian metric on the space of all closed~\mbox{\gtstr s} in a fixed cohomology class 
the Laplacian flow is the (upward) gradient flow of $\mathcal{V}$. Moreover, Hitchin proved that any critical point of $\mathcal{V}$ on $[\varphi]$ is a strict maximum (modulo the action of diffeomorphisms). This suggests that if one could prove 
long-time existence for solutions of Laplacian flow and if the volume is bounded above for all $t$ that 
perhaps the solution should converge as $t \to \infty$ to a torsion-free~\gtstr~$\varphi_\infty$ on $M$ 
in the original cohomology class.

Like Ricci flow, because of its diffeomorphism invariance the Laplacian flow is not strictly parabolic. 
However, by making a suitable gauge-fixing in the spirit of DeTurck and appealing to some of 
Hamilton's Nash--Moser-type methods, Bryant and Xu \cite{Bryant:Xu}, proved short-time existence (and uniqueness) for Laplacian flow on any compact manifold with any smooth closed~\gtstr~as initial data. The extra technical difficulties arise because the linearisation is parabolic only in the direction of exact forms: see also the recent note by Bedulli and Vezzoni \cite{Bedulli:Vezzoni:AGAG} observing that short-time existence can also be proven using that the fact that solutions to the gauged Laplacian flow 
fit into a general framework introduced by Hamilton in his original 1982 paper on Ricci flow \cite{Hamilton:JDG:1982}.

More recently, Laplacian flow analogues of various analytic results well known
in Ricci flow were proven by Lotay and Wei in a series of three
papers~\cites{lotay:wei:shi,Lotay:Wei:analytic,Lotay:Wei:stability}.
These include
long-time existence criteria based on curvature and torsion estimates along the flow, Shi-type estimates, uniqueness and compactness theory (the analogue of Hamilton's compactness theorem for Ricci flows), real analyticity of the the flow, and stability of critical points, \ie when the initial data $\varphi(0)$ is sufficiently close to 
a torsion-free~\gtstr~$\varphi_{\tu{tf}}$ in the same cohomology class then the solution to Laplacian flow 
exists for all time and converges modulo diffeomorphisms to $\varphi_{\tu{tf}}$.
For short introductions to many of these analytic results we refer the reader 
to the recent Fine--Yao survey article~\cite{Fine:Yao:HS:report} on hypersymplectic flow
(which can be viewed as a dimensional reduction of Laplacian flow when $M^7 = \mathbb{T}^3 \times N^4$).

\subsection{Singularity models, ancient solutions and solitons}
Recall that an ancient solution to Ricci flow is a smooth solution that exists on a time interval $(-\infty,b]$ (where $b$ could be finite or infinite; in the latter case the solution is said to be eternal). Ancient solutions to Ricci flow are fundamental to finite-time singularity analysis
because performing the natural parabolic blow-up procedure for such a singularity produces an
ancient solution. Any ancient solution of Ricci flow that arises as the blow-up limit of a finite-time singularity 
of smooth Ricci flow is called a singularity model. One also knows that any singularity model is necessarily $\kappa$-noncollapsed. 
The $\kappa$-noncollapsed condition already enables one to prove that many ancient solutions are not (finite-time) singularity models (\eg the product of Hamilton's cigar soliton with a Euclidean space).

A special class of ancient solutions to Ricci flow are provided by shrinking and steady Ricci solitons. 
By now there is a very extensive literature on (complete gradient) Ricci solitons. 
Broadly speaking, one can divide these results into three categories: 
\begin{enumerate}[label=(\alph*)]
\item
the construction of various gradient Ricci solitons;  
\item
general structural results about gradient Ricci solitons  (of both topological and geometric natures);
\item 
classification results about gradient Ricci solitons. 
\end{enumerate}
Even though gradient Ricci solitons are easier to study directly than general ancient solutions there are still many challenges
on the construction side. 
Ricci solitons satisfy an elliptic (modulo diffeomorphism) system of PDEs that generalises the Einstein equations; 
since we lack any general analytic methods to produce Einstein metrics (except in the setting of special holonomy) it is not too surprising that it has proven difficult to construct solitons by analytic methods.  Rather, as with Einstein metrics, many constructions are based on a symmetry assumption or 
other special metric ansatz (\eg a warped product structure or bundle structure) that reduces the system of PDEs to ODEs or even to algebraic equations (as for homogeneous Einstein metrics or solitons).

One can then try to leverage results about steady or shrinking solitons either to construct or prove structural or classification results for more general ancient solutions:
Perelman proved that any complete nonflat $3$-dimensional ancient $\kappa$-solution has a rescaled backward time
limit which is a nonflat gradient shrinking Ricci soliton; he also
constructed a compact rotationally-invariant ancient solution to $3$-dimensional Ricci flow 
that at large negative times resembles two steady Bryant solitons glued together to obtain a $3$-sphere 
with a long neck and which close 
to its extinction time approaches a shrinking round $3$-sphere; 
Brendle's recent proof of Perelman's conjecture on the complete classification of noncompact ancient $\kappa$-solutions in three dimensions \cite{Brendle:ancient:kappa:nc}
builds on his earlier classification result for 3-dimensional $\kappa$-noncollapsed steady gradient solitons with positive curvature~\cite{Brendle:3d:Bryant:uniqueness}.

\subsection{Finite-time singularity formation in Laplacian flow}
Currently significantly less is known about finite-time singularity formation in Laplacian flow than in Ricci flow (or 
various other better-studied geometric flows like mean curvature flow or harmonic map heat flow). 
The first difficulty is that as in Hamilton's compactness theorem for Ricci flows, the Lotay--Wei compactness results~\cite[Theorems 7.1 and 7.2]{lotay:wei:shi} assume that a uniform lower bound on the injectivity radius holds. 
In Ricci flow, Perelman's $\kappa$-noncollapsing theorem guarantees this holds at any finite-time singularity. 
Gao Chen extended Perelman's $\kappa$-noncollapsing theorem to perturbations of Ricci flow by a symmetric two-tensor $h$ under some boundedness assumptions on $h$ along the flow~\cite[Theorem 4.2]{Chen:G2:noncollapse}.
Since by~\eqref{eq:g:evolve:LF} under Laplacian flow the induced metric indeed evolves by such a perturbation 
of Ricci flow, under some assumptions on the behaviour of the torsion (need to give the required control of the perturbation term $h$) one can  
find a singularity model that is a complete nonflat torsion-free~\gtstr~with Euclidean volume growth. 
But without making such a priori assumptions on the behaviour of the torsion along Laplacian flow it is not yet known
how to pass to a singularity model at finite-time singularities.

An analogous result in Ricci flow to Chen's result in Laplacian flow is that under uniform (upper and lower) bounds on the 
scalar curvature any finite-time singularity model is a smooth nonflat complete Ricci-flat manifold with Euclidean volume growth. 
Finite-time singularities of this kind are now known to occur in $\unitary{2}$-invariant non-K\"ahler Ricci flow 
(on certain simple asymptotically cylindrical $4$-manifolds and where the singularity model is the Eguchi--Hanson metric) \cite{Appleton:U2:RF}.
However, such a uniform bound on scalar curvature is not always satisfied along a Ricci flow:
for K\"ahler-Ricci flow on closed manifolds it is known that the scalar curvature 
must blow up at any finite-time singularity. In some cases one can prove that finite-time singularities must occur 
and even identify the singularity model that appears: in a compact $\unitary{2}$-invariant K\"ahler setting
Maximo \cite{Maximo} proved that an embedded $(-1)$-sphere can collapse to a point in finite time and that the 
associated singularity model is the $\unitary{2}$-invariant Feldman--Ilmanen--Knopf K\"ahler shrinker on the one-point blowup of~$\C^2$~\cite{FIK}. 
\subsection{Ancient and eternal solutions to Laplacian flow}
Rather little work has been done so far to understand ancient solutions of the Laplacian flow. 
Homogeneous solitons in Laplacian flow are 
by far the simplest solitons to study  and the resulting problems have a Lie-theoretic flavour. 
There is now a growing literature on homogeneous Laplacian solitons 
and more generally on the evolution of homogeneous metrics under Laplacian flow~\cite{Lauret:JLMS}
using some of the techniques developed for the study of Ricci solitons and Ricci flows of homogeneous metrics.
Note that there are no nontrivial \emph{gradient} homogeneous Ricci solitons~\cite[Theorem 2.3]{Petersen:Wylie:RS:symmetry}.
We refer the reader to~\cite{Lauret:JLMS} for further references on homogeneous Laplacian solitons.

Outside the homogeneous setting 
Ball~\cite[\S 6]{Ball} has found complete nontrivial steady gradient Laplacian solitons on topological cylinders $\R \times N$, 
where $N^6$ is either the twistor space of an anti-self-dual Ricci-flat $4$-manifold $B$ or a particular $T^2$-bundle
over certain~\hk~$4$-manifolds. 
Note that a nontrivial complete noncompact steady gradient Ricci soliton must be connected at infinity~\cite[Corollary 1.1]{Munteanu:Wang:SMM} so these nontrivial steady gradient Laplacian solitons on topological cylinders 
are a new feature of Laplacian flow. 
In the first case Ball exhibits a $2$-parameter family of explicit solutions 
and when $B$ is compact (a $K3$ surface or a $4$-torus for instance) on one end the~\gtstr~is asymptotic to the product torsion-free~\gtstr~on $\R^3 \times B$ while the other end has finite volume. 
(Note that a finite volume end is not possible for a complete steady gradient Ricci soliton because by~\cite[Theorem 5.1]{Munteanu:Sesum:GRS} any end must have at least linear volume growth).
In the second case the solutions have linear volume growth at one end and cubic volume growth at the other end.

\subsection{Laplacian solitons}
Given a smooth $7$-manifold $M$ (compact or noncompact), a triple $(\varphi,X,\lambda)$ consisting of a \gtstr~$\varphi$, 
a vector field $X$ and a constant $\lambda \in \R$ is called a \emph{Laplacian soliton} if the triple satisfies the 
following system of equations
\[\left\{
  \begin{aligned}
    d \varphi &= 0,\\
\Delta_\varphi \varphi&= \lambda \varphi + \mathcal{L}_X \varphi.
      \end{aligned}\right.\tag{LSE}
\]
Basic facts about the torsion of a closed \gtstr~$\varphi$ imply that $\Delta_\varphi \varphi = d\tau$ 
where $\tau$ is the torsion $2$-form of type $14$ determined by $d(\ast \varphi) = \tau \wedge \varphi$ (see \eqref{eq:Laplacian:dtau2}). Hence 
an alternative formulation of the Laplacian soliton system is
\begin{equation}
\left\{
  \begin{aligned}
d \varphi &= 0, \\
d (\tau - \iota_X\varphi)&= \lambda \varphi.
      \end{aligned}\right.
      \label{eq:lap:soliton:v2}
\end{equation}

\begin{remark}
\label{rmk:soliton:exact}
A simple but important observation is that for any nonsteady Laplacian soliton the closed $3$-form $\varphi$ 
must actually be exact. 
\end{remark}

Any Laplacian soliton $(\varphi,X,\lambda)$ gives rise to a self-similar solution to the Laplacian flow as follows:
for any time $t\in \R$ satisfying $2\lambda t + 3>0$ we define a $1$-parameter family of closed \gtstr s~$\varphi_t$ 
with $\varphi_0 = \varphi$ that evolves by Laplacian flow by defining
\[
\varphi_t = \left(\frac{3+2\lambda t}{3}\right)^{\tfrac{3}{2}} \phi_t^* \varphi,
\]
where $\phi_t$ is the $1$-parameter family of diffeomorphisms of $M$ generated by the time-dependent vector field
$X(t) = \left(\frac{3}{3+2\lambda t} \right) X$ such that $\phi_0$ is the identity, 
\eg see~\cite[Section 9]{lotay:wei:shi}. The proof that $\varphi_t$ evolves via Laplacian flow is an elementary calculation.

Based on the behaviour of the scaling factor that appears in the definition of $\varphi_t$ one says that 
a Laplacian soliton is \emph{steady} if $\lambda=0$, \emph{expanding} if $\lambda>0$ and 
\emph{shrinking} if $\lambda<0$. 
For a shrinking soliton with $\lambda$ normalised to be $-1$ we therefore have an ancient solution to Laplacian flow 
defined on the time interval $(-\infty,\tfrac{3}{2})$, whereas for an expanding soliton with $\lambda$ normalised to be $1$ 
we have an immortal solution to Laplacian flow defined on the time interval $(-\tfrac{3}{2},\infty)$. 
Steady solitons give rise to eternal solutions to Laplacian flow defined for all $t\in \R$.

\section{Cohomogeneity-one closed \texorpdfgtstr s}
\label{s:closed:coh1:G2}

In this section we discuss generalities about cohomogeneity-one actions
and closed \gtstr s, before restricting attention to the cases of closed
\gtstr s with $\sunitary{3}$ and $\Sp{2}$ symmetries in Section \ref{sec:su3sp2}.

Section~\ref{ss:closed:g2} recalls key facts about closed~\gtstr s
and Section~\ref{ss:symmetries:g2} collects some basic facts about the automorphism group of a closed~\gtstr.
Section~\ref{ss:G2:evolve:closed} explains how to pass from certain $1$-parameter families 
of \sunitary{3}-structures to a closed~\gtstr:  we use this construction in our discussion of 
cohomogeneity-one closed~\gtstr s later. 
Section~\ref{ss:closed:coh1:G2} recalls the facts that we will 
need from the general theory of cohomogeneity-one spaces and the work of Cleyton and Swann on 
cohomogeneity-one closed~\gtstr s.

\subsection{Closed \texorpdfgtstr s}
\label{ss:closed:g2}

We recall some basic facts about closed~\gtstr s following Bryant  \cite{Bryant:remarks:g2} (to which we refer the reader for further details and proofs). For readers looking for a thorough introduction to the linear algebra associated with~\gtwo~we also recommend the notes of Salamon--Walpuski \cite{Salamon:Walpuski}.
\\[0.3em]
The first-order local invariants of a~\gtstr~$\varphi$ are all encoded in a terms of a quadruple of differential forms called the torsion forms of $\varphi$.
These torsion forms arise as components of the decomposition of the exterior derivatives of $\varphi$ and $\ast \varphi$ into their \gtwo-irreducible components. Recall that the exterior powers of the standard $7$-dimensional representation $V$ of~\gtwo~decompose as 
\begin{subequations}
\begin{align}
\Lambda^2(V^*) & = \Lambda^2_{14} \oplus \Lambda^2_{7}, \\
\Lambda^3(V^*) & = \Lambda^3_{27}\oplus  \Lambda^3_{7} \oplus \Lambda^3_{1},\\
\Lambda^4(V^*) & = \Lambda^4_{27} \oplus \Lambda^4_{7} \oplus \Lambda^4_{1},\\
\Lambda^5(V^*) & = \Lambda^5_{14} \oplus \Lambda^5_{7},
\end{align}
\end{subequations}
where the subscript denotes the dimension of the irreducible module and
\begin{subequations}
\begin{align}
\Lambda^2_7&= \{X \lrcorner \varphi\, |\, X \in V\} = \{ \omega \in \Lambda^2(V^*)\, | \ast(\varphi \wedge \omega) = 2\omega\} \cong V, \\
\Lambda^2_{14} &=\{ \omega \in \Lambda^2(V^*)\, | \,\omega \wedge\! \ast\varphi=0\} = \{ \omega \in \Lambda^2(V^*)\, | \ast(\varphi \wedge \omega) = -\omega\} \cong \mathfrak{g}_2\\
\Lambda^3_1 &= \{ r \varphi\, |\, r \in \R\} \cong \R,\\
\Lambda^3_7 &= \{ X \lrcorner \ast \!\varphi\, | \, X \in V\} \cong V,\\
\Lambda^3_{27} &= \{\gamma \in \Lambda^3(V^*) \, | \, \gamma \wedge \varphi=0, \gamma \wedge \ast \varphi =0\} \cong \tu{Sym}^2_0(V) .
\end{align}
\end{subequations}
The Hodge star gives isomorphisms $\Lambda^p_j \cong \Lambda^{7-p}_j$. In particular this gives us the 
irreducible decomposition of $\Lambda^4$ from that of $\Lambda^3$.
\begin{lemma}
For any~\gtstr~$\varphi$ on $M^7$ there exist unique differential forms $\tau_1 \in \Omega^0(M)$, $\tau_7 \in \Omega^1(M)$, $\tau_{14} \in \Omega^2_{14}(M,\varphi)$ 
and $\tau_{27} \in \Omega^3_{27}(M,\varphi)$  such that
\begin{gather*}
d \varphi = \tau_1 \ast \! \varphi + 3 \tau_7 \wedge \varphi + \ast \tau_{27},\\
d (\ast \varphi) = 4 \tau_7 \wedge \ast \varphi + \tau_{14} \wedge \varphi.
\end{gather*}
\end{lemma}
\noindent
The quadruple $(\tau_1,\tau_7,\tau_{14},\tau_{27})$ defined above can be identified with the intrinsic torsion of $\varphi$. 
We will only be interested in \emph{closed}~\gtstr s, \ie $d \varphi=0$, which by the previous lemma is equivalent to the vanishing of the torsion forms $\tau_1$, $\tau_7$, and $\tau_{27}$. Hence 
the intrinsic torsion of a closed \gtstr{} $\varphi$ can be identified with $\tau_{14} \in \Omega^2_{14}(M,\varphi)$.
In the rest of the paper, we will simply denote this by $\tau$.
It satisfies
\begin{equation}
\label{eq:tau2:closed}
d (\ast\varphi) = \tau \wedge \varphi.
\end{equation}
Since $\tau$ is of type $14$, $\ast(\tau \wedge \varphi) = -\tau$ and so $d^*_\varphi \varphi = -\!\ast \!d \!\ast \!\varphi = \tau$. Hence
\begin{equation}
\label{eq:Laplacian:dtau2}
\Delta_\varphi \varphi = d d^*_\varphi \varphi = d\tau.
\end{equation}
Since $\tau$ is of type $14$ it also satisfies $\tau \wedge \ast \varphi=0$. 
Taking the exterior derivative of both sides of this equation and using again the characterisation of $\Lambda^2_{14}$ 
as the $-1$  eigenspace of $\ast(\varphi \wedge \cdot)$ yields
\[
d \tau \wedge \ast \varphi = \abs{\tau}^2 \ast \!1.
\]
Taking the exterior derivative of~\eqref{eq:tau2:closed} implies that
\begin{equation}
\label{eq:d:tau2:7}
d \tau \wedge \varphi=0
\end{equation}
and therefore the $3$-form $d\tau$ has no type 7 component. Hence we can write 
\begin{equation}
\label{eq:d:tau2}
\Delta_\varphi \varphi = d\tau = \frac{1}{7}\abs{\tau}^2 \varphi + \gamma_{27}
\end{equation}
for some $3$-form $\gamma_{27} \in \Omega^3_{27}(M,\varphi)$.

The scalar curvature of $g_\varphi$ for a closed~\gtstr~$\varphi$ is given by
\begin{equation}
\label{scalar:curv:closed}
\tu{S}(g_\varphi) = - \tfrac{1}{2} \abs{\tau}^2.
\end{equation}
In particular its scalar curvature is nonpositive and vanishes if and only if $\varphi$ is torsion free. 

\subsection{Symmetries of closed \texorpdfgtstr s}
\label{ss:symmetries:g2}
For any \gtstr~$\varphi$ on a $7$-manifold $M$ we define its automorphism group to be
\[
\Aut_{\varphi}(M):= \{ f \in \tu{Diff}(M) |\, f^*\varphi = \varphi \}.
\]
$\Aut_{\varphi}(M)$ is a closed subgroup of $\tu{Iso}_{g_\varphi}(M)$ and 
therefore it is compact whenever $\tu{Iso}_{g_\varphi}(M)$ is, \eg when $M$ is compact.
The Lie algebra to the identity component $\Aut^0_{\varphi}(M)$ of $\Aut_{\varphi}(M)$ defined by 
\[
\aut_{\varphi}(M):= \{ X \in \mathfrak{X}(M) |\, \mathcal{L}_X\varphi =0\}
\]
is then a Lie subalgebra of the algebra of Killing fields $\killing_{\varphi}$ of $g_\varphi$.

If $M$ is compact and $\varphi$ is closed the Lie algebra $\aut_{\varphi}(M)$ must be abelian; 
if moreover $\varphi$ is exact then $\aut_{\varphi}(M)=(0)$.
The first statement was proven by Podest\`a and Raffero~\cite[\S 2]{Podesta:Raffero}, where further details about 
the possible dimensions of $\aut_\varphi(M)$ are given. 
The second statement was proven by Fowdar~\cite[Prop. 4.13]{Fowdar:S1:invariant:LF}.
For completeness and because these results do not seem to be that well known we recall their proofs; 
these results are not however used in the remainder of the paper.
\begin{lemma}
\label{lem:autg2:closed}
The Lie algebra $\aut_\varphi(M) \subseteq \killing_\varphi$ of a closed~\gtstr~ on a compact manifold $M$ is abelian
and satisfies $\dim{\aut_\varphi(M)} \le b^2(M)$. If $\varphi$ is exact then $\aut_\varphi(M)=(0)$.
\end{lemma}
\begin{proof}
When $\varphi$ is closed the image of $\aut_{\varphi}$ under the isomorphism between $\mathfrak{X}(M)$ and $\Omega^2_{7}(M)$ given by $X \mapsto  X \lrcorner \varphi$ consists of closed and closed $2$-forms:
$X \lrcorner \varphi$ is closed because, $0=\mathcal{L}_X\varphi = d(X \lrcorner \varphi)$ 
(the latter equality holding because $\varphi$ is closed)  for any $X \in \aut_\varphi(M)$. 
Then any closed $2$-form $\alpha$ of type $7$ is also coclosed, because $\ast_\varphi \alpha = 2 \alpha \wedge \varphi$ and 
$
d (\alpha \wedge \varphi) = d\alpha  \wedge \varphi + \alpha \wedge d\varphi =0.
$
Hence the $2$-form $X \lrcorner \varphi$ is $\Delta_\varphi$-harmonic.

For any harmonic form $\alpha$ and $X \in \aut_\varphi(M)$, the Lie derivative $\mathcal{L}_X\alpha$ is also harmonic (since the Laplacian commutes with isometries). 
But since $\alpha$ is closed the harmonic form $\mathcal{L}_X \alpha=d(X\lrcorner \alpha)$ is also exact. 
Hence if $M$ is compact then by the Hodge decomposition $\mathcal{L}_X\alpha=0$.
In particular for any $X, Y \in \aut_\varphi(M)$ we have $\mathcal{L}_X (Y \lrcorner \varphi)=0$ and therefore also
\[
[X,Y] \lrcorner \varphi  = \mathcal{L}_X (Y \lrcorner \varphi) - Y \lrcorner (\mathcal{L}_X\varphi) =0.
\]
Suppose now that $\varphi = d \vartheta$. Then if $M$ is compact without boundary for any $X \in \aut_\varphi(M)$ we have
\[
6 \norm{X}^2_\varphi = 6 \int_M g_\varphi (X,X) \vol_{\varphi }= \int_M {(X \lrcorner \varphi) \wedge (X \lrcorner \varphi) \wedge \varphi} = \int_M {d\left((X \lrcorner \varphi) \wedge (X \lrcorner \varphi) \wedge \vartheta\right)}=0
\]
and hence the vector field $X$ must vanish identically. 
\end{proof}
\begin{remark}
Lemma~\ref{lem:autg2:closed} implies (a) that a $7$-manifold that admits a \emph{closed}~\mbox{\gtstr}  with a nonabelian symmetry algebra $\aut_\varphi(M)$ must be noncompact and (b) that Laplacian expanders on a compact manifold must have $ \aut_\varphi(M)=(0)$.
The hypothesis that $M$ be compact in both (a) and (b) is necessary, \eg 
there are noncompact complete cohomogeneity-one torsion-free examples with nonabelian symmetries
and there are noncompact homogeneous expanders \cite{Lauret:JLMS}.
Since our interest is in constructing highly-symmetric Laplacian solitons we are therefore forced to consider noncompact manifolds $M$. 
(Recall also that there are no non-torsion free steady and no shrinking Laplacian solitons when $M$ is compact, regardless of the symmetry question.)
Clearly the assumption that $\varphi$ is closed is also necessary, since the $7$-sphere admits homogeneous~\gtstr s.
\end{remark}

\subsection{Closed \texorpdfgtstr s~from 1-parameter families of~\texorpdfsunitary{3}-structures}
\label{ss:G2:evolve:closed}
Any smooth oriented hypersurface in a $7$-manifold $M$ with a~\gtstr~$\varphi$ inherits an~$\sunitary{3}$-structure. 
Recall that an {\suthreestr} on a $6$-manifold is a pair $(\omega,\Omega)$,
where $\omega$ is a real $2$-form and $\Omega$ is a complex-valued 3-form,
pointwise equivalent to the pair
($\frac{i}{2}(dz^1 \wedge d\bar z^1 +dz^2 \wedge d\bar z^2
+ dz^3 \wedge d\bar z^3, dz^1 \wedge dz^2 \wedge dz^3)$ on $\C^3$.
This condition implies that $\omega$ and $\Omega$ are both non-degenerate
(\ie $\omega^3$ and $\Omega \wedge \overline{\Omega}$ vanish nowhere
and $\Omega$ is locally decomposable) and
satisfy the algebraic constraints 
\begin{subequations}\label{eq:SU(3):structure:Constraints}
\begin{gather}
\label{eq:omega:wedge:Omega}
\omega\wedge \Omega=0, \\
\label{eq:SU3:vol}
\tfrac{1}{6}\omega^3 = \tfrac{1}{4}\Real\Omega\wedge\Imag\Omega.
\end{gather}
\end{subequations}
(Conversely, non-degeneracy of $\Omega$ implies that it defines a canonical
almost complex structure with respect to which $\Omega$ has type (3,0)
(see Hitchin \cite{Hitchin:3forms}),
\eqref{eq:omega:wedge:Omega} implies that $\omega$ has type (1,1),
non-degeneracy of $\omega$ that the associated hermitian form is
non-degenerate, and \eqref{eq:SU3:vol} that the hermitian form has signature
$(3,0)$ or $(1,2)$.)

If we choose a family of equidistant hypersurfaces $P_t$ in $(M,\varphi)$ then for $t$ sufficiently small we can view this as giving us a $1$-parameter family of~\sunitary{3}-structures on a fixed $6$-manifold $P$. 
We can also reverse this procedure and recover a~\gtstr~from a $1$-parameter family of~\sunitary{3}-structures on $P$. 
If we start with a $1$-parameter family of homogeneous~\sunitary{3}-structures on $P$ we will obtain 
a cohomogeneity-one~\gtstr~on $I \times P$, but the method applies more generally. 
If additionally we impose some conditions on the torsion of the~\gtstr~then the torsion of the $1$-parameter family of~\sunitary{3}-structures induced
on its equidistant hypersurfaces will also satisfy some constraints on its torsion. 
This idea was popularised by Hitchin~\cite[Theorem 8]{Hitchin:stable:forms:special:metrics} in the setting of torsion-free~\gtstr s, in which case the induced~\suthreestr~is a so-called half-flat structure. 
Hitchin viewed his equations as the Hamiltonian flow of a certain functional. 
It has proven to be a useful formalism for understanding some properties of the
systems of ODEs governing cohomogeneity-one torsion-free~\gtstr s
\cites{Brandhuber,FHN:JEMS,Madsen:Salamon} and a similar idea proved to be
useful in the study of cohomogeneity-one nearly K\"ahler
$6$-manifolds~\cite{Foscolo:Haskins:NK}. 
We will be interested in the case that $\varphi$ is closed. 

Let $P$ be a fixed $6$-manifold and suppose that $(\omega,\Omega)$ is a $1$-parameter family of~\sunitary{3}-structures on $P$ depending on $t\in I \subset \R$.
Consider the \gtstr~ on $I \times P$ defined by
\begin{subequations}
\begin{align}
\label{eq:phi:evolve}
\varphi &= dt\wedge\omega + \Real\Omega,\\
\ast\varphi &=\frac{1}{2}\omega^2  -dt\wedge\Imag\Omega.
\end{align}
\end{subequations}

The exterior derivatives of $\varphi$ and $\ast \varphi$ are given by
\begin{subequations}
\begin{align}
d \varphi &=  d \Real{\Omega} + (d \omega - \partial_t \Real{\Omega}) \wedge dt,\\
d(\ast \varphi) &= \omega \wedge d \omega + (d \Imag{\Omega} + \omega \wedge \partial_t \omega) \wedge dt.
\end{align}
\end{subequations}
Hence the condition that $\varphi$ be closed is equivalent to 
\begin{subequations}\label{eq:closed:Evolution}
\begin{align}
\label{eq:closure:static}
d \Real{\Omega}&=0, \\
\label{eq:closure:dynamic}
\partial_t\Real\Omega &= d\omega.
\end{align}
\end{subequations}
We call~\eqref{eq:closure:static} the static closure condition: it imposes a restriction on the torsion of the $\sunitary{3}$-structure $(\omega,\Omega)$ 
that holds for every $t \in I$. We call~\eqref{eq:closure:dynamic} the dynamic closure condition: it imposes a condition on how $(\omega,\Omega)$ evolves with $t$.
\begin{remark*}
The torsion of a general $\sunitary{3}$-structure $(\omega,\Omega)$ takes values in a $42$-dimensional space. The   
static closure condition imposes $15$ conditions on this torsion: 
in the notation of \cite[Prop. 2.10]{ALC:from:AC} it implies that the function $\hat{w}_1=0$, the $1$-form $w_5=0$ and the primitive $(1,1)$-form $w_2=0$. 
Hence the exterior derivatives of $\omega$ and $\Omega$ satisfy
\[
\begin{aligned}
&d\omega = 3w_1 \Real\Omega  + w_3 + w_4 \wedge\omega,\\
&d\Imag\Omega =-2w_1\omega^2 + \hat{w}_2\wedge\omega,
\end{aligned}
\]
where $w_1$ is a function, $\hat{w}_2$ is a primitive $(1,1)$-form, $w_3$ is a $3$-form of type $12$ and $w_4$ is a $1$-form. 
Note that an~\suthreestr~which in addition satisfies $d\omega^2=0$ is called \emph{half-flat}: this imposes a further $6$ conditions on the torsion, namely the $1$-form $w_4$ also vanishes.
\end{remark*}

\begin{remark}
\label{rmk:discrete:sym:G2}
An important point to notice is that the Klein four-group $\Z_2 \times \Z_2$ acts naturally on the space of $\sunitary{3}$-structures satisfying the static closure condition~\eqref{eq:closure:static}.
For any $\sunitary{3}$-structure $(\omega,\Omega)$ define a pair of involutions $\tbar$ and $\tpi$ by
\begin{equation}
\label{eq:tbar:tpi}
\tbar(\omega,\Omega) = (-\omega,\overline{\Omega}), \quad \tpi(\omega,\Omega) = (\omega,-\Omega).
\end{equation}
Both $\tbar$ and $\tpi$ preserve the set of $\sunitary{3}$-structures; 
$\tbar$ and $\tpi$ commute and therefore $\tbar \circ \tpi$ defines another involution on the space of 
$\sunitary{3}$-structures. The group generated by $\tbar$ and $\tpi$ is therefore isomorphic to the Klein four-group $\Z_2 \times \Z_2$.
$\tpi(\omega,\Omega)$ induces the same orientation as  $(\omega,\Omega)$ whereas $\tbar (\omega,\Omega)$ induces the opposite orientation. All three involutions preserve the set of $\sunitary{3}$-structures satisfying the static closure 
condition~\eqref{eq:closure:static}. 

Note that the $3$-form $\varphi$ defined in~\eqref{eq:phi:evolve}  is invariant under $(t,\omega,\Omega)  \mapsto  (-t,\tbar(\omega,\Omega))$. 
In other words, if we reverse the sense of the interval $I$ then we get the same~\gtstr~$\varphi$ but parametrised in the opposite sense. 
Under $(t,\omega,\Omega) \mapsto (-t,\tpi(\omega,\Omega))$,  the ~\gtstr~ satisfies $\varphi \mapsto - \varphi$; 
but clearly this still sends a closed~\gtstr~to another closed~\gtstr.
\end{remark}

\subsection{Cohomogeneity-one closed \texorpdfgtstr s}
\label{ss:closed:coh1:G2}
In preparation for our review of the work of Cleyton and Cleyton--Swann~\cite{Cleyton:Swann} on cohomogeneity-one (closed) \gtstr s
we now present a brief summary of the requisite background from the theory of cohomogeneity-one manifolds. 
For general references on cohomogeneity-one theory we refer the 
reader to~\cite[Chapter IV]{Bredon} for the smooth aspects of the theory and to~\cite{AA} for its more Riemannian aspects.

\enlargethispage{0.5\baselineskip}

\subsubsection*{Basic cohomogeneity-one theory} 
Recall that the orbit space $M/G$ of a cohomogeneity-one isometric action of a compact Lie group $G$ is a connected Riemannian $1$-manifold (potentially) with boundary,  \ie it is either  $\Sph^1$ or an interval $I$ and there are 3 kinds of intervals: $\R$, $[0,\infty)$ or $[0,\ell]$.
Interior points of $M/G$ correspond to principal orbits and any boundary points of $I$ correspond 
to singular orbits. The isotropy group of any principal orbit is conjugate to a fixed Lie subgroup $K \subset G$, 
and the isotropy subgroup $H$ of any singular orbit has the properties that $K \subset H \subset G$ and that $H/K$ 
is diffeomorphic to a sphere. 
Moreover, there is an orthogonal representation $\rho: H \ra \orth{V}$ such that a neighbourhood of the singular orbit $G/H$ in $M$ is $G$--equivariantly diffeomorphic to a neighbourhood of the zero section of the vector bundle $G\times _{H}V \ra G/H$.
There are at most two singular orbits $G/H_1$ and $G/H_2$.  When $I=[0,\ell]$, so that there are two singular orbits, then $M$ is necessarily compact:
it is obtained by identifying the two disc bundles $G\times_{H_i}D_i$, $D_i \subset V_i$ over the singular orbits along their common boundary $G/K$.
When the orbit space is the circle $\Sph^1$ there are no singular orbits and $M$ is compact with infinite fundamental group. In the remaining two cases $I=\R$ or $I=[0,\infty)$, $M$ is noncompact.

Given any point $p \in M^o$, the open dense set of principal points in $M$, there exists a unique unit-speed geodesic $\gamma$ through $p$ that is orthogonal to every principal orbit $G/K$.
When $M/G$ is an interval $I$ the map $I \times G/K \to M$ given by $(t,gK) \mapsto g \cdot \gamma(t)$ is surjective and 
the restriction of this map to the interior $I^0$ of $I$ is a diffeomorphism onto $M^0$ with the property that the composition $\pi \circ \gamma: I^o \to M^o/G$ 
is an isometry with respect to the standard metric $dt^2$ on $I^o$ and the quotient metric on $M^o/G$ (here $\pi: M \to M/G$ denotes the orbit 
projection).
Therefore we can identify smooth $G$-invariant tensors on $M^o$ with smooth $t$-dependent $G$-invariant tensors on $G/K$. 
The latter can be determined by standard methods in representation theory.

It is a separate matter to analyse when a smooth $G$-invariant tensor on $M^o$ extends to a smooth $G$-invariant tensor on $M$.
Since this extension question depends only on the geometry of $M$ in a neighbourhood of its (at most two) singular orbits and
a neighbourhood of any singular orbit $G/H$ is diffeomorphic to a neighbourhood of the zero section of the vector bundle $G\times _{H}V \ra G/H$
one can again reduce to problems of a representation-theoretic nature: see Eschenburg--Wang~\cite[\S 1]{eschenburg:wang} for details.
In the cases of interest to us for a given principal orbit type $G/K$ there will be a unique singular orbit type $G/H$ 
and also there must be precisely one singular orbit, \ie the orbit space $M/G$ is the interval $I=[0,\infty)$
and $M$ is diffeomorphic to a homogeneous vector bundle $G\times _{H}V \ra G/H$ over the unique singular orbit.

\subsubsection*{Principal orbits of cohomogeneity-one closed \gtstr s}
Cleyton in his thesis \cite{Cleyton:thesis} and Cleyton--Swann \cite[Theorem 3.1]{Cleyton:Swann} analysed the possible principal orbits for \mbox{\gtstr s} (not necessarily closed) that admit an isometric cohomogeneity-one action of a compact connected Lie group $G$. 
The requirement that $G$ preserve the $3$-form $\varphi$ implies that the representation of the isotropy 
group $K \subset G$ on the tangent space of a principal orbit $G/K$ must occur as a subgroup of $\sunitary{3}$ on its standard \mbox{$6$-dimensional} representation on $\C^3$, 
\ie $\mathfrak{k}$ must be $\mathfrak{su}(3)$, 
$\mathfrak{u}(2)$, $\mathfrak{su}(2)$, $\mathfrak{u}(1)\oplus\mathfrak{u}(1)$, $\mathfrak{u}(1)$ or $\{0\}$.
For each of these six possible $\mathfrak{k}$ they determine what the possible $\tu{Ad}_G(K)$-invariant complements to the isotropy representation of a principal orbit. They prove that (up to finite quotients) there are seven possibilities for the topology of a principal orbit: 
$\Sph^6$, $\CP^3$, $\flag$, $\Sph^3 \times \Sph^3$, $\Sph^5 \times \Sph^1$, $\Sph^3 \times \T^3$
and $\T^6$. Except for $\Sph^3 \times \Sph^3$, which admits three homogeneous space
structures with different isotropy groups, the homogeneous space structure
$G/K$ on each topological orbit type is unique. 
Moreover, any cohomogeneity-one $7$-manifold with such principal orbits admits some cohomogeneity-one \gtstr.

In the first three cases the group $G$ is simple---$\gtwo, \Sp{2}$ or $\sunitary{3}$ respectively---with principal isotropy subgroup $K$ being \sunitary{3}, $\Sp{1} \times \unitary{1}$ or $\unitary{1} \times \unitary{1}$ respectively. 
(For the remaining four cases see~ \cite[Theorem 3.1]{Cleyton:Swann} for the list of $G$ and $K$ that arise.)
The case with isometry group~$G=\gtwo$~and principal isotropy $K=\sunitary{3}$ (which acts irreducibly on the tangent space of a principal orbit) is easy to analyse: any cohomogeneity-one closed~\gtstr~is necessarily torsion-free 
and the associated metric must be flat \cite[Theorem 8.1]{Cleyton:Swann}).
We will discuss in detail the two remaining cases where $G$ is simple:  $\Sp{2}$ or $\sunitary{3}$. 
Some of our motivations for considering these two cases are the following:
\begin{enumerate}[left=0em]
\item
The two symmetry groups $\Sp{2}$ and $\sunitary{3}$ should naturally be considered together: as pointed out by Cleyton--Swann~\cite{Cleyton:Swann} 
the ODEs governing \Sp{2}-invariant closed~\gtstr s can be regarded as specialisations of the ODEs governing 
\sunitary{3}-invariant closed~\gtstr s when some of the coefficients are set equal. An equivalent 
way of saying the same thing is that the ODEs satisfied by \sunitary{3}-invariant closed~\gtstr s on $I \times \flag$ whose induced 
Riemannian metrics possess a certain additional free orientation-reversing
isometric $\Z_2$-action are the same ODEs 
satisfied by  \Sp{2}-invariant closed~\gtstr s, see Remark \ref{rmk:involution}.
\item
Four of the seven principal orbit types arise as the principal orbit of a complete cohomogeneity-one torsion-free~\gtstr.
The standard constant~\gtstr~on $\R^7$ is cohomogeneity one with respect to the action of $\gtwo \subset \sorth{7}$ and 
clearly has $\Sph^6$ as its principal orbit type. However the induced metric is the Euclidean metric on $\R^7$ and so has trivial holonomy;  this is the only way $\Sph^6$ arises as the principal orbit of a closed cohomogeneity-one~\gtstr.
$\CP^3$, $\sunitary{3}/\T^2$ and $S^3 \times S^3$ arise as the principal orbits of the (irreducible holonomy) 
Bryant--Salamon metrics on  $\Lambda^2_-\Sph^4$, $\Lambda^2_-\CP^2$ and on the spinor bundle of $\Sph^3$ respectively.
\item
The $3$-form underlying a nonsteady closed Laplacian soliton is necessarily exact and in general it is expected that 
finite-time singularity models for Laplacian flow must be exact. For topological reasons
complete cohomogeneity-one~\gtstr s with our principal orbits are necessarily exact, whereas
the AC torsion-free~\gtstr~on the spinor bundle of $S^3$ (with principal orbit $S^3 \times S^3$) fails to be exact 
(the zero-section is a nontrivial compact associative $3$-fold).
\item
Atiyah and Witten's physics-inspired description of the potential relations between well-known 
complete noncompact torsion-free~\gtstr s and certain noncompact special Lagrangian $3$-folds in $\C^3$
admits a natural extension that
suggests the existence of some links between Laplacian solitons and  
solitons of Lagrangian mean curvature flow (LMCF) in $\C^3$. Consideration of some of the known cohomogeneity-one solitons in LMCF naturally leads one to the study of our two principal orbit types. 
Some of the Laplacian solitons we construct can therefore  be viewed as $7$-dimensional (M-theory uplifts in physics terms) analogues of known LMCF solitons in $\C^3$.
\end{enumerate}

\subsubsection*{Singular orbits of cohomogeneity-one \gtstr s}
Since $G$ is nonabelian, Lemma~\ref{lem:autg2:closed} implies that we cannot obtain compact $G$-invariant manifolds which admit closed \gtstr s.
This implies that the orbit space cannot be $\Sph^1$ or $[0,1]$ and hence there can be at most one singular orbit. 
In fact in both of our cases there must be a singular orbit (see for instance Lemma~\ref{lem:closed:su3:vol}), \ie we must have $I=[0,\infty)$. $M$ is therefore diffeomorphic to some $G$-equivariant vector bundle over the singular orbit $G/H$.

Cleyton--Swann analysed the possible singular orbits that can appear. 
It follows from their analysis \cite[Tables 2 \&3]{Cleyton:Swann} that the only spaces admitting a cohomogeneity-one action of $G= \Sp{2}$ or~$G=\sunitary{3}$
on which there exist any (metrically complete) closed cohomogeneity-one~\gtstr s are the total spaces of the vector bundles $\Lambda^2_-\Sph^4$ or $\Lambda^2_-\CP^2$ respectively, with the zero section of the vector bundle being the unique singular orbit of the $G$-action. 
Cleyton--Swann \cite[\S 9]{Cleyton:Swann} understood the conditions under which a smooth closed $G$-invariant\mbox{~\gtstr}~on the open dense set 
of principal orbits $M^0$ extends to a smooth $G$-invariant~\gtstr~on $M$. We will use these smooth extension conditions later in the paper, 
where their results will be recalled.
The asymptotically conical \mbox{\gthol~metrics} on  $\Lambda^2_-\Sph^4$ and $\Lambda^2_-\CP^2$ constructed by  Bryant--Salamon~\cite{Bryant:Salamon}
provide instances of such closed~\gtstr s.  

In this paper we will also construct complete cohomogeneity-one Laplacian shrinkers and steady solitons   
on both these vector bundles (and in the sequel we also construct complete expanders on them). For most of the examples we construct, 
 the metrics underlying these solitons will be asymptotically conical, like the well-known \mbox{\gthol} examples; unlike for AC torsion-free~\gtstr s, however,  the asymptotic cones 
of these solitons need not be torsion-free~\gtwo-cones. Indeed the complete shrinkers we construct are asymptotic to 
$G$-invariant \emph{closed but non-torsion-free} \gtwo-cones (with cross-section $\CP^3$ or $\flag$ respectively). 
We also construct complete steady solitons on  $\Lambda^2_-\CP^2$ that are not asymptotically conical: instead they have exponential volume growth.

\section{\texorpdfsunitary{3}-invariant and~\texorpdfSp{2}-invariant closed~\texorpdfgtstr s}
\label{sec:su3sp2}

In this section we write down the most general cohomogeneity-one
\sunitary{3}-invariant (respectively \Sp{2}-invariant)
closed~\gtstr s. These results are used in Section~\ref{s:ODE:cohom1:LS} to
derive the ODE system satisfied by invariant Laplacian solitons, which is
central to the rest of the paper. 

To proceed further we need to recall the description from Cleyton--Swann~\cite{Cleyton:Swann} of the invariant forms of interest on the principal orbit $G/K$ in the cases $G/K=\sunitary{3}/ \T^2$ and $G/ K=\Sp{2} / \Sp{1} \times \unitary{1}$.
We refer the reader to their paper and also to Cleyton's thesis for further details if needed.

\subsection{The flag manifold}

We begin with the description of 
the principal orbit $G/K = \sunitary{3}/\T^2$
and its  $\sunitary{3}$-invariant tensors.

$\sunitary{3}$ acts transitively on the set of ordered triples $V$ of pairwise
orthogonal lines in $\C^3$. The stabiliser of the standard triple is the
maximal torus $\T^2 \subset \sunitary{3}$ of diagonal elements, so we can
identify $\sunitary{3}/\T^2$ with $V$. Picking one of the three lines of an
element of $V$ defines a map $V \to \CP^2$. This yields three different
ways to write $V$ as an $\Sph^2$-fibre bundle over $\CP^2$.
Equivalently, we could consider the following three diagonal $\sunitary{3}$
matrices
\begin{equation}
\label{eq:U2:i}
I_1=
\begin{pmatrix}
1 & 0 & 0\\
0 & \hspace{-0.5em}-1 & 0\\ 
0 & 0 & \hspace{-0.5em}-1
\end{pmatrix}, 
\quad 
I_2 = \begin{pmatrix}
-1 & 0 & 0\\
0 & 1 & 0\\ 
0 & 0 & \hspace{-0.5em}-1
\end{pmatrix}, 
\quad
I_3 = 
\begin{pmatrix}
-1 & 0 & 0\\
0 & -1 & 0\\ 
0 & 0 & 1
\end{pmatrix}.
\end{equation}
Since all three matrices have two coincident eigenvalues their centralisers in $\sunitary{3}$,  which we denote by $\unitary{2}_i$, are each subgroups containing the (diagonal) maximal torus $\T^2 \subset \sunitary{3}$ and isomorphic to $\unitary{2}$.
These three different $\unitary{2}$ subgroups give rise to three distinct homogeneous $\Sph^2$-fibrations of $\sunitary{3}/\T^2$  
\[
\Sph^2_j = \unitary{2}_j/\T^2 \longrightarrow \sunitary{3}/\T^2 \longrightarrow \sunitary{3}/\unitary{2}_j = \CP^2_j.
\]
Each of these bundles can alternatively be described as the unit sphere
bundle in $\Lambda^2_- \CP^2$.

Let $W$ be the Weyl group of $\sunitary{3}$, \ie $W=N(\T^2)/\T^2$ where $N(\T^2)$ denotes the normaliser of~$\T^2$ in $\sunitary{3}$. 
By standard Lie theory $W$ is isomorphic to the symmetric group on $3$ letters:
we can take the cosets that contain 
\begin{equation}
\label{eq:genW}
a_{231}=
\begin{pmatrix}
0 & 0 & 1\\
1 & 0 & 0\\
0 & 1 & 0
\end{pmatrix}
\quad \text{and}\quad 
a_{132} = 
\begin{pmatrix}
-1 & 0 & 0\\
0 & 0 & 1\\
0 & 1 & 0
\end{pmatrix}
\end{equation}
as order $3$ and order $2$ generators respectively of $W$.
The natural action of the normaliser $N(\T^2)$ on~$\sunitary{3}/\T^2$, \ie $(n,g\T^2) \mapsto  g\T^2 \cdot n^{-1} = gn^{-1}\T^2$, induces a free action of $W$ on $\sunitary{3}/\T^2$ (commuting with the \sunitary{3} action) by
diffeomorphisms and hence also on all invariant tensors on $\sunitary{3}/\T^2$. 
In terms of $V$, this is simply the action of changing the ordering of a
triple of orthogonal lines. Note that each involution in $W$ preserves
exactly one of the three maps to $\CP^2$, acting as the antipodal map on
each fibre.

When $G=\sunitary{3}$ the principal isotropy group $K=\T^2=\Sph^1_1 \times \Sph^1_2$ acts on the standard representation on $\C^3$ 
as $L_1 + L_2 + \overline{L}_1 \overline{L}_2$ where $L_i$ are the standard (complex) representations of $\Sph^1_i \simeq \unitary{1}$.
The isotropy representation is $\Rrep{L_1 \overline{L}_2}+\Rrep{L_1 L_2^2} + \Rrep{L_1^2L_2}$, where
$\Rrep{L}$ denotes the real representation such that $\Rrep{L} \otimes_\R \C = L \oplus \overline{L}$.
Each of the three irreducible submodule of the isotropy representation carries an invariant metric $g_i$ and an invariant $2$-form $\omega_i$ for $i=1,2,3$, while
the space of invariant $3$-forms is $2$-dimensional. 
To be more concrete we identify $\T^2$ with the diagonal matrices in \sunitary{3} and fix the following basis $E_1, \ldots ,E_6$ of the tangent space at the origin
\[
\begin{aligned}
E_1 &= \frac{1}{2}\begin{pmatrix}
0 & 0 & 0\\
0 & 0 & -1\\
0 & 1 & 0
\end{pmatrix}, 
\qquad 
& E_3 &= \frac{1}{2}\begin{pmatrix}
0 & 0 & 1\\
0 & 0 & 0\\
-1 & 0 & 0
\end{pmatrix}, 
\qquad 
& E_5 = \frac{1}{2}\begin{pmatrix}
0 & -1 & 0\\
1 & 0 & 0\\
0 & 0 & 0
\end{pmatrix}, \\
E_2 &= \frac{1}{2i}\begin{pmatrix}
0 & 0 & 0\\
0 & 0 & 1\\
0 & 1 & 0
\end{pmatrix}, 
\quad 
& E_4 &= \frac{1}{2i}\begin{pmatrix}
0 & 0 & 1\\
0 & 0 & 0\\
1 & 0 & 0
\end{pmatrix}, 
\qquad 
& E_6 = \frac{1}{2i}\begin{pmatrix}
0 & 1 & 0\\
1 & 0 & 0\\
0 & 0 & 0
\end{pmatrix}.
\end{aligned}
\]
Then the invariant bilinear forms and $2$-forms are spanned by
\begin{gather*}
g_1= e_1^2+e_2^2, \quad g_2 = e_3^2+e_4^2, \quad g_3 = e_5^2 + e_6^2, \qquad\\
\intertext{and}
\omega_1 = e_{12}, \quad \omega_2  = e_{34}, \quad \omega_3 = e_{56}
\end{gather*}
respectively where $\{e_1, \ldots ,e_6\}$ denotes the dual basis to $\{E_i\}$.
In particular an arbitrary $\sunitary{3}$-invariant Riemannian metric on $\sunitary{3}/\T^2$ takes the form 
\begin{equation}
\label{eq:gf:SU3}
g_f = f_1^2 g_1 + f_2^2 g_2 + f_3^2 g_3
\end{equation}
for some triple $f=(f_1,f_2,f_3)$ with $f_1f_2f_3 \neq 0$.

For any homogeneous metric $g$ on $\sunitary{3}/\T^2$, the three fibres $\Sph^2_1$, $\Sph^2_2$ and $\Sph^2_3$ are mutually orthogonal
 and the restriction of $g$ to the $j$th fibre determines a homogeneous metric on $\Sph^2$. Moreover, 
$g$ is determined by the size of these three fibres at a single point, \ie by three positive parameters.
Each of the invariant metrics $g_j$ can therefore be regarded as a homogeneous metric on $\Sph^2_j$ which
with our conventions is the standard round metric of sectional curvature $1$.
In other words, the parameter $f_j^2$ determines the ``size'' of the $j$th spherical fibre $\Sph^2_j$ with respect to the homogeneous metric $g_f$.

The $2$-dimensional space of invariant $3$-forms is spanned by
\[
\alpha = e_{246}- e_{235}-e_{145} - e_{136}, \quad \beta = e_{135}-e_{146}-e_{236}-e_{245}.
\]
There are no nontrivial invariant $1$-forms or $5$-forms; in particular the wedge product of any invariant $2$-form with any invariant $3$-form vanishes
(as one can also check directly).
We set $\vol_0 = e_{123456}$.

The exterior derivatives of these invariant forms satisfy the structure equations
\begin{subequations}
\label{eq:str:su3}
\begin{gather}
d \omega_1 = d \omega_2 = d \omega_3= \tfrac{1}{2}\alpha, \\
\quad d\alpha =0, \qquad d \beta = -2(\omega_1 \wedge \omega_2 +\omega_2 \wedge \omega_3 + \omega_3 \wedge \omega_1).
\end{gather}
\end{subequations}

\begin{remark}
\label{rmk:su3_taut}
$e_{34} - e_{56}$ is invariant under the adjoint action of not
just the diagonal $\T^2$ but all of $\unitary{2}_1$. Thus the corresponding
2-form $\omega_2 - \omega_3$ on $\sunitary{3}/\T^2$
is the pull-back of a 2-form from $\CP^2_1 := \sunitary{3}/\unitary{2}_1$.
This closed 2-form is the K\"ahler form of the Fubini-Study metric.
Note in particular that the canonical orientation on $\CP^2_1$
corresponds to $-e_{3456}$.

Hence at each point of $\sunitary{3}/\T^2$, the 2-form $\omega_2 + \omega_3$
corresponds to an anti-self-dual 2-form on the tangent space of $\CP^2_1$
at the image. We can identify $\sunitary{3}/\T^2$ with the unit sphere
in $\Lambda^2_- \CP^2_1$ in such a way that $\omega_2 + \omega_3$ is the
tautological 2-form.
\end{remark}

Now we consider the action of the Weyl group $W$ on the cone of left-invariant metrics and on the invariant $2$-forms and $3$-forms. 
Using the explicit generators of $W \cong S_3$ specified in \eqref{eq:genW}
one can verify (see also \cite[p. 214]{Cleyton:Swann}) the following:  

\begin{lemma}
\label{lem:Weyl:forms}
The Weyl group $W \cong S_3$ acts on \sunitary{3}-invariant tensors on $\sunitary{3}/\T^2$ as follows. 
\begin{enumerate}[left=0.25em]
\item $S_3$ acts as the standard representation on our chosen basis of
invariant symmetric 2-tensors 
$(g_1,g_2,g_3)$; 
\item
$S_3$ leaves the $3$-form $\beta$ invariant and acts as the sign representation on $\alpha$;
\item
A transposition $(ij) \in S_3$ acts on invariant $2$-forms by sending 
$(\omega_i,\omega_j,\omega_k) \mapsto - (\omega_j,\omega_i,\omega_k)$;  
hence the subgroup $A_3$ acts via cyclic permutations of $(\omega_1,\omega_2,\omega_3)$.
\end{enumerate}
\end{lemma}

\subsection{\texorpdfsunitary{3}-invariant \texorpdfsunitary{3}-structures}
\label{ss:closed:SU3}
We want to describe the $\sunitary{3}$-invariant $\sunitary{3}$-structures on $\sunitary{3}/\T^2$. In fact, 
since our main interest in them is as a tool to describe cohomogeneity-one closed~\gtstr s, ~\eqref{eq:closure:static} implies that 
we really want to understand $\sunitary{3}$-structures $(\omega,\Omega)$ that satisfy the additional condition that $d \Real{\Omega}=0$. 
(Invariant \sunitary{3}-structures on $\sunitary{3}/\T^2$ were not discussed explicitly in Cleyton's thesis or by Cleyton--Swann.)

Define an invariant (3,0)-form
\begin{equation}
\label{eq:Omega}
\Omega:= \alpha + i \beta.
\end{equation}
\begin{lemma}
\label{lem:SU3str:invt}\hfill
\begin{enumerate}[left=0.5em]
\item
Any $\sunitary{3}$-invariant \suthreestr~on $\sunitary{3}/\T^2$ can be written in the form
\begin{subequations}
\begin{gather}
\label{eq:SU3:invt}
\begin{cases}
(\omega_f,\Omega_{f,\theta}), \quad \textrm{or}\\
(-\omega_f,\overline{\Omega}_{f,\theta}) = \tbar (\omega_f,\Omega_{f,\theta})
\end{cases}
\intertext{where $\tbar$ is the involution defined in Remark~\ref{rmk:discrete:sym:G2} and}
\label{eq:SU3:f}
(\omega_f,\Omega_{f,\theta}):= \left(\,f_1^2 \omega_1 + f_2^2 \omega_2 + f_3^2 \omega_3, (f_1 f_2 f_3) \,e^{-i\theta}\Omega\right)
\end{gather}
\end{subequations}
for some $e^{i\theta} \in \Sph^1$ and some triple of real numbers $(f_1,f_2,f_3)$ satisfying $f_1 f_2 f_3 \neq 0$. 
\item
In the first case in~\eqref{eq:SU3:invt} the induced orientation is $\vol_0$, while the second case induces the opposite orientation.
\item
The induced invariant metric is $g_f$ as defined in~\eqref{eq:gf:SU3}.
\item
$\Omega_{f,\theta}$ as in ~\eqref{eq:SU3:f} satisfies $d \Real{\Omega_{f,\theta}}=0$ if and only if $\sin{\theta}=0$.
\end{enumerate}
\end{lemma}
\begin{proof}
Note that the two cases in~\eqref{eq:SU3:invt} are clearly exchanged by the involution $\tbar$ defined in Remark~\ref{rmk:discrete:sym:G2}. 
They are never in the same connected component of the space of $\sunitary{3}$-structures because they induce opposite orientations.

Since there are no invariant $5$-forms on $\sunitary{3}/\T^2$ the condition~\eqref{eq:omega:wedge:Omega} holds for any invariant
$2$-form and $3$-form.
Any invariant complex volume form can be written as either
$\mu e^{-i\theta}\Omega$ or $\mu e^{i\theta} \overline{\Omega}$
for some $e^{i\theta} \in \Sph^1$ and $\mu > 0$. The induced almost complex
structure is $J : E_1 \mapsto E_2$, $E_3 \mapsto E_4$, $E_5 \mapsto E_6$
in the former case and $-J$ in the latter case. An invariant 2-form $\omega$
that is positive-definite with respect to $\pm J$ must be of the form
$\pm \omega_f$ for some triple $f$. Regardless of the sign, \eqref{eq:SU3:vol} becomes equivalent
to $\mu = f_1 f_2 f_3$.

The condition on $\theta$ required for $d \Real{\Omega_{f,\theta}}=0$ 
follows immediately from the structure equations~\eqref{eq:str:su3}, more specifically that $d \alpha =d \Real{\Omega}=0$ and $d \beta = d\Imag{\Omega} \neq 0$.
\end{proof}

\begin{remark*}[The invariant nearly K\"ahler structure on $\sunitary{3}/\T^2$]
The invariant $\sunitary{3}$-structure 
\begin{equation}
\label{eq:nK:str}
\omega_{nK}:= \tfrac{1}{4} \omega = \tfrac{1}{4} (\omega_1+\omega_2+\omega_3) \quad \text{and} \quad \Omega_{nK}:=\tfrac{1}{8}\Omega
\end{equation}
satisfies
\[
d \omega_{nK} = 3 \Real{\Omega_{nK}}, \qquad d \Imag{\Omega}_{nK} = -2 \omega_{nK}^2.
\]
Hence $(\omega_{nK},\Omega_{nK})$ is the unique $\sunitary{3}$-invariant nearly K\"ahler structure on $\sunitary{3}/\T^2$. 
It corresponds to taking $f_1 = f_2 = f_3 = \tfrac{1}{2}$ and $\theta=0$ in~\eqref{eq:SU3:f}.
\end{remark*}

The next result tells us that the Klein four-group $\Z_2 \times \Z_2$ action defined in Remark~\ref{rmk:discrete:sym:G2} acts simply transitively on the connected components of the 
space of invariant $\sunitary{3}$-structures on $\sunitary{3}/\T^2$ satisfying the static constraint~$d \Real{\Omega}=0$.
\begin{corollary}[The connected components of  the space of invariant \suthreestr s~satisfying the static constraint]\hfill
\label{cor:components:SU3}
\begin{enumerate}[left=0.25em]
\item
The space of $\sunitary{3}$-invariant $\sunitary{3}$-structures on $\sunitary{3}/\T^2$ satisfying the static constraint $d \Real{\Omega}=0$ is a smooth noncompact $3$-manifold with four connected components each diffeomorphic to the positive octant in $\R^3$. The Klein four-group described in Remark~\ref{rmk:discrete:sym:G2} acts simply transitively on these connected components 
\item
The connected component containing the nearly K\"ahler \suthreestr~$(\omega_{nK},\Omega_{nK})$ defined in ~\eqref{eq:nK:str} consists of all 
\suthreestr s of the form $(\omega_f, \Omega_{f,0})$ where $(f_1,f_2,f_3)$ is any positive triple and we use the notation of~\eqref{eq:SU3:f}.
\item
The smooth map that sends an invariant $\sunitary{3}$-structure on $\sunitary{3}/\T^2$ satisfying $d \Real{\Omega}=0$ to its induced invariant metric is a smooth covering map of degree $4$.
\end{enumerate}
\end{corollary}
\begin{proof}
By Lemma~\ref{lem:SU3str:invt} we know that any invariant~\suthreestr~can be written in the form~\eqref{eq:SU3:invt}.
We already observed that the two cases in~\eqref{eq:SU3:invt} are exchanged by $\tbar$, induce different orientations and so are necessarily in different connected components
of the space of $\sunitary{3}$-structures. Hence it suffices to consider the case that $\omega_f=\omega_{f'}$ and $\Omega_{f,\theta} = \Omega_{f',\theta'}$ for some $(\theta,f)$ and $(\theta',f')$.
(The second case in~\eqref{eq:SU3:invt} can be analysed the same way).
These two equalities hold if and only if there exist $k_1,k_2,k_3 \in \Z_2$ such that
\[f_i = (-1)^{k_i} f_i' \quad \text{and} \quad \tu{sgn}\left(\frac{f_1'f_2'f_3'}{f_1f_2f_3}\right) = (-1)^{k_1+k_2+k_3}=e^{i(\theta-\theta')}.\]
Hence $\theta'=\theta \mod{2\pi}$ when $k_1+k_2+k_3$ is even and $\theta'-\theta=\pi \mod{2\pi} $
when $k_1+k_2+k_3$ is odd.
In the space of all invariant $\sunitary{3}$-structures where we are free to vary the parameter $\theta$ continuously,  
by shifting $\theta$ by $\pi$ we could change the sign of $f_1 f_2 f_3$ remaining within the 
same connected component. However, once we impose the condition that $\sin{\theta}=0$ this is no longer possible; hence 
the sign of $f_1 f_2 f_3$ is well-defined on each connected component of invariant $\sunitary{3}$-structures satisfying the static constraint~$d \Real{\Omega}=0$.
The involution $\tpi$ exchanges the two connected components that share the same orientation but on which $f_1 f_2 f_3$ 
has have different signs.
Hence there are four connected components of the space of invariant $\sunitary{3}$-structures satisfying the static constraint~$d \Real{\Omega}=0$ and they correspond to the four elements of $\Z_2 \times \Z_2$.
\end{proof}

\pagebreak[3]

The action of the Weyl group on \sunitary{3}-structures is immediate from
Lemma \ref{lem:Weyl:forms}.

\begin{lemma}[The Weyl group action]\hfill
\label{lem:Weyl:SU3}
\begin{enumerate}[left=0.25em]
\item
$S_3$ acts on invariant $\sunitary{3}$-structures: 
any transposition $\sigma \in S_3$ acts via \[(\omega_f,\Omega_{f,\theta}) \mapsto (-\omega_{\sigma(f)},- e^{-2i\theta} \overline{\Omega}_{f,\theta}) . \]
and hence
$\sigma \in A_3$ acts via $(\omega_f,\Omega_{f,\theta}) \mapsto (\omega_{\sigma(f)},\Omega_{f,\theta})$;
\item
$S_3$ preserves the set of invariant  $\sunitary{3}$-structures satisfying the static constraint $d \Real{\Omega}=0$.
\end{enumerate}
\end{lemma}
Combining the previous Lemma with our earlier work on the connected components of the space of invariant~\suthreestr s we obtain 
the following normal form for any invariant~\suthreestr~on $\sunitary{3}/\T^2$~satisfying the static constraint.
\begin{corollary}
\label{cor:SU3:normal}
Up to the actions of $\Z_2 \times \Z_2$ described in Remark~\ref{rmk:discrete:sym:G2} and the Weyl group $W\cong S_3$
any invariant \suthreestr~on $\sunitary{3}/\T^2$ satisfying the static constraint can be written uniquely in the form 
\begin{equation}
\label{eq:omf:Omf}
(\omega_f,\Omega_f)=\left(f_1^2 \omega_1+f_2^2 \omega_2 + f_3^2 \omega_3,  f_1 f_2 f_3\, \Omega\right)
\end{equation}
for some positive real triple $(f_1,f_2,f_3)$ satisfying $f_1^2 \le f_2^2 \le f_3^2$.
\end{corollary}
\begin{proof}
By Corollary~\ref{cor:components:SU3}, after acting with some element of  the group $\Z_2 \times \Z_2$ described in Remark~\ref{rmk:discrete:sym:G2} 
we can assume that our invariant
\suthreestr~satisfying the static constraint belongs to the connected component containing $(\omega_{nK},\Omega_{nk})$; 
this component is parametrised by
\[
(\omega_f,\Omega_f)=\left(f_1^2 \omega_1+f_2^2 \omega_2 + f_3^2 \omega_3,  f_1 f_2 f_3\, \Omega\right)
\]
where the real triple $(f_1,f_2,f_3)$ is positive.
By acting with $A_3 \subset W$ we can further arrange that $f_1^2 \le f_2^2 \le f_3^2$.
\end{proof}

\begin{remark}
\label{rmk:weyl:asd}
Using the Weyl group to arrange that $f_1^2 \leq f_2^2 \leq f_3^2$ makes sense
in a context where we are working just on the $\sunitary{3}/\T^2$
principal orbits. However, when we consider tensors extending smoothly
from $\R_{>0} \times \sunitary{3}/\T^2$ to $\Lambda^2_-\CP^2$ then the only
symmetry in the Weyl group that it makes sense to use is the involution that
preserves the chosen fibration over $\CP^2$.
\end{remark}

\begin{remark}
\label{rmk:weyl}
Finally we remark that $\sunitary{3}$-invariant objects on $\sunitary{3}/\T^2$ invariant under a nontrivial subgroup of $W$ enjoy extra discrete symmetries.
\begin{itemize}[left=0.25em]
\item
Any invariant metric $g_f$ for which $f_i^2=f_j^2$ for unique $i \neq j$ admits
an additional free isometric involution corresponding
to the fibrewise antipodal map of one of three $\Sph^2$-fibrations over
$\CP^2$. 
However, this $\Z_2$-action cannot preserve any $\sunitary{3}$-structure
since it is orientation-reversing. (The involution acts on
$(\omega_f,\Omega_f)$ by $\tbar$). 
\item
The invariant metrics with $f_1^2=f_2^2=f_3^2$  possess an additional free isometric action of $S_3$; these are precisely the metrics that arise from multiples of the Cartan--Killing form on $\sunitary{3}$.
The corresponding invariant $\sunitary{3}$-structures possess an additional $A_3$ symmetry, but not the full $S_3$ 
symmetry. Note that the nearly K\"ahler structure $\sunitary{3}$-structure $(\omega_{NK},\Omega_{NK})$ defined in~\eqref{eq:nK:str}
corresponds to $f_1=f_2=f_3=\tfrac{1}{2}$ and hence possesses this additional $A_3$ symmetry.
\end{itemize}
\end{remark}

\subsection{Closed invariant \texorpdfgtstr s from 1-parameter families of
invariant \texorpdfsunitary{3}-structures}

We are ready now to apply the general method described in Section~\ref{ss:G2:evolve:closed} to understand $\sunitary{3}$-invariant closed\mbox{~\gtstr s}
on $I \times \sunitary{3}/\T^2$.

\begin{prop}
\label{prop:closed:gtstr:SU3}
Up to the $\Z_2 \times \Z_2$-action defined in Remark~\ref{rmk:discrete:sym:G2}
any smooth closed $\sunitary{3}$-invariant \gtstr~ on $I \times \sunitary{3}/\T^2 = I \times \flag$ can be written uniquely in the form
\begin{subequations}
\label{eq:su3}
\begin{align}
\label{eq:phi:su3}
\varphi_f  & = \omega_f\wedge dt +\Real{\Omega_f} = (f_1^2 \omega_1 + f_2^2 \omega_2+ f_3^2 \omega_3) \wedge dt + f_1 f_2 f_3 \Real{\Omega},\\
\label{eq:*phi}
\ast \varphi_f  & =\tfrac{1}{2} \omega_f^2 - \,dt \wedge \Imag{\Omega_{f}} = f_2^2 f_3^2 \omega_2 \omega_3 + f_3^2 f_1^2 \omega_3 \omega_1 + f_1^2 f_2^2 \omega_1 \omega_2 - dt \wedge f_1f_2 f_3 \Imag{\Omega},\\
\label{eq:g_f}
g_{\varphi_f}  & = dt^2 + g_f = dt^2 + f_1^2 g_1 + f_2^2 g_2 +f_3^2 g_3,  \\ 
\vol_{\varphi_f} & = f_1^2 f_2^2 f_3^2 \vol_0 \wedge dt,\\
\intertext{where $t \in I \subset \R$ is the arclength parameter of an orthogonal geodesic,  
and $f=(f_1,f_2,f_3): I \to \R^3$ is a triple of positive smooth
real functions satisfying the ODE}
\label{eq:closed:su3}
2(f_1 f_2 f_3)' & =  f_1^2 + f_2^2 +f_3^2,
\end{align}
\end{subequations}
and where $(\omega_f,\Omega_f)$ denotes the invariant \suthreestr~defined in~\eqref{eq:omf:Omf}.
Furthermore, by using the action of the Weyl group $W\cong S_3$ we can take the triple $f$ to satisfy $f_1 \le f_2 \le f_3$.
\end{prop}
\begin{proof}
By the discussion in Section~\ref{ss:G2:evolve:closed} (and by~\eqref{eq:closed:Evolution} in particular) for any closed \sunitary{3}-invariant \gtstr~on $I \times \sunitary{3}/\T^2$ we can assume that its restriction to any constant $t$-slice 
is an invariant~\suthreestr~on $\sunitary{3}/\T^2$ satisfying the static constraint $d \Real{\Omega}=0$. 
Therefore by Corollary~\ref{cor:SU3:normal}, up to the action of the Weyl group $W$ and the Klein four-group action defined in Remark~\ref{rmk:discrete:sym:G2},
any smooth closed $\sunitary{3}$-invariant \gtstr~ on $I \times \sunitary{3}/\T^2$ can be written in the form
given in~\eqref{eq:phi:su3} where the $1$-parameter family of $\sunitary{3}$-structures 
$(\omega_f,\Omega_f)$ is defined by~\eqref{eq:omf:Omf} and must satisfy the system~\eqref{eq:closed:Evolution}. 
We have imposed the static closure condition~\eqref{eq:closure:static} throughout, so it remains only to understand the dynamic closure condition~\eqref{eq:closure:dynamic}. Using the structure equations~\eqref{eq:str:su3}~we see easily that \eqref{eq:closure:dynamic} is equivalent to~\eqref{eq:closed:su3}.
\end{proof}

Equation~\eqref{eq:closed:su3} has the following elementary but important consequences. 
\begin{lemma}[\mbox{cf. \cite[equation~(9.2)]{Cleyton:Swann}}]
Assume that $\varphi_f$ is a closed \sunitary{3}-invariant~\gtstr~on $I \times \sunitary{3}/\T^2$ in the form specified in Proposition~\ref{prop:closed:gtstr:SU3}, \ie $\varphi_f = dt \wedge \omega_f + \Real{\Omega_f}$ and $f:I \to \R^3$ is a smooth positive triple satisfying~\eqref{eq:closed:su3}. 
\label{lem:closed:su3:vol}
\begin{enumerate}[left=0em]
\item
The triple $f=(f_1,f_2,f_3)$ satisfies 
\begin{equation}
\label{eq:g:dot:lower:bound}
\frac{d}{dt}(f_1 f_2 f_3)^{1/3}\ge \frac{1}{2}
\end{equation}
with equality if and only if $f_1=f_2=f_3$.  
\item 
No $6$-dimensional orbit of $\varphi_f$ is a critical point of the orbital volume. Hence there can be no exceptional orbits and there is at most one singular orbit.
\item
If $g_\varphi$ is complete then there is a unique singular orbit; this singular orbit must be of the form $\CP^2=\sunitary{3}/\unitary{2}$ and $g_\varphi$ defines a complete Riemannian metric on $\Lambda^2_- \CP^2$ with at least Euclidean volume growth and with nonpositive scalar curvature. 
\end{enumerate}
\end{lemma}
\begin{proof}
Recall that for $p>0$ the $p$-th power mean $\mathcal{M}_p$ of a nonnegative triple $f=(f_1,f_2,f_3)$ is defined by
$\mathcal{M}_p(f): = \left( \tfrac{1}{3}  (f_1^p + f_2^p + f_3^p) \right)^{1/p}$
and $\mathcal{M}_0(f)$ is defined to be the geometric mean $(f_1 f_2 f_3)^{1/3}$. The power means inequality states that for any real numbers $r<s$ 
\begin{equation}
\label{eq:powers:mean:ineq}
\mathcal{M}_r(f) \le \mathcal{M}_s(f)
\end{equation}
with equality if and only if all the $f_i$ are equal (or $s \le 0$ and $f_i=0$ for some $i$).
\eqref{eq:closed:su3} written in terms of power means is equivalent to 
$2 (\mathcal{M}_0^3)' = 3\mathcal{M}_2^2$.
Since we are assuming that $f_1f_2f_3 \neq 0$~\eqref{eq:powers:mean:ineq} therefore implies that
\[
(\mathcal{M}_0(f) )' = \frac{1}{2} \left(\frac{\mathcal{M}_2(f)}{\mathcal{M}_0(f)}\right)^2 \ge \frac{1}{2}
\]
with equality if and only if $f_1=f_2=f_3$. Exceptional orbits are immediately ruled out because they are necessarily critical points of 
the orbital volume. 
If there were two singular orbits then since the orbital volume goes to zero for both singular orbits 
then there would also have to be an orbit of maximal volume. 
Suppose for a contradiction that the solution is complete but contains no singular orbit, then the orbit space must be $\R$. 
Then at $t=0$ say, $(f_1 f_2 f_3)^{1/3}$ is some finite positive number, 
but $(f_1 f_2 f_3)^{1/3}$ decreases  backwards in $t$ at least as fast as $\tfrac{1}{2}t$ and hence in finite backwards time it reaches zero, which contradicts our assumptions. Once we know that the orbit space is $[0,\infty)$ it follows by  Cleyton--Swann's classification of singular orbits
that the singular isotropy group is $H=\unitary{2}$, so that $G/H = \CP^2$ and that $M$ is $G$-equivariantly diffeomorphic to $\Lambda^2_- \CP^2$. 
\end{proof}

\subsubsection*{The type decomposition on invariant forms for closed \sunitary{3}-invariant~\gtstr s}
We need the following straightforward result about the~\gtwo-type decomposition when restricted to invariant $2$-forms and $3$-forms. %
\begin{lemma}
\label{lem:type:dec:invt}
Assume that $\varphi_f$ is a closed \sunitary{3}-invariant\mbox{~\gtstr} on $I \times \sunitary{3}/\T^2$ in the form specified in Proposition~\ref{prop:closed:gtstr:SU3}, \ie 
$\varphi_f=dt \wedge \omega_f + \Real{\Omega_f}$ and $f:I \to \R^3$ is a smooth positive triple satisfying~\eqref{eq:closed:su3}.
\begin{enumerate}[left=0em]
\item
The invariant $2$-forms of type $7$ are generated by $\omega_f$; the invariant $2$-form 
\[ \beta = \beta_1 \omega_1 + \beta_2 \omega_2 + \beta_3 \omega_3\] 
is of type $14$ if and only if its coefficients satisfy the constraint
\begin{equation}
\label{eq:2form:14}
\sum{\frac{\beta_i}{f_i^2}}=0.
\end{equation}
Therefore the invariant forms of type $14$ are generated by %
$f_1^2 \omega_1 - f_2^2 \omega_2$ and  $f_1^2 \omega_1 - f_3^2 \omega_3$.
\item
The invariant $3$-forms of type $7$ are generated by the invariant 3-form $\beta$ and the invariant $3$-forms of type $27$ are generated by the triple
\[\omega_i \wedge dt - \frac{f_j f_k}{4f_i} \alpha\] 
for $i=1, 2, 3$ and $(ijk)$ a permutation of $(123)$.
\end{enumerate}
\end{lemma}
\begin{proof}
Recall that the decomposition for $2$-forms takes the form
\begin{align*}
\Lambda^2_7 &= \{ X \lrcorner \varphi_f\} = \{ \beta\,|\, \beta \wedge \varphi_f = 2\!\ast\!\beta\}, \\
\Lambda^2_{14} &= \{ \beta \,|\, \beta \wedge \ast \varphi_f=0\} = \{ \beta\, | \beta\, \wedge \varphi_f = -\!\ast\!\beta \}.
\end{align*}
The first characterisation of type $7$ shows that the invariant $2$-form $\omega_f$ generates the invariant $2$-forms of type $7$. 
For an invariant $2$-form $ \beta = \beta_1 \omega_1 + \beta_2 \omega_2 + \beta_3 \omega_3$ 
the first characterisation of type $14$ shows that $\beta$ is of type $14$ if and only if
$\sum{\beta_i f_j^2 f_k^2} = 0$ which is clearly equivalent to~\eqref{eq:2form:14}.

Recall that for $3$-forms the type decomposition takes the form 
\[
\Lambda^3_1 = \langle \varphi_f \rangle, \quad \Lambda^3_7 = \{X \lrcorner \ast \varphi_f\}, \quad \Lambda^3_{27} = \{ \eta  \in \Lambda^3 |\, \varphi_f \wedge \eta =0, \ast \varphi_f \wedge \eta = 0\}.
\]
If $X$ is an invariant vector field, \ie $X= u(t) \partial_t$ 
for some function $u$, then $X \lrcorner \ast \varphi_f = (u f_1 f_2 f_3) \beta$, \ie the invariant $3$-forms of type $7$ 
are generated by the invariant 3-form $\beta$.
If $\eta$ is an arbitrary invariant $3$-form with coefficients \[
\eta = (\eta_1 \omega_1 + \eta_2 \omega_2 + \eta_3 \omega_3) \wedge dt + \eta_\alpha \alpha + \eta_\beta \beta
\]
then $\eta \wedge \varphi =0$ 
if and only if $\eta_\beta=0$ and $\eta \wedge \ast \varphi =0$ if and only if 
\begin{equation}
\label{eq:3form:27}
\eta_\alpha = -\frac{f_1 f_2 f_3}{4} \sum{\frac{\eta_i}{f_i^2}}.
\end{equation}
Therefore the invariant $3$-forms of type $27$ are generated by the triple of invariant $3$-forms claimed.
Note also that $\eta_\alpha=0$ if and only if the invariant $2$-form $\eta_1 \omega_1 + \eta_2 \omega_2 + \eta_3 \omega_3$ is of type $14$.
\end{proof}

\subsection{The intrinsic torsion of \texorpdfsunitary{3}-invariant
closed \texorpdfgtstr s}

Recall that the intrinsic torsion of any closed~\gtstr~$\varphi$ is the unique $2$-form $\tau$ of type $14$ that satisfies~\eqref{eq:tau2:closed}, 
\ie
$d(\ast \varphi) = \tau \wedge \varphi.$
If $\varphi$~is also \sunitary{3}-invariant then so is its torsion $\tau$ and so by the results of the previous section 
we can write the torsion $2$-form as 
\[
\tau = \tau_1\, \omega_1 + \tau_2\, \omega_2 + \tau_3\, \omega_3
\] 
for a triple of functions $(\tau_1,\tau_2,\tau_3)$ satisfying the type $14$
constraint~\eqref{eq:2form:14}, \ie
\begin{equation}
\label{eq:tau_14}
\frac{\tau_1}{f_1^2} + \frac{\tau_2}{f_2^2} + \frac{\tau_3}{f_3^2} = 0 .
\end{equation}
The following lemma determines these torsion coefficients $\tau_i$ for $\varphi_f$ in terms of the triple $f=(f_1,f_2,f_3)$ and its first derivatives.
\begin{lemma}
\label{lem:torsion:SU3}
Let $\varphi_f$ be a $\sunitary{3}$-invariant closed~\gtstr~on $I \times \sunitary{3}/\T^2$ written in the normal form described in Proposition~\ref{prop:closed:gtstr:SU3}. The intrinsic $2$-form $\tau=\tau_1\, \omega_1 + \tau_2\, \omega_2 + \tau_3\, \omega_3$ of $\varphi_f$
is the unique invariant $2$-form of type $14$
whose coefficients are the triple of real-valued functions $(\tau_1, \tau_2, \tau_3): I \to \R^3$ satisfying 
\begin{subequations}
\begin{gather}
\label{eq:tau:i:SU3}
\tau_i = - \frac{f_i^2}{f_j^2 f_k^2} \left(\tau_j f_k^2 + \tau_k f_j^2 \right) = - \frac{f_i^2}{f_j^2 f_k^2} \left( (f_j^2f_k^2)' - 2f_1 f_2 f_3\right)\!,\\
\intertext{for $(ijk)$ any permutation of $(123)$ or equivalently}
\label{eq:tau:i:SU3a}
\tau_i = (f_i^2)' + \frac{f_i^2}{f_1 f_2 f_3}\left( 2f_i^2 - \fb\right),
\end{gather}
\end{subequations}
where for a more compact notation we define
\begin{equation}
\label{eq:fb}
\fb:=f_1^2+f_2^2+f_3^2.
\end{equation}
\end{lemma}
\begin{proof}
It follows from the structure equations \eqref{eq:str:su3} that
\begin{equation}
\label{eq:d*phi:su3}
d (\ast \varphi_f) = \sum_{(ij)}{ \left( (f_i^2 f_j^2)' -2 f_1 f_2 f_3)  \right)\, \omega_i \omega_j \wedge dt}, 
\end{equation}
where $(ij)$ is summed over $(12)$, $(23)$ and $(31)$.
For $(ijk)$ any permutation of $(123)$ we define 
\begin{equation}
\label{eq:Ai:def}
A_i:= (f_j^2f_k^2)' - 2f_1 f_2 f_3.
\end{equation}
Then using~\eqref{eq:phi:su3} and~\eqref{eq:d*phi:su3} it is straightforward to verify that the
condition~\eqref{eq:tau2:closed} is equivalent to 
\begin{equation}
\label{eq:Ai:torsion}
A_i = \tau_k f_j^2 + \tau_j f_k^2, \quad \text{for\ } (ijk)\ \text{any permutation of\ }(123).
\end{equation}
Multiplying this by $f_i^2$ and using~\eqref{eq:2form:14} and~\eqref{eq:Ai:def} we obtain~\eqref{eq:tau:i:SU3}.
To see the equivalence of~\eqref{eq:tau:i:SU3} and~\eqref{eq:tau:i:SU3a} note that from the equality \[
(f_i^2)' f_j^2 f_k^2 = (f_1^2 f_2^2 f_3^2)' - f_i^2 (f_j^2 f_k^2),'\]
using the closure of $\varphi_f$, \ie~\eqref{eq:closed:su3}, 
to rewrite the first term of the right-hand side, the definition of $A_i$ and \eqref{eq:tau:i:SU3}~we obtain 
\[
(f_i^2)' (f_j^2 f_k^2) = f_1 f_2 f_3 \fb - f_i^2 (A_i + 2 f_1 f_2 f_3) = f_1 f_2 f_3 (\fb - 2f_i^2) + f_j^2 f_k^2 \tau_i.
\]
Rearranging this gives~\eqref{eq:tau:i:SU3a}. 
\end{proof}

\begin{remark}
\label{rmk:involution_asd}
A consequence in this context of Remark \ref{rmk:weyl} is that
\sunitary{3}-invariant \gtstr s on
${\R_{>0} \times \sunitary{3}/\T^2}$ where two of the variables are equal, say $f_1 = f_2$, are acted on as $-1$ by a certain orientation-reversing involution
commuting with the \sunitary{3} action (thus the metric
is $\sunitary{3} \times \Z_2$-invariant).
Thinking of $\R_{>0} \times \sunitary{3}/\T^2$ as the complement to the zero
section in $\Lambda^2_- \CP^2$, this is the involution defined by multiplying
the fibres by~$-1$.
\end{remark}

\begin{example}
\label{ex:BS}
By \eqref{eq:closed:su3} and \eqref{eq:tau:i:SU3}, the
$\sunitary{3}$-invariant \gtstr{} defined by
$f_1, f_2, f_3$ is torsion-free if and only if
$2(f_1f_2f_3)' = f_1^2 + f_2^2 + f_3^2$ and each $(f_j^2f_k^2)' = 2f_1f_2f_3$.
The simplest solution to these equations is 
\[
f_1 = f_2 = f_3 = \tfrac{1}{2}t.
\]
This incomplete solution describes a conical \sunitary{3}-invariant torsion-free~\gtwo-cone 
with cross-section $\flag=\sunitary{3}/\T^2$ arising from the unique \sunitary{3}-invariant nearly
K\"ahler structure on $\flag$ that we described in~\eqref{eq:nK:str}.
To instead obtain a complete metric, the solution must satisfy the conditions
for closing smoothly over a nonprincipal orbit (described in detail in section
\ref{ss:IC:CP2}, see also Cleyton--Swann \cite[p.~217]{Cleyton:Swann}).  
The smoothly-closing conditions force two of the $f_i$ to coincide
and (up to the Weyl group action from Lemma~\ref{lem:Weyl:SU3}) any
complete solution can be written in terms of the
parameter $r = f_1f_3$ and a constant $\mu \in \R$ as
\[
f_1 = r(r^2 + \mu^2)^{-\frac{1}{4}},
\quad f_2 = f_3 = (r^2 + \mu^2)^{\frac{1}{4}}.
\]
These solutions define the asymptotically conical holonomy $G_2$ metrics
on $\Lambda^2_- \CP^2$ first found by Bryant and Salamon \cite{Bryant:Salamon}; they are
asymptotic to the torsion-free~\gtwo-cone just described. Permuting the
$f_i$ gives three AC \gtmfd s asymptotic to the same cone, but topologically different in that the special $\CP^2$ orbit fits in differently. 
In fact, one can show that any complete torsion-free~\gtstr~asymptotic to the  torsion-free~\gtwo-cone 
described above must be one of these Bryant--Salamon solutions~\cite[Corollary 6.10]{karigiannis:lotay}.

Another way of expressing the claim that the three solutions extend
across topologically different $\CP^2$s is that while the full Weyl
group $W$ acts isometrically on the cone, only the involution that
preserves the particular fibration over $\CP^2$ acts isometrically on
the AC $G_2$-metric, and the other non-trivial elements of the
Weyl group do not even extend to homeomorphisms of $\Lambda^2_- \CP^2$
(see Atiyah--Witten \cite[\S 2.3]{atiyah03}).
This gives rise to a $G_2$ geometric transition, somewhat analogous to
the well-known ``flop'' between the two topologically distinct
small resolutions of the complex 3-dimensional ordinary double point
singularity.
\end{example}

\subsection{Closed \texorpdfSp{2}-invariant~\texorpdfgtstr s}
\label{ss:Sp2:symmetry}
When $G=\Sp{2}$, the principal isotropy group $K = \unitary{1} \times \Sp{1}$ acts on the standard representation $\mathbb{H}^2 \cong \C^4$ as
$H+L+\overline{L}$ where $H \cong \mathbb{H}$ is the standard representation of $\Sp{1}$ and $L \cong \C$ is the standard representation of $\unitary{1}$.
It follows that the isotropy representation is $\Rrep{L^2} + \Rrep{H \overline{L}}$. Both of these irreducible modules admit an invariant metric $g_i$ and invariant 
$2$-form $\omega_i$ and the space of invariant $3$-forms on their sum is $2$-dimensional. 
To be more concrete we equip the isotropy representation with the basis
\begin{align*}
E_1 &= \frac{1}{2}\begin{pmatrix}
0 & 0  \\
0 & j
\end{pmatrix}, 
\qquad 
E_2 = \frac{1}{2}\begin{pmatrix}
0 & 0\\
0 & -k
\end{pmatrix}, 
\qquad 
E_3 = \frac{1}{2\sqrt{2}}\begin{pmatrix}
0 & -1\\
1 & 0
\end{pmatrix}, \\
E_4 &= \frac{1}{2\sqrt{2}}\begin{pmatrix}
0 & i\\
i & 0
\end{pmatrix}, 
\quad 
E_5 = \frac{1}{2\sqrt{2}}\begin{pmatrix}
0 & j\\
j & 0
\end{pmatrix}, 
\quad 
E_6 = \frac{1}{2\sqrt{2}}\begin{pmatrix}
0 & k\\
k & 0
\end{pmatrix},
\end{align*}
and denote the corresponding dual basis by $\{e_1, \ldots ,e_6\}$  Then $\{e_1,e_2\}$ is a basis for $\Rrep{L^2}^*$ and $\{e_3, \ldots ,e_6\}$ is a basis for $\Rrep{H \overline{L}}^*$.
We can scale $g_i$ and $\omega_i$ so that 
\begin{subequations}
\begin{gather*}
g_1 = e_1^2+e_2^2, \quad  g_2 = e_3^2+e_4^2+e_5^2 + e_6^2, \\
\omega_1  = e_{12}, \qquad  \omega_2  = e_{34}+e_{56}.
\end{gather*}
\end{subequations}
$g_1$ may be identified with an $\Sp{1}$-invariant metric on the $2$-sphere $\Sp{1} \times \Sp{1}/\Sp{1}\times \unitary{1}$. 
By computing its sectional curvature one can check that $g_1$ is the standard round metric with constant sectional curvature $1$. 
Similarly $g_2$ can be identified with an $\Sp{2}$-invariant metric on $\Sp{2}/\Sp{1} \times \Sp{1} \cong \mathbb{HP}^1 \cong \Sph^4$.
Again by computing its sectional curvature one can check that $g_2$ is the standard round metric with constant sectional curvature $\tfrac{1}{2}$.
The $2$-dimensional space of invariant $3$-forms is then spanned by
\[
\alpha = e_{246}- e_{235}-e_{145} - e_{136}, \quad \beta = e_{135}-e_{146}-e_{236}-e_{245},
\]
and we set $\vol_0 = e_{123456}$. 
The exterior derivatives of these invariant forms satisfy
\begin{equation}
\label{eq:str:eqns:sp2}
d \omega_1 = \tfrac{1}{2}\alpha, \quad d\omega_2 = \alpha, \quad d\alpha =0, \quad d \beta = -2\omega_1 \wedge \omega_2 - \omega_2 \wedge \omega_2.
\end{equation}

\begin{remark}
\label{rmk:sp2_taut}
$-\omega^2_2$ is the pull-back of a nowhere-vanishing form on $\Sph^4$.
With respect to that orientation on $\Sph^4$, the value of $\omega_2$ at
a point $x \in \CP^3$ is an anti-self-dual 2-form on the tangent space of the
image of $x$ in $\Sph^4$. We can therefore identify $\CP^3$ with the
unit sphere bundle in $\Lambda^2_- \Sph^4$ so that $\omega_2$ is the
tautological 2-form.

Meanwhile, note that $2\omega_1 - \omega_2$ is closed. Up to scale, this is the Fubini-Study 2-form on $\CP^3$; however, this is not the 2-form component of
any \sunitary{3}-structure, since the standard complex structure on $\CP^3$
has non-zero first Chern class.
\end{remark}

Using the same method as in the $\sunitary{3}$-invariant case we obtain the following description of
$\Sp{2}$-invariant $\sunitary{3}$-structures on $\CP^3$ that satisfy the static closure condition $d \Real{\Omega}=0$.

\pagebreak[2]
\begin{lemma}\hfill
\label{lem:Sp2:invt:SU3:strs}
\begin{enumerate}[left=0em]
\item
Up to the action of $\Z_2 \times \Z_2$ described in Remark~\ref{rmk:discrete:sym:G2} a general $\Sp{2}$-invariant 
$\sunitary{3}$-structure on $\Sp{2}/\Sp{1} \times \unitary{1} = \CP^3$ can be written in the form
\[
(\omega_f,\Omega_{f,\theta}) = (f_1^2 \omega_1 + f_2^2 \omega_2,f_1f_2^2 e^{-i\theta} \Omega)
\]
for $e^{i \theta} \in \Sph^1$, $\Omega:= \alpha + i \beta$ and $f=(f_1,f_2) \in \R^2_> \subset \R^2$. The induced metric is $g_f = f_1^2 g_1 + f_2^2 g_2$.
\item
$(\omega_f,\Omega_{f,\theta})$ satisfies the static closure condition  $d \Real{\Omega_{f,\theta}}=0$ if and only if $\sin{\theta}=0$.
\item
The $\Sp{2}$-invariant $\sunitary{3}$-structure with $f_1=f_2=\tfrac{1}{2}$ and $\theta=0$ is the standard $\Sp{2}$-invariant
nearly K\"ahler structure on $\CP^3$. 
\item
The space of $\Sp{2}$-invariant $\sunitary{3}$-structures on $\CP^3$ satisfying the static closure condition  has four connected components each diffeomorphic to 
the positive quadrant $\R^2_>$ in $\R^2$ and $\Z_2 \times \Z_2$ acts simply transitively on these four components.
The connected component containing the nearly K\"ahler structure on $\CP^3$ is parametrised by 
\[
(\omega_f,\Omega_{f}) = (f_1^2 \omega_1 + f_2^2 \omega_2,f_1f_2^2 \Omega)
\]
with $f=(f_1,f_2) \in \R^2_>$.
\end{enumerate}

\end{lemma}
Similarly we have the following result about the `Weyl group' $W$ and its action on invariant~\suthreestr s in the $\Sp{2}$-invariant setting. 
\begin{lemma}Let $K$ denote the principal isotropy group $\Sp{1}\times \unitary{1}$.
\label{lem:Sp2:Weyl}
\begin{enumerate}[left=0em]
\item
The normaliser of $K$ in $\Sp{2}$, $N_{\Sp{2}}(K) \cong \Sp{1} \times N_{\Sp{1}} (\unitary{1})$ and therefore the quotient $W:=N(K)/K$ is isomorphic to $N_{\Sp{1}} \unitary{1} \cong \Z_2$.
\item
The generator of $W$ fixes the invariant bilinear forms $g_1$, $g_2$, and the $3$-form $\beta$ and acts 
as $(-1)$ on $\omega_1$, $\omega_2$ and $\alpha$. In particular any invariant metric $g_f$ on $\CP^3$ possesses an additional 
free isometric $\Z_2$-symmetry, but this $\Z_2$ does not preserve any invariant $\sunitary{3}$-structure.
\item
$W$ acts on invariant $\sunitary{3}$-structures and preserves the subset of structures satisfying the static constraint $d \Real{\Omega}=0$.
\end{enumerate}
\end{lemma}
\begin{proof}
Parts (i) and (ii): The normaliser and the Weyl group $W$ are described
explicitly in \mbox{\cite[p. 216]{Cleyton:Swann}} (but note that there is
a sign error there in describing its action on invariant $3$-forms); see
also \mbox{\cite[pp. 120--121]{Cleyton:thesis}}, but notice that
Cleyton's choice of $\alpha$ and $\beta$ is the opposite to ours and to
that in~\cite{Cleyton:Swann}, namely his form $\beta$ is closed.
Part (iii) follows easily from part (ii) and the previous lemma.
\end{proof}

By appealing to the method we used in the $\sunitary{3}$-invariant case but using 
Lemmas~\ref{lem:Sp2:invt:SU3:strs} and~\ref{lem:Sp2:Weyl} we deduce the following result. 
\begin{prop}
\label{prop:closed:gtstr:Sp2}
Up to the action of the discrete symmetries $\Z_2 \times \Z_2$ of Remark~\ref{rmk:discrete:sym:G2} any smooth closed $\Sp{2}$-invariant \gtstr~ on $I \times \Sp{2}/\Sp{1} \times \unitary{1} = I \times \CP^3$ can be written in the form
\begin{subequations}
\label{eq:sp2}
\begin{align}
\varphi & =\omega_f \wedge dt +  \Real{\Omega_f} = (f_1^2 \omega_1 + f_2^2 \omega_2)\wedge dt + f_1f_2^2\, \alpha\\
\ast\varphi & =  \tfrac{1}{2} \omega_f^2 + \Imag{\Omega_f} \wedge dt =  f_1^2f_2^2\, \omega_1\wedge \omega_2 + \tfrac{1}{2} f_2^4\, \omega_2 \wedge \omega_2 + f_1f_2^2\, \beta  \wedge dt,\\
g_\varphi &= dt^2 + g_f = dt^2 + f_1^2 g_1 + f_2^2 g_2,\\
\vol{g_\varphi} &= f_1^2f_2^4\, dt \wedge \vol_0,
\intertext{where $t \in I \subset \R$ is the arclength parameter of an orthogonal geodesic and  $f=(f_1,f_2): I \to \R^2$ is a pair of positive smooth real functions satisfying}
\label{eq:closed:sp2}
2 (f_1f_2^2)' & =  f_1^2 + 2f_2^2.
\end{align}
\end{subequations}
\end{prop}
\begin{remark}
\label{rmk:involution}
An important observation (due to Cleyton--Swann) is that if we set $f_2=f_3$ 
in Proposition~\ref{prop:closed:gtstr:SU3} then all the statements there reduce to those in Proposition~\ref{prop:closed:gtstr:Sp2} above, and
we can treat $\Sp{2}$-invariant \gtstr s as if they were a special case
of $\sunitary{3}$-invariant ones (namely the ones with an extra $\Z_2$ symmetry as in Remark \ref{rmk:involution_asd}).
The one difference to note is that because in the $\Sp{2}$-invariant setting the Weyl group $W \cong \Z_2$ preserves all invariant bilinear forms 
(whereas in the $\sunitary{3}$ setting $W\cong S_3$ permutes $g_1$, $g_2$ and $g_3$) we cannot a priori assume 
that $f_1 \le f_2$.

Another point of view is that for any self-dual positive Einstein
4-manifold $X$ and $\Sigma$ the unit sphere
bundle in $\Lambda^2_- X$, the structure equations for
a suitably normalised vertical form $\omega_1 \in \Omega^2(\Sigma)$
and the tautological 2-form $\omega_2 \in \Omega^2(\Sigma)$ coincide
with \eqref{eq:str:eqns:sp2}
(\cf Remarks \ref{rmk:su3_taut} and \ref{rmk:sp2_taut}).
\end{remark}

As in the $\sunitary{3}$-invariant setting the closed condition~\eqref{eq:closed:sp2} already implies significant restrictions on the possible orbit structure of a closed \Sp{2}-invariant~\gtstr. 
\begin{lemma}Assume that $\varphi_f$ is a closed $\Sp{2}$-invariant~\gtstr~on $I \times \CP^3$ in the form specified in Proposition \ref{prop:closed:gtstr:Sp2}, \ie
$\varphi_f = dt \wedge \omega_f + \Real{\Omega}_f$ and $f=(f_1,f_2):I \to \R^2$ is a smooth positive pair satisfying~\eqref{eq:closed:sp2}.
\label{lem:closed:Sp2:vol}
\begin{enumerate}[left=0em]
\item
$f=(f_1,f_2)$ satisfies 
$\tfrac{d}{dt}(f_1 f_2^2)^{1/3}\ge \frac{1}{2}$
with equality if and only if $f_1=f_2$.  
\item 
No $6$-dimensional orbit is a critical point of the orbital volume. Hence there are no exceptional orbits and there is at most one singular orbit.
\item
If $g_\varphi$ is complete then there is a unique singular orbit; this singular orbit must be of the form 
$\Sph^4=\Sp{2}/\Sp{1} \times \Sp{1}$ and $g_\varphi$ defines a complete Riemannian metric on $\Lambda^2_- \Sph^4$ with at least Euclidean volume growth and with nonpositive scalar curvature. 
\end{enumerate}
\end{lemma}

The \gtwo-type decomposition for $\Sp{2}$-invariant forms on $\varphi_f$ takes the following form.
\begin{lemma}
Assume that $\varphi_f$ is a closed $\Sp{2}$-invariant~\gtstr~on $I \times \CP^3$ in the form specified in Proposition \ref{prop:closed:gtstr:Sp2}.
\begin{enumerate}[left=0em]
\item
The invariant $2$-forms of type $7$ are generated by $\omega_f$; the invariant $2$-form $\beta = \beta_1 \omega_1 + \beta_2 \omega_2$ is of type $14$ if and only if its coefficients satisfy the constraint
\[
\frac{1}{f_1^2}\beta_1 + \frac{2}{f_2^2}\beta_2=0.
\]
Therefore the invariant $2$-forms of type $14$ are generated by $2f_1^2 \omega_1 - f^2_2 \omega_2$.
\item
The invariant $3$-forms of type $7$ are generated by the invariant $3$-form $\beta$. 

\item
The torsion $2$-form $\tau = \tau_1 \omega_1 + \tau_2 \omega_2$ of $\varphi_f$ is given in terms of $f_1$ and $f_2$ by
\begin{equation}
\label{eq:taui:Sp2}
\tau_1 = (f_1^2)' - 2f_1 + \frac{f_1^3}{f_2^2}, \quad \tau_2 = (f_2^2)' - f_1,
\end{equation}
and satisfies the type $14$ constraint
\begin{equation}
\label{eq:tau1:tau2:sp2}
\frac{1}{f_1^2}\tau_1  +\frac{2}{f_2^2} \tau_2=0.
\end{equation}
In particular the torsion coefficients $\tau_1$ and $\tau_2$ have opposite signs whenever they are nonzero.
\end{enumerate}
\end{lemma}

\begin{remark}
\label{rmk:sp2torfree}
Since both types of torsion-free $\sunitary{3}$-invariant solution given in
Example \ref{ex:BS} had $f_2 = f_3$, they also give rise to solutions to the
torsion-free \Sp{2}-invariant ODE system. They correspond to the
unique \Sp{2}-invariant torsion-free cone over $\CP^3$
arising from the  \Sp{2}-invariant nearly K\"ahler structure on $\CP^3$,
and to the complete \Sp{2}-invariant solutions on $\Lambda^2_{-}\Sph^4$
asymptotic to that cone respectively.
In fact, Bryant and Salamon~\cite{Bryant:Salamon} defined torsion-free
AC \gtstr s like this on the anti-self-dual bundle of any closed self-dual positive Einstein
4-manifold $X$, \cf Remark \ref{rmk:involution}.
\end{remark}

\section{The ODE systems for cohomogeneity-one Laplacian solitons}
\label{s:ODE:cohom1:LS}
In this section we derive the ODE systems satisfied by $G$-invariant Laplacian 
solitons for the cohomogeneity-one actions of $G=\sunitary{3}$ and $G=\Sp{2}$ described 
in the previous section. 

Section~\ref{ss:rmks:closed:LS} makes some general remarks 
about closed Laplacian solitons that we use later. In Section~\ref{ss:SU3:soliton:ODEs} 
we derive the $G$-invariant Laplacian soliton equations. 
Section~\ref{ss:order1:su3} gives a reformulation of these equations as a real-analytic 
first-order system;  for most purposes the first-order reformulations turn out to be more convenient 
than the form of the systems derived in Section~\ref{ss:SU3:soliton:ODEs}. 
In particular, Theorem~\ref{mthm:local:solitons} follows immediately: see Corollaries~\ref{cor:su3:local} and~\ref{cor:sp2:local}.

Recall from \eqref{eq:lap:soliton:v2}~that the system \eqref{eq:LS} for a closed Laplacian soliton 
may be recast in the form
\[
d\varphi =0, \quad d(\tau - X\lrcorner\varphi) = \lambda \varphi,
\]
where $\tau$ is the intrinsic torsion of $\varphi$, \ie the unique $2$-form of type $14$ satisfying 
\[
d(\ast \varphi) = \tau \wedge \varphi.
\]
For the cohomogeneity-one actions of the groups \Sp{2} and \sunitary{3} considered in the previous section 
we now seek cohomogeneity-one Laplacian solitons, \ie the $3$-form $\varphi$ has the form specified in Proposition~\ref{prop:closed:gtstr:SU3} or
Proposition~\ref{prop:closed:gtstr:Sp2} respectively
and the vector field $X=u\, \partial_t$ for some function $u=u(t):I \to \R$.
In both cases the system of PDEs~\eqref{eq:lap:soliton:v2} 
reduces to a system of nonlinear ODEs that we will write down explicitly. Moreover, by Remark \ref{rmk:involution} 
the ODE system for \Sp{2}-invariant solitons can be obtained from the system for \sunitary{3}-invariant solitons 
by setting $f_2=f_3$, $\tau_2=\tau_3$ and considering the resulting ODE system in a smaller number of variables. 

\subsection{General remarks on closed Laplacian solitons}
\label{ss:rmks:closed:LS}
It is useful to decompose by type the $3$-forms appearing on both sides of the Laplacian soliton equation~\eqref{eq:LS}
and to consider the resulting equations on the components of type $1$, $7$ and $27$ separately. Recall from~\eqref{eq:d:tau2} that $d\tau$ has no type $7$ component and its type $1$ component is $\tfrac{1}{7} \abs{\tau}^2 \varphi$. 
Lotay--Wei \cite[\S9]{lotay:wei:shi} show that the type decomposition of $\mathcal{L}_X \varphi$ for any closed \gtstr~$\varphi$ has
components of type $1$ and $7$ %
respectively given by
\[
\tfrac{3}{7} \tu{div}(X)\varphi, \quad
\tfrac{1}{2} \left(d^*(\iota_X\varphi) \right)^\sharp \lrcorner \ast \varphi,
\]
while the type 27 component corresponds
to
\[ \tfrac{1}{2} \mathcal{L}_X g_\varphi - \tfrac{1}{7} \tu{div}(X) g_\varphi \]
under an isomorphism between $\Lambda^3_{27}$ and the trace-free symmetric
bilinear forms.
Hence the $1$ and $7$ components of the Laplacian soliton equation~\eqref{eq:LS} read
\begin{subequations}
\begin{gather}
\tag{LS$_1$}
\label{eq:LS:1}
\tu{div}\,X = \tfrac{1}{3} \abs{\tau}^2 - \tfrac{7}{3}\lambda, \\
\tag{LS$_7$}\label{eq:LS:7}
d^*(X \lrcorner \varphi)=0
\end{gather}
\end{subequations}
respectively. 
The 27 component of~\eqref{eq:LS} can be derived from the expression given in \cite[(9.12)]{lotay:wei:shi}.

\begin{remark}
Integrating both sides of~\eqref{eq:LS:1} over a compact manifold and applying the Divergence Theorem implies that compact shrinking Laplacian solitons cannot exist and that for compact steady Laplacian solitons the underlying~\gtstr~must be torsion free and therefore the vector field~$X$ must be an automorphism of $\varphi$, \ie $\mathcal{L}_X\varphi=0$. Hence also $X$ must also be a Killing field of $g_\varphi$: a torsion-free~\gtstr~admits nontrivial Killing fields only when the holonomy of $g_\varphi$ is reducible. 
If we have a gradient Laplacian soliton, \ie if $X$ is the gradient of a potential $f$, and~$\lambda \le 0$, 
then we have  $\Delta f \ge \frac{1}{3} \abs{\tau}^2 \ge 0$, \ie the potential function $f$ is subharmonic on any steady or shrinking gradient 
Laplacian soliton. This fact can be useful in the complete noncompact setting.
\end{remark}

In the \sunitary{3}-invariant setting note, that for any invariant
closed $3$-form $\varphi_f$ as in \eqref{eq:phi:su3} and invariant
vector field $X=u \, \partial_t$, we have
 $\ast(X \lrcorner \varphi_f) = u\,\left(\sum{f_j^2 f_k^2\, \omega_j \wedge \omega_k \wedge dt}\right)$.
Since $d\omega_i = \tfrac{1}{2}\alpha$ and $\omega_i \wedge \alpha =0$, this implies that $d\!\ast\!(X \lrcorner \varphi) =0$. Hence the $7$ component of the Laplacian soliton equation is automatically satisfied in this case.
The 1 component of the Laplacian soliton equation however yields useful information. Since $X$ is assumed to be invariant
we have
\[
\tu{div}X = \tu{div}(u\,\partial_t) = \frac{1}{(f_1f_2f_3)^2} (u f_1^2f_2^2 f_3^2)' = u' + 2u \,(\ln{f_1 f_2 f_3})' = u' + u \left(\frac{\fb}{ f_1f_2f_3}\right)
\]
where in the final equality we use~\eqref{eq:closed:su3}.
Hence in this setting~\eqref{eq:LS:1} reads
\begin{subequations}
\label{eq:LS:1:su3+}
\begin{equation}
\label{eq:LS:1:su3}
 u' + u \left(\frac{\fb}{ f_1f_2f_3}\right) = \frac{1}{3} \abs{\tau}^2 - \frac{7\lambda}{3},
\end{equation}
where by \eqref{eq:g_f} %
\begin{equation}
\label{eq:tau_square}
\abs{\tau}^2 = \sum \frac{\tau_i^2}{f_i^4} .
\end{equation}
\end{subequations}

As a immediate consequence of~\eqref{eq:LS:1} we deduce the following result about  \sunitary{3}-invariant shrinkers.
\begin{lemma}
\label{lem:u:shrinker:+}
Let $(\varphi_f,X=u \,\partial_t,\lambda)$ be any \sunitary{3}-invariant shrinking Laplacian soliton.  If the vector field $X=u\, \partial_t$ is positive at some $t_+$ then $X$ remains positive for all $t \ge t_+$  (within the lifetime of the soliton). Moreover 
if soliton is  forward-complete, \ie the solution exists for all $t$ sufficiently large, then $X$ is eventually positive.\end{lemma}
\begin{proof}
\eqref{eq:LS:1:su3} implies that for any shrinker $u' \ge - \frac{7\lambda}{3}>0$ whenever $u \le 0$. 
\end{proof}
\begin{remark*}
In the same spirit, but with a little more work to understand the torsion term, one can also make deductions about the positivity of $u$ being preserved for steady solitons. 
\end{remark*}

\subsection{The ODE system for \texorpdfsunitary{3}-invariant closed Laplacian solitons on \texorpdfstring{$M^0$}{M\textasciicircum 0}}
\label{ss:SU3:soliton:ODEs}
In this section we derive the ODEs satisfied by  \sunitary{3}-invariant closed Laplacian solitons on $M^0$, the 
open dense set of principal points for the \sunitary{3}-action, and derive some of the basic properties of this ODE system.
The ODEs satisfied by \Sp{2}-invariant Laplacian solitons arise by specialisation of this system.

\enlargethispage{0.4\baselineskip}

\begin{lemma}
The triple $(\varphi_f,X,\lambda)$ is an \sunitary{3}-invariant closed Laplacian soliton on $I \times \flag$ if and only if up to the action 
of the Klein four-group defined in~\eqref{rmk:discrete:sym:G2} it can be written in the form
\[
\varphi_f = \omega_f \wedge dt + \Real{\Omega_f}, \qquad X=u(t) \partial_t
\]
where $(\omega_f,\Omega_f)$ is the invariant \suthreestr~defined in~\eqref{eq:omf:Omf}, 
$u$ is a real function on the interval $I$ and  $(f_1,f_2,f_3):I \to \R^3$ is a positive triple satisfying the equations
\begin{subequations}
\label{eq:ODEs:SU3}
\begin{align}
\label{eq:SU3:closure}
2(f_1 f_2 f_3)' &= f_1^2 + f_2^2 + f_3^2,\\
\label{eq:ODEs:SU3:taui}
(\tau_i - u f_i^2)' &= \lambda f_i^2, \quad \tu{for\ } i=1,2,3,\\
\label{eq:ODEs:SU3:tausum}
\tau_1 + \tau_2 + \tau_3 &= u(f_1^2+f_2^2+f_3^2) +2\lambda f_1 f_2 f_3,
\end{align}
\end{subequations}
where  $\tau_i$ are the components of the torsion $2$-form $\tau = \sum{\tau_i \omega_i}$ of $\varphi_f$ as determined by 
\eqref{eq:tau:i:SU3}.
\end{lemma}
\begin{proof}
Proposition~\ref{prop:closed:gtstr:SU3} already showed that up to the action of the Klein four-group any closed \sunitary{3}-invariant \gtstr~ on $I \times \flag$ 
can be written in the form above and must satisfy~\eqref{eq:SU3:closure}. So it remains to establish~\eqref{eq:ODEs:SU3:taui} 
and~\eqref{eq:ODEs:SU3:tausum}. Since $d \omega_i = \frac{1}{2}\alpha$ we calculate that
\[
d\tau = \frac{1}{2} \left(\sum{\tau_i}\right) \alpha + \sum{ \tau_i' dt \wedge \omega_i}
\]
and 
\[
d(X\lrcorner \varphi_f) = \left(\tfrac{1}{2} u \sum{f_i^2}\right) \alpha + \sum{ (uf_i^2)' dt \wedge \omega_i}.
\]
\eqref{eq:ODEs:SU3:taui} and~\eqref{eq:ODEs:SU3:tausum} now follow by equating the coefficients of $\omega_i \wedge dt$ and of $\alpha$ respectively
on both sides of the second equation in~\eqref{eq:LS}.
\end{proof}

\subsubsection*{The ODE system for \Sp{2}-invariant closed Laplacian solitons}
Recall from~\eqref{eq:sp2} that up to the action of discrete symmetries any  closed \Sp{2}-invariant \gtstr~ on $I \times \CP^3$ can be written in the form
\[
\varphi_f = \omega_f \wedge dt + \Real{\Omega_f}  =  (f_1^2 \omega_1 + f_2^2 \omega_2)\wedge dt + f_1f_2^2 \alpha,
\]
where $(f_1,f_2): I \to \R^2$ is a positive pair satisfying the first-order ODE~\eqref{eq:closed:sp2}.
The Laplacian soliton system for such a closed \Sp{2}-invariant \gtstr~reduces to
\begin{subequations}
\label{eq:ODEs:Sp2}
\begin{align}
\label{eq:ODE:closure}
(f_1f_2^2)' &= \tfrac{1}{2}f_1^2 + f_2^2,\\
\label{eq:ODE:tau1}
(\tau_1 - uf_1^2)' &= \lambda f_1^2,\\
\label{eq:ODE:tau2}
(\tau_2 - uf_2^2)' & = \lambda f_2^2,\\
\label{eq:ODE:conserve}
\tau_1+2 \tau_2 &= u(f_1^2+2f_2^2) + 2\lambda f_1f_2^2,
\end{align}
\end{subequations}
where the components $\tau_1$ and $\tau_2$ of the torsion $2$-form $\tau = \tau_1 \omega_1 + \tau_2 \omega_2$ are 
given by~\eqref{eq:taui:Sp2}.
This ODE system can be obtained from~\eqref{eq:ODEs:SU3} by 
setting $f_2=f_3$ and $\tau_2=\tau_3$, and as a result many of its properties are inherited from that larger ODE system.

\subsubsection*{Basic properties of the \sunitary{3}-invariant soliton system}
We make several observations about the structure of the ODE system~\eqref{eq:ODEs:SU3}
governing \sunitary{3}-invariant Laplacian solitons.
Define
\[
H:= H(f_1,f_2,f_3,f_1',f_2',f_3',u,\lambda) = \sum_{i=1}^3{(\tau_i - u f_i^2)} - 2 \lambda f_1 f_2 f_3 = \tb - u \fb -2 \lambda f_1 f_2 f_3,
\]
where as above we use~\eqref{eq:tau:i:SU3} to express the torsion components $\tau_i$ in terms of $f_i$ and $f_i'$ and 
where for a more compact expressions we introduce the notation
\begin{equation}
\label{eq:tb:fb}
\tb:= \tau_1 + \tau_2 + \tau_3, \qquad \fb:=f_1^2+f_2^2+f_3^2.
\end{equation}
Then the four equations in~\eqref{eq:SU3:closure} and~\eqref{eq:ODEs:SU3:taui} imply that $H'=0$
on any \sunitary{3}-invariant Laplacian soliton. However, the final equation~\eqref{eq:ODEs:SU3:tausum} 
implies that in fact $H \equiv 0$. For this reason we will sometimes 
refer to equation~\eqref{eq:ODEs:SU3:tausum} as the \emph{conservation law} for the system~\eqref{eq:ODEs:SU3},
since it fixes the value of the conserved quantity $H$ to be zero.

An immediate consequence of this observation is that if $f_i$, $\tau_i$ and $u$ satisfy 
equations~\eqref{eq:SU3:closure}, \eqref{eq:ODEs:SU3:tausum} and any two out of the three equations~\eqref{eq:ODEs:SU3:taui} then they necessarily satisfy the full system of ODEs. Also if $\lambda \neq 0$ then differentiating equation~\eqref{eq:ODEs:SU3:tausum} and subtracting the sum of the three equations in~\eqref{eq:ODEs:SU3:taui} 
shows that equation~\eqref{eq:SU3:closure} is a consequence of the other equations in~\eqref{eq:ODEs:SU3}.
Using~\eqref{eq:2form:14} and the fact that $\tau$ is of type $14$ we can also rewrite the left-hand side of the conservation law as
\[
\tb= \tau_1 + \tau_2 + \tau_3 = \frac{f_1^2-f_3^2}{f_1^2} \tau_1 + \frac{f_2^2-f_3^2}{f_2^2} \tau_2.
\]

Equations~\eqref{eq:SU3:closure} and~\eqref{eq:ODEs:SU3:tausum} involve at most first derivatives of the $f_i$ and no
derivatives of $u$. Moreover each term $f_i'$ appears linearly in both equations, $u$ does not appear in ~\eqref{eq:SU3:closure}
and appears linearly in~\eqref{eq:ODEs:SU3:tausum}.
The $i$th equation in~\eqref{eq:ODEs:SU3:taui} depends linearly on $f_i''$, does not depend on the other 2nd derivatives of the $f$, and also depends on all the first derivatives $f_1', f_2', f_3'$ and $u'$. 

\begin{remark}
\label{rmk:scaling}
It is often useful to keep in mind how solutions of the ODE system~\eqref{eq:ODEs:SU3} behave under rescaling:
any solution $(f_1,f_2,f_3, u)(t)$ of~\eqref{eq:ODEs:SU3} for a given value of $\lambda \in \R$ transforms into 
a solution of~\eqref{eq:ODEs:SU3} with $\hat{\lambda}= \mu^{-2} \lambda$ provided we make the following rescalings
\begin{equation}
\label{eq:scaling}
(t,f_i, u, \lambda, \tau, \varphi, g_\varphi, *\varphi) \mapsto (\mu t, \mu f_i,  \mu^{-1}u, \mu^{-2} \lambda,  \mu\tau, \mu^3 \varphi, \mu^2 g_\varphi, \mu^4\!\ast\!\varphi).
\end{equation}
We will call the powers of $\mu$ that appear above the scaling weights.
Since we also have $\partial_t \mapsto \mu^{-1}\partial_t$, each successive time-derivative of $f_i$ or $u$ decreases its scaling weight by one.
Note the elementary but important fact that re-scalings of steady solitons are steady solitons.
\end{remark}

\subsection{A first-order reformulation of the \texorpdfsunitary{3}-invariant Laplacian soliton system}\label{ss:order1:su3}

In this section we reformulate the mixed-order ODE system~\eqref{eq:ODEs:SU3}
for \sunitary{3}-invariant Laplacian solitons on $M^0$ as a nonsingular real
analytic first-order system in an alternative system of variables.
The proof is given in the next subsection, and exploits the conservation
law~\eqref{eq:ODEs:SU3:tausum}, %
the equation~\eqref{eq:LS:1} satisfied by the $1$ component of a soliton and the type $14$ constraint on the torsion $\tau$~\eqref{eq:2form:14}.

Some simple consequences of the reformulation are that for each fixed $\lambda$
there is a $4$-parameter family of local real analytic \sunitary{3}-invariant
Laplacian solitons (up to time translation), and constraints on how a solution
with finite lifetime must degenerate.
The reformulation will also enable us to apply a known
technique for proving the existence of solutions that extend smoothly 
over the singular orbit $\CP^2$: see Section~\ref{s:smooth:closure:solns}.

\begin{prop}
\label{prop:1st:order:ODE:SU3}
Let $f=(f_1, f_2, f_3): I \to \R^3$ be a triple of positive functions and let $\tau_1, \tau_2, \tau_3:I \to \R$ be functions defined on the connected interval $I$ satisfying the first-order ODE system
\begin{subequations}
\label{eq:SU3:order1}
\begin{align}
\label{eq:SU3:fi:dot}
(f_i^2)' & = \tau_i - \frac{f_i^2}{f_1 f_2 f_3}\left( 2f_i^2 - \fb\right),\\
\fb \,\tau_i' %
\label{eq:SU3:taui:dot}
&= \frac{4\lambda }{3}f_i^2 S_i + \tb \left( \tau_i - \frac{2f_i^4}{f_1 f_2 f_3} \right) + \frac{1}{3} \abs{\tau}^2 \fb f_i^2,
\end{align}
\end{subequations}
where the quantity $S_i$ is defined by 
\begin{equation}
\label{eq:Si:def}
S_i := 3f_i^2 - \fb - \frac{3 f_1 f_2 f_3 \tau_i}{2 f_i^2},
\end{equation}
and $\abs{\tau}^2 = \sum\frac{\tau_i^2}{f_i^4}$ by \eqref{eq:tau_square},
and the initial data satisfies $\sum{\frac{\tau_i}{f_i^2}}=0$ for some given $t_0 \in I$. 
Then for $\lambda \in \R$ the triple $(\varphi_f,X,\lambda)$ 
\begin{align*}
\varphi_f &:= \omega_f \wedge dt + \Real{\Omega_f} = (f_1^2 \omega_1 + f_2^2 \omega_2 + f_3^2 \omega_3) \wedge dt + f_1 f_2 f_3 \alpha, \\
&X:=u\, \partial_t \quad \text{with $u$ defined by\quad } u\fb:= -2\lambda f_1 f_2 f_3 + \tb
\end{align*}
is an \sunitary{3}-invariant closed Laplacian soliton on $I \times \flag$, \ie $(f_1,f_2,f_3,u)$ satisfies the ODE system~\eqref{eq:ODEs:SU3}.
In particular $\sum{\frac{\tau_i}{f_i^2}}=0$ holds on $I$. 
Conversely, (up to discrete symmetries) any \sunitary{3}-invariant closed Laplacian soliton on $I \times \flag$ arises from a solution of~\eqref{eq:SU3:order1} 
that satisfies $\sum{\frac{\tau_i}{f_i^2}}=0$ throughout its lifetime.
\end{prop}

Before proving the result in the next subsection we note some consequences. 

\begin{remark}
\label{rmk:phase:space}
Define a $5$-dimensional smooth (in fact real analytic) noncompact manifold $\mathcal{P}$ by 
\[
\mathcal{P}:= \left\{(f,\tau) \in (\R^3_> \times \R^3) \, | \sum{\tau_i f_i^{-2}}=0 \right\}.
\]
Proposition~\ref{prop:1st:order:ODE:SU3} implies that $\mathcal{P}$ is invariant under the flow of the first-order ODE system~\eqref{eq:SU3:order1}, 
that the open dense set of principal orbits of any $\sunitary{3}$-invariant Laplacian soliton gives rise to an integral curve of~\eqref{eq:SU3:order1} 
that remains in $\mathcal{P}$ and conversely that any integral curve of~\eqref{eq:SU3:order1} that starts in $\mathcal{P}$ gives rise to an $\sunitary{3}$-invariant Laplacian soliton. In other words, we can view $\mathcal{P}$ as the phase space that parametrises all possible principal orbits of $\sunitary{3}$-invariant Laplacian solitons.

Note also that $\mathcal{P}$ has some additional structure:  it is naturally a real analytic $2$-plane subbundle $\xi$ of the trivial real $3$-plane bundle $\underline{\R}^3$ over the positive octant $\R^3_>$
with projection map $\pi: (f,\tau) \mapsto f$ and fibre $\xi_f:=\{(\tau_1,\tau_2,\tau_3) \in \R^3\, | \sum{\tau_i f_i^{-2}}=0\}$.
We note also that the fibre $\xi_f$ of $\xi$ depends 
only on the homothety class $[f_1^2,f_2^2,f_3^2] \in \Sph^2_>$. In other words, $\mathcal{P}$ is the radial extension of a real $2$-plane bundle 
over $\Sph^2_>$. \end{remark}

\begin{corollary}
\label{cor:su3:local}
For any $\lambda \neq 0$ there is a $4$-parameter family of distinct local real
analytic \sunitary{3}-invariant Laplacian solitons up to translation; 
there is a $3$-parameter family of distinct local real analytic \sunitary{3}-invariant steady Laplacian solitons up to translation and scale.
\end{corollary}

\begin{proof}
On any principal orbit the system~\eqref{eq:SU3:order1} is a nonsingular real analytic first-order system of ODEs. 
Hence standard ODE theory gives the existence of unique (local in $t$) real analytic solutions depending 
real analytically on the initial data, \ie on the choice of principal orbit or equivalently of a point in the 
phase space $\mathcal{P}$ described in Remark~\ref{rmk:phase:space}. Since $\mathcal{P}$ is $5$-dimensional 
and an \sunitary{3}-invariant Laplacian soliton corresponds to an integral curve of ~\eqref{eq:SU3:order1} that remains in  $\mathcal{P}$ 
there is a $4$-parameter family of local real analytic \sunitary{3}-invariant Laplacian solitons for any $\lambda \in \R$. 
However,  since rescalings take steady solitons to other steady solitons, geometrically this implies the existence of a $3$-parameter family of distinct local \sunitary{3}-invariant steady Laplacian solitons, but a $4$-parameter family of distinct local \sunitary{3}-invariant Laplacian solitons for any $\lambda \neq 0$.
\end{proof}

Another easy but important consequence of the first-order formulation of the soliton ODE system is the following
incompleteness criterion. 
\begin{prop}
\label{prop:lifetime}
Any forward-incomplete soliton
has some ratio $\frac{f_i}{f_j}$ unbounded when approaching its extinction time.
\end{prop}
\begin{proof}
Suppose for a contradiction that all the ratios $f_i/f_j$ remain bounded as $t \to t_*$, the extinction time.
First we observe that on any finite interval on which all ratios $f_i/f_j$ remain bounded above then
the $f_i$ are bounded away from zero and bounded above.
Indeed writing 
\[
f_i^3 = f_1 f_2 f_3 \left(\frac{f_i}{f_j}\right) \,\left(\frac{f_i}{f_k}\right)
\]
and using the fact that $f_1f_2 f_3$ is increasing implies that each $f_i$ is bounded away from $0$ as $t \to t_*$.
Moreover, all the $f_i$ are bounded above if and only if $f_1 f_2 f_3$ is. 
Since 
\[
\frac{d}{dt} \log{g} = \frac{1}{6} \left( \frac{f_1^2+f_2^2+f_3^2}{f_1 f_2 f_3} \right) 
= \frac{1}{6f _1} \left( \frac{f_1}{f_2} \right) \left(\frac{f_1}{f_3} \right) 
\left( 1 + \frac{f_2^2}{f_1^2} + \frac{f_3^2}{f_1^2}\right),
\]
$g^3=f_1 f_2 f_3$ is therefore also bounded above as $t \to t_*$ and hence so are all the $f_i$.
This also yields upper and lower bounds for all the $\tilde \tau_i$, because $\tilde \tau_i' = \lambda f_i^2$.
Using $\tilde{\tau}_i = \tau_i - u f_i^2$ , where $u$ is the coefficient of the soliton vector field $X = u\, \partial_t$, 
and the type $14$ condition on $\tau$ \eqref{eq:tau_14} we see that 
\[
\sum_i\, \frac{\tilde{\tau}_i}{f_i^2} = -3u.
\]
Hence $u$ is also bounded as $t \to t_*$ and therefore so are all of the
$\tau_i = \tilde \tau_i + u f_i^2$.  Using the first-order ODE system~\eqref{eq:SU3:order1} this also implies bounds for the first derivatives
of all the $f_i$ and $\tau_i$. Thus we have finite limiting values of all the $f_i$
(that are strictly positive) and all the $\tau_i$ as $t \to t^*$, and we
can use those limiting values as (regular) initial data for the first-order ODE
system~\eqref{eq:SU3:order1}. Then
local existence for this system allows us to extend the original solution beyond its extinction time $t_*$.
\end{proof}

One further consequence of the first-order ODE system is the following simple, but important, result about 
sign changes or preserved orderings for the coefficients $f_i$ and $\tilde \tau_i:= \tau_i - uf_i^2$ 
that holds for all non-shrinking solitons. 
\begin{lemma}
Suppose $\lambda \ge 0$. 
For $i \neq j$ define the quantity $T_{ij}:=(\tilde \tau_i - \tilde \tau_j)(f_i^2-f_j^2)$.  Then either $T_{ij}$ is identically zero 
or any zero of $T_{ij}$ is nondegenerate with $T_{ij}'>0$. 
Hence the sign of $T_{ij}$ changes at most once and if it changes sign then it must change from negative to positive.
If $\lambda=0$ and $T_{ij}$ vanishes identically then either $f_i-f_j$ vanishes identically or it has no zeros. 
\end{lemma}

\begin{proof}
Using~\eqref{eq:ODEs:SU3:taui} we see that the  evolution of $T_{ij}$ is given by
\[
T_{ij}' = (\tilde \tau_i - \tilde \tau_j)' (f_i^2-f_j^2) + (\tilde \tau_i - \tilde \tau_j)(f_i^2 - f_j^2)' = \lambda (f_i^2-f_j^2)^2 + (\tilde \tau_i - \tilde \tau_j)(f_i^2 - f_j^2)'.
\]
If $T_{ij}$ vanishes at a point then either $f_i-f_j$ or $\tilde \tau_i - \tilde \tau_j$ or both vanish at that point. 
If both $f_i-f_j$ and $\tilde \tau_i - \tilde \tau_j$ vanish at the same point then also $\tau_i=\tau_j$ there.
Then the unique solution of the first-order ODE system with these initial values has $f_i=f_j$ and $\tau_i=\tau_j$ throughout its lifetime
and so $T_{ij}$ vanishes identically. (Moreover, this solution has a $\Z_2$-symmetry given by exchanging $f_i$ and $f_j$.)

So we may now assume that  $f_i-f_j$ and $\tilde \tau_i - \tilde \tau_j$ cannot both vanish at the same point. 
In the case that $f_i-f_j$ vanishes at some point, 
then using~\eqref{eq:tau:i:SU3a} we see that
 $(f_i^2-f_j^2)' = \tau_i-\tau_j = \tilde \tau_i - \tilde \tau_j$ holds there and so $T_{ij}' = (\tilde \tau_i - \tilde \tau_j)^2 >0$. 
In the case that $\tilde \tau_i - \tilde \tau_j$ vanishes at some point and $\lambda>0$ then there $T_{ij}' = \lambda (f_i^2-f_j^2)^2 >0$.

If $\lambda =0$ then $\tilde \tau_i - \tilde \tau_j$ is constant and so if it vanishes at one point it vanishes identically
and therefore $T_{ij}$ vanishes identically. 
In that case then the derivative of $f_i^2- f_j^2$ at a point where $f_i=f_j$ is given by $(f_i^2-f_j^2)' = \tau_i - \tau_j = \tilde \tau_i - \tilde \tau_j=0$. Thus $\tau_i=\tau_j$ and $f_i=f_j$ hold at the same point and so again by uniqueness of solutions to the first-order ODE system both equalities hold throughout
the lifetime of the solution. So if $\lambda=0$ and $\tilde \tau_i = \tilde \tau_j$ then either $f_i^2-f_j^2$ is identically zero or it has no zeros.
\end{proof}

\begin{remark*}
A soliton possessing a $\Z_2$-symmetry that exchanges $f_i$ and $f_j$ for some pair $i \neq j$ has $T_{ij}$ vanishing identically, but the converse is not true. 
For instance, any~\sunitary{3}-invariant torsion-free~\gtstr~has all the $T_{ij}$ vanishing identically, but generically none of the different
$f_i$ will be equal.
\end{remark*}

\begin{corollary}
\label{cor:sign_changes}
If $\lambda \ge 0$ then the sign of $f_i-f_j$ can change at most once and if such a sign change occurs then $\tilde \tau_i - \tilde \tau_j$ has a fixed sign
throughout the lifetime of the solution. Moreover, if at some instant $\tilde \tau_i - \tilde \tau_j$ and $f_i-f_j$ are both positive 
then they remain so at all subsequent times. 
\end{corollary}

If we think of two variables as being ``aligned'' if
$T_{ij} = (\tilde\tau_i - \tilde\tau_j)(f_i^2-f_j^2) > 0$, then  the previous result says that when $\lambda \ge 0$ a pair of variables
can go from
non-aligned to aligned, but not vice versa. Thus any given pair of the three variables $f_i$ will swap
order at most once during the lifetime of a solution.

\subsection{Proof of Proposition~\ref{prop:1st:order:ODE:SU3}}

The equivalence stated in Proposition \ref{prop:1st:order:ODE:SU3}
of the soliton system \eqref{eq:ODEs:SU3} and the first-order system
\eqref{eq:SU3:order1} is proved in the next two lemmas.
 
\begin{lemma}
If $(f_1, f_2, f_3, u)$ is a solution of \eqref{eq:ODEs:SU3} and $\tau_i$
are defined in terms of $f_i$ by \eqref{eq:tau:i:SU3} then
$(f_1, f_2, f_3, \tau_1, \tau_2, \tau_3)$ satisfies \eqref{eq:SU3:order1}. 
\end{lemma} 

\begin{proof}
Since \eqref{eq:SU3:fi:dot} is a restatement of \eqref{eq:tau:i:SU3a}, so
we just need to derive~\eqref{eq:SU3:taui:dot}. 
To this end we now seek to eliminate $u$ from the ODE system~\eqref{eq:ODEs:SU3} and hence obtain expressions for $\tau_i'$ in terms of~$\lambda$, the coefficients $f_i$ (but not their derivatives) and the $\tau_i$. 
Recall the conservation law~\eqref{eq:ODEs:SU3:tausum} for the ODE system has the form
\begin{equation}
\tag{CL}
\label{eq:SU3:conserve}
u \fb =  \tb -2\lambda f_1 f_2 f_3.
\end{equation}
Substituting this expression for $u$ into~\eqref{eq:LS:1:su3+}, which recall was simply the $1$ component of the Laplacian soliton equation, yields
\begin{equation}
\label{eq:u':su3}
3(u' + \lambda) = 2\lambda -  \frac{3\tb}{f_1 f_2 f_3} + \sum{\frac{\tau_i^2}{f_i^4}} = 2\lambda + \abs{\tau}^2 -  \frac{3\tb}{f_1 f_2 f_3}.  
\end{equation}
Since we are assuming all $f_i$ to be positive the system of equations~\eqref{eq:ODEs:SU3:taui} is equivalent to
\begin{equation}
\label{eq:taui:tilde:dot}
\frac{\tau_i'}{f_i^2} = u \frac{(f_i^2)'}{f_i^2} + (u'+\lambda) \qquad \text{for\ } i=1, 2, 3.
\end{equation}
Substituting the expressions for $(f_i^2)'$, $u$ and $u' + \lambda$ given in~\eqref{eq:SU3:fi:dot},~\eqref{eq:SU3:conserve} and~\eqref{eq:u':su3} respectively into the right-hand side of these equations yields
\[
\frac{\tau_i'}{f_i^2} = \left(\frac{\tb-2\lambda f_1 f_2 f_3}{\fb}\right) \left( \frac{\tau_i}{f_i^2} - \frac{(2f_i^2-\fb)}{f_1 f_2 f_3}  \right)
+
 \frac{1}{3} \left( 2\lambda + \abs{\tau}^2 -  \frac{3\tb}{f_1 f_2 f_3} \right).
\]
Multiplying both sides by $f_1 f_2 f_3 \fb$ and rearranging yields
\[
\begin{split}
\label{eq:taui':fi2}
\frac{f_j f_k}{f_i} \fb \,\tau_i' &= \left( \tb - 2 \lambda f_1 f_2 f_3 \right) \left(  \tfrac{f_j f_k}{f_i} \tau_i - (2f_i^2 - \fb) \right) + 
\tfrac{1}{3}\fb \left( f_1 f_2 f_3 \left(2\lambda  + \abs{\tau}^2\right) - 3 \tb \right)  \\
&=  \tfrac{4\lambda}{3} f_1 f_2 f_3 (3f_i^2 - \fb)  - \left(2 \tb f_i^2 + 2 \lambda f_j^2 f_k^2 \tau_i \right) 
 + \left(\tb \tfrac{f_j f_k}{f_i} \tau_i  + \tfrac{1}{3} \fb f_1 f_2 f_3 \abs{\tau}^2 \right). 
\end{split}
\]
In the last line we have first grouped the terms not involving the coefficients of $\tau$, then 
those linear in its coefficients and finally those quadratic in its coefficients.
The previous equation is equivalent to 
\[
\fb \tau_i' = \frac{4\lambda}{3} f_i^2(3f_i^2-\fb) - 2\tb \frac{f_i^4}{f_1 f_2 f_3}  +(\tb  -2\lambda f_1 f_2 f_3) \tau_i + \frac{1}{3}\abs{\tau}^2 \fb f_i^2.
\]
Grouping the terms containing $\lambda$ together gives the form of the equation stated in the Proposition.
\end{proof}

\begin{lemma}
Suppose that $(f_1,f_2,f_3, \tau_1,\tau_2,\tau_3)$ satisfy~\eqref{eq:SU3:order1}
and that initially the type 14 constraint \eqref{eq:tau_14} is satisfied,
\ie $\sum{\tfrac{\tau_i}{f_i^2}}=0$ at some initial $t_0 \in I$.
Define a function $u$ in terms of $f_i$ and $\tau_i$ via the conservation
law~\eqref{eq:SU3:conserve}. Then $(f_1, f_2, f_3, u)$
solves \eqref{eq:ODEs:SU3}. 
\end{lemma}

\begin{proof}
First we want to prove that the resulting \sunitary{3}-invariant~\gtstr~$\varphi$ is closed and
has the invariant $2$-form $\tau:= \sum{\tau_i \omega_i}$ as its torsion $2$-form (and therefore satisfies $\sum{\tfrac{\tau_i}{f_i^2}} =0$ 
throughout its lifetime). 
It is straightforward to verify that the invariant $3$-form $\varphi$ is closed if and only if the condition $\sum{\tfrac{\tau_i}{f_i^2}}=0$ holds 
for all $t\in I$.
Since we have defined $u$ via the conservation law~\eqref{eq:SU3:conserve} we can rewrite~\eqref{eq:SU3:taui:dot} as 
\[
\frac{\tau_i'}{f_i^2} = \frac{4 \lambda E_i}{3\fb} + \frac{u \tau_i}{f_i^2} - \frac{2 f_i^2 \tb}{f_1 f_2 f_3 \fb} + \frac{1}{3}\abs{\tau}^2,
\]
where $E_i:= 3f_i^2 - \fb$. 
Summing these equations over $i$ yields
\[
\left(\sum{\frac{\tau_i'}{f_i^2}} \right)= -\frac{2\tb}{f_1f_2f_3} + \abs{\tau}^2  + u \left(\sum{\frac{\tau_i}{f_i^2}}\right).
\]
Using~\eqref{eq:SU3:fi:dot} yields 
\[
\frac{\tau_i}{f_i^2} \frac{(f_i^2)'}{f_i^2} = \frac{\tau_i^2}{f_i^4} - \frac{2 \tau_i}{f_1 f_2 f_3} + \frac{\fb}{f_1 f_2 f_3} \frac{\tau_i}{f_i^2},
\] 
and therefore summing over $i$ (and using \eqref{eq:tau_square}) gives
\[
\sum{\frac{\tau_i}{f_i^2} \frac{(f_i^2)'}{f_i^2}} = \abs{\tau}^2 - \frac{2\tb}{f_1 f_2 f_3} + \frac{\fb}{f_1f_2f_3} \left(\sum{\frac{\tau_i}{f_i^2}}\right).
\]
Hence we find that 
\[
\frac{d}{dt} \left(\sum{\frac{\tau_i}{f_i^2} }\right)= \left(\sum{\frac{\tau_i'}{f_i^2}} \right) - \sum{\frac{\tau_i}{f_i^2} \frac{(f_i^2)'}{f_i^2}} = 
\left(\sum{\frac{\tau_i}{f_i^2}}\right) \left(u - \frac{\fb}{f_1f_2f_3} \right).
\]
In particular, this equation implies that if a solution to~\eqref{eq:SU3:order1} has $\sum{\tfrac{\tau_i}{f_i^2}}=0$ at some $t_0\in I$ 
then it continues to vanish for all $t \in I$. Hence $\varphi$ defines a \emph{closed} \sunitary{3}-invariant~\gtstr~whose
torsion 2-form $\tau \in \Omega^2_{14}$ is equal to the invariant $2$-form $\tau: = \sum{\tau_i \omega_i}$. 

It remains to prove that $(f_1,f_2,f_3,u)$ satisfies the \sunitary{3}-invariant Laplacian soliton equations~\eqref{eq:ODEs:SU3}. 
Clearly~\eqref{eq:SU3:closure} is satisfied (because we just proved that $\varphi$ is closed)
and since we are defining $u$ to satisfy the conservation law our solution automatically satisfies~\eqref{eq:ODEs:SU3:tausum}.
Therefore it remains only to check that $u$ satisfies~\eqref{eq:ODEs:SU3:taui} for all $i$.
 Earlier in the proof we  saw that~\eqref{eq:ODEs:SU3:taui} holds if and only if~\eqref{eq:taui:tilde:dot} does.
Differentiating the conservation law~\eqref{eq:SU3:conserve} that defines $u$ yields
\[
u'  \fb  = \tb'  -u (\fb)' -2 \lambda (f_1f_2f_3)' = \tb'  -u (\fb)' - \lambda \fb - \frac{2\lambda f_1 f_2 f_3}{\fb} \left(\sum{\frac{\tau_i}{f_i^2}}\right)
\]
and hence
\[
(u'+\lambda) = \frac{\tb'}{\fb} - \frac{u (\fb)'}{\fb} - \frac{2\lambda f_1 f_2 f_3}{\fb} \left(\sum{\frac{\tau_i}{f_i^2}}\right) =  \frac{\tb'}{\fb} - \frac{u (\fb)'}{\fb},
\]
where the final equality holds because by our assumptions the term  $\sum{\tfrac{\tau_i}{f_i^2}}$ vanishes (initially and therefore throughout the lifetime of the solution).
Letting
\begin{equation}
\label{eq:D}
D(f):= \sum_{i<j}{(f_i^2-f_j^2)^2},
\end{equation}
using~\eqref{eq:SU3:taui:dot} and summing over $i$ yields
\[
\frac{\tb'}{\fb} = \left(\frac{4\lambda D(f)}{3 (\fb)^2} - \frac{2\tb \fbf}{f_1 f_2 f_3 \, (\fb)^2}\right)+ \frac{1}{3} \abs{\tau}^2 + \frac{u \tb}{\fb} =  \left(-\frac{4\lambda}{3} - \frac{2u \fbf}{f_1 f_2 f_3 \fb} \right)+ \frac{1}{3} \abs{\tau}^2 + \frac{u \tb}{\fb},
\]
while using~\eqref{eq:SU3:fi:dot} and summing over $i$ yields
\[
\frac{u(\fb)'}{\fb} = \frac{u\,\tb}{\fb} + \frac{u\fb}{f_1 f_2 f_3} - \frac{2u\fbf}{f_1 f_2 f_3\, \fb}.
\]
Hence
\[
(u'+\lambda) = \frac{\tb'}{\fb} - \frac{u(\fb)'}{\fb} = - \frac{4\lambda}{3}+  \frac{1}{3}\abs{\tau}^2 - \frac{u \fb}{f_1 f_2 f_3}.
\]
Therefore the right-hand side of~\eqref{eq:taui:tilde:dot} is equal to 
\[
\left(\frac{u\tau_i}{f_i^2} - \frac{2uf_i^2}{f_1f_2f_3} + \frac{u \fb}{f_1f_2f_3} \right) + \left( -\frac{4\lambda}{3}+ \frac{1}{3}\abs{\tau}^2 - \frac{u \fb}{f_1 f_2 f_3} \right) =  \frac{u\tau_i}{f_i^2} - \frac{2u f_i^2}{f_1f_2f_3} - \frac{4\lambda}{3} + \frac{1}{3}\abs{\tau}^2.
\]
Hence the difference between the two sides of~\eqref{eq:taui:tilde:dot} is
\[
\frac{\tau_i'}{f_i^2} - \left(\frac{u(f_i^2)'}{f_i^2} + (u'+\lambda) \right) = \frac{4 \lambda (E_i+\fb)}{3\fb}  - \frac{2 f_i^2 \tb}{f_1 f_2 f_3 \fb} + \frac{2u f_i^2}{f_1f_2f_3} = 2\left(\frac{2\lambda f_1 f_1 f_3 + u\fb-\tb}{ f_1f_2 f_3\fb}\right) f_i^2 =0
\]
as required (the final equality follows because of the conservation law~\eqref{eq:SU3:conserve} used to define $u$).
\end{proof}

\subsection{Specialisation to the \texorpdfSp{2}-invariant case}
By Remark \ref{rmk:involution}, we can treat \Sp{2}-invariant
\gtstr s as a special case of \sunitary{3}-invariant ones.
For future purposes we record the form of the first-order system that arises from specialising the system~\eqref{eq:SU3:order1} to 
the \Sp{2}-invariant case, \ie by setting $f_2=f_3$ and $\tau_2=\tau_3$. 

\begin{prop}
\label{prop:1st:order:tau2}
If $(f_1, f_2,  \tau_1, \tau_2, u)$ is any solution to the \Sp{2}-invariant soliton system 
\eqref{eq:ODEs:Sp2} then $(f_1,f_2,\tau_2)$ satisfies the
first-order ODE system
\begin{equation}
\label{eq:ODE:tau2'}
(f_1^2)' = 2f_1 - \frac{f_1^2}{f_2^2} (f_1+ 2\tau_2), \qquad  (f_2^2)' = f_1 + \tau_2, \qquad \tau_2' = \frac{4R_1 S}{3f_1(f_1^2+2f_2^2)},
\end{equation}
where 
\[R_1:=\lambda f_1f_2^2-3\tau_2,  \quad \ S:=f_2^2-f_1^2-\tfrac{3}{2}f_1\tau_2.
\]
Conversely, if $(f_1,f_2, \tau_2)$ is a solution to the first-order system \eqref{eq:ODE:tau2'} with $f_1,f_2>0$ 
and if we define $\tau_1$ and $u$ in terms of $(f_1,f_2, \tau_2)$ by
\begin{equation}
\label{eq:tau1:u:tau2}
\tau_1 := -\frac{2f_1^2}{f_2^2}\tau_2, \qquad u := \frac{2(f_2^2-f_1^2)\tau_2 -2 \lambda f_1f_2^4}{f_2^2(f_1^2+2f_2^2)},
\end{equation}
then $(f_1,f_2,\tau_1,\tau_2,u)$ solves the \Sp{2}-invariant soliton system \eqref{eq:ODEs:Sp2}.
\end{prop}
\begin{proof}
The proof is a routine calculation obtained by 
specialising the formulae from the \sunitary{3}-invariant case to the \Sp{2}-invariant case where $f_2=f_3$ and $\tau_2= \tau_3$.
Below we give a few of the more important formulae since some of them are needed in our subsequent analysis of \Sp{2}-invariant shrinkers and expanders
in the sequel~\cite{HJN}.

First we note
\[
\tb = \tau_1 + 2 \tau_2= \frac{2\tau_2}{f_2^2}(f_2^2-f_1^2) \quad \text{and} \quad |\tau|^2 = 
 \sum{\frac{\tau_i^2}{f_i^4}} = \frac{\tau_1^2}{f_1^4} +  \frac{2\tau_2^2}{f_2^4} =  \frac{6\tau_2^2}{f_2^4}.
\]
Hence~\eqref{eq:u':su3} specialises to
\begin{equation}
\label{eq:udot:Sp2}
u' =- \frac{\lambda}{3} - \frac{2 \tau_2}{f_1 f_2^4}(f_2^2-f_1^2) + \frac{2 \tau_2^2}{f_2^4} = - \frac{\lambda}{3} - \frac{\tau_2}{f_1f_2^4}(2S+f_1\tau_2).
\end{equation}
Note also that $S_2=S_3= S=(f_2^2-f_1^2) - \tfrac{3}{2}f_1 \tau_2$ and $S_1=-2S_2$. 
The second and third summands on the right-hand side of the equation for $\tau_2'$ in~\eqref{eq:SU3:taui:dot} reduce to
\begin{align*}
\tb \left(\tau_2 - \frac{2f_2^4}{f_1f_2 f_3} \right) &= \tb \left(\tau_2 - \frac{2f_2^4}{f_1f_2^2} \right) = \frac{2\tau_2}{f_1f_2^2} (f_2^2-f_1^2)(\tau_2 f_1 - 2f_2^2)
\intertext{and}
\frac{1}{3}\abs{\tau}^2 \fb f_2^2 &= \frac{2 \tau_2^2 (f_1^2+2f_2^2)}{f_2^2}
\end{align*}
respectively. Their sum is therefore
\[
\frac{2\tau_2}{f_1f_2^2} \left((f_2^2-f_1^2)(\tau_2f_1-2f_2^2) + \tau_2 f_1 (f_1^2+2f_2^2) \right) = - \frac{4\tau_2 S}{f_1}.
\]
Hence the right-hand side of~\eqref{eq:SU3:taui:dot} specialises to 
\[
4S\left(\frac{\lambda f_2^2}{3} -  \frac{\tau_2}{f_1}\right) = \frac{4S}{3f_1} \left(\lambda f_1f_2^2 - 3\tau_2 \right) = \frac{4R_1S}{3f_1}. 
\]
Therefore the ODE for $\tau_2'$ reduces to $f_1(f_1^2+2f_2^2) \tau_2' = \tfrac{4}{3}R_1 S$ as claimed.
\end{proof}

\begin{corollary}
\label{cor:sp2:local}
For any $\lambda \neq 0$ there is a $2$-parameter family of distinct local real analytic $\Sp{2}$-invariant Laplacian solitons up to translation; 
there is a $1$-parameter family of distinct local real analytic $\Sp{2}$-invariant steady Laplacian solitons up to translation and scale.
\end{corollary}

\section{Smooth extension over a singular orbit and finite extinction time}
\label{s:smooth:closure:solns}
The main result of this section is Theorem \ref{thm:SU3:smooth:closure},
a more technical version of
Theorem \ref{mthm:smooth:closure}, concerning
 the existence for any $\lambda\in \R$ 
of a $2$-parameter family of local \sunitary{3}-invariant \gtwo-solitons that extend smoothly over some neighbourhood of the singular orbit $\CP^2 \subset \Lambda^2_-\CP^2$; the same method also yields the existence of a $1$-parameter family of $\textup{Sp}(2)$-invariant \gtwo-solitons 
that extend smoothly over the singular orbit $\Sph^4 \subset \Lambda^2_-\Sph^4$. 

The general approach is by now relatively standard. It originates in the setting of 
cohomogeneity-one Einstein metrics  in the work of Eschenburg--Wang \cite{eschenburg:wang};
see also~\cites{FHN:JEMS, Foscolo:Haskins:NK} where this approach was adapted
to solve the same problem for cohomogeneity-one torsion-free\mbox{~\gtstr s}
and nearly K\"ahler structures respectively and 
(Maria)~Buzano's work~\cite{Buzano:IVP:GRS} in the setting of cohomogeneity-one gradient Ricci solitons.
Given the first-order reformulation of the \sunitary{3}-invariant \gtwo-soliton equations derived in the previous section it is straightforward 
to adapt this approach to our setting:
the conditions that a smooth invariant soliton defined on the principal set $M^0$ extends smoothly over the singular orbit lead to a singular initial value problem for the first-order real analytic ODE system~\eqref{eq:SU3:order1}. 
The singularities turn out to be of regular type: for suitable initial data we prove the existence of a unique formal power series solution to this system; then a general result 
guarantees the existence of a real analytic solution whose Taylor series is equal to the formal series solution (which is therefore convergent). 

The formal power series solutions can in principle be computed symbolically (in terms of $\lambda$ and two additional real parameters) using a computer algebra package: a 
number of terms for these power series solutions are detailed in Appendix~\ref{app:sc:power:series}. Explicit knowledge of
some of
the lowest-order terms proves to be important in several later arguments and some general features of the series 
also turn out to be suggestive when trying to understand limiting behaviour of solutions in certain regimes. In this section consideration of these power series solutions 
leads to Theorem~\ref{thm:explicit:AC:shrinker}, which details the explicit
complete (asymptotically conical) shrinkers on $\Lambda^2_-\Sph^4$ and
$\Lambda^2_-\CP^2$ from Theorem~\ref{mthm:shrinkers}. 

Rather less standard is Theorem~\ref{thm:fwd:incomplete} (a more technical
version of Theorem \ref{mthm:extinction}) in which we construct solitons with
finite forward-extinction time (so these solitons are necessarily incomplete). 
It follows from this result that the flow-lines of these incomplete solutions fill up an open subset of phase space, 
so that one expects that this type of incompleteness is a stable property of a soliton. 
Our construction is perhaps best motivated by a Proposition~\ref{prop:steady_extinct} later in the paper, which proves 
that steady solitons with finite future extinction time must satisfy the conditions assumed in our result. 
However, the construction does not rely on that later result and indeed works for any value of the dilation constant. 
Its proof relies on the analysis of a singular initial value problem of regular type (as in the case of smoothly-closing solitons). 

\enlargethispage{0.1\baselineskip}

\subsection{The initial conditions for closing smoothly over a \texorpdfstring{$\CP^2$}{CP\textasciicircum 2} singular orbit.}
\label{ss:IC:CP2}
Table 2 in Cleyton--Swann~\cite{Cleyton:Swann} tells us that for $G/K=\sunitary{3}/\T^2$, complete closed $\sunitary{3}$-invariant
\mbox{\gtstr s} exist only when the singular orbit is $G/H=\sunitary{3}/\unitary{2} \cong \CP^2$.
Without loss of generality (recall we can always arrange this by acting with the Weyl group $W =S_3$) we will take the isotropy group of the singular orbit to be the subgroup $\unitary{2}_1$ defined as the centraliser 
of the matrix $I_1$ given in~\eqref{eq:U2:i}, \ie 
as we approach the singular orbit the size $f_1^2$ of the fibre $\Sph^2_1$ goes to zero. 

The conditions for a smooth $\sunitary{3}$-invariant \gtstr~defined on on $(0,\epsilon) \times\CP^3$ to extend smoothly across a singular orbit $\CP^2 = \sunitary{3}/\unitary{2}$ at $t=0$ were already determined in \cite[\S 9]{Cleyton:Swann}. 
Here we recall their result without proof, referring the reader to \cite{Cleyton:Swann} or Cleyton's thesis for further details as needed. Alternatively some readers may prefer the approach taken by Chi to prove \cite[Prop. 2.8]{Chi:AGAG}: Chi applies the representation-theoretic approach of Eschenburg--Wang~\cite{eschenburg:wang} to study the smooth extension problem for an $\sunitary{3}$-invariant Riemannian metric to extend smoothly over the zero section of $\Lambda^2_-\CP^2$. (Clearly
if the $\sunitary{3}$-invariant~\gtstr~$\varphi$ extends smoothly over $\CP^2$ then the invariant Riemannian metric $g_\varphi$ extends smoothly over 
$\CP^2$, so that Chi's stated conditions are certainly necessary conditions for smooth extension. See \cite[\S 10]{Cleyton:Swann} for a converse.)

For a closed \sunitary{3}-invariant \gtstr~expressed in the form
\[
\varphi_f = \omega_f \wedge dt + \Real{\Omega_f} = (f_1^2 \omega_1 + f_2^2 \omega_2 + f_3^2 \omega_3) \wedge dt +  f_1 f_2 f_3 \alpha
\]
the conditions that $\varphi_f$ extend smoothly over the singular orbit $\CP^2_1 = \sunitary{3}/\unitary{2}_1$ at $t=0$ are:
\begin{itemize}
\item
$f_1$ is odd in $t$ and $\abs{f_1'(0)}=1$,
\item
$f_2^2(t) = f_3^2(-t)$ and $f_2(0) \neq 0$.
\end{itemize}
In particular these conditions force that the product $f_2 f_3$ is even in $t$.
Recall also that by acting with Klein four-group we are able to assume that $(f_1,f_2,f_3)$ is a positive triple for all $t>0$. 
In particular we must have that $f_1'(0)=+1$ and $f_2(0)=f_3(0)>0$.
With these sign assumptions understood then the smooth extension conditions above then imply that $f_2+f_3$ is even and $f_2-f_3$ is odd. In particular, we can write the coefficients $f_i$ as
\begin{equation}
\label{eq:fi:su3:singular:ic}
f_1 = t + t^3 \hat{f}_1, \quad f_2 = b+ t \hat{f}_2, \quad f_3 = b+ t\hat{f}_3,
\end{equation}
for some $b>0$ and 
new functions $\hat{f}_i$ with the following symmetries: 
\begin{itemize}
\item $\hat{f}_1$, $\hat{f}_2 \hat{f}_3$ and $\hat{f}_2 - \hat{f}_3$ are even in $t$;
\item $\hat{f}_2+\hat{f}_3$ is odd in $t$. 
\end{itemize}
In particular, $\hat{f}_2(0)+\hat{f}_3(0)=0$, so if $\hat{f}_2(0) \neq 0$ then the solutions will not satisfy $f_2=f_3$. 

\pagebreak[2]

\subsection{\texorpdfsunitary{3}-invariant \texorpdfgtwo-solitons extending smoothly over \texorpdfstring{$\CP^2$}{CP\textasciicircum 2}}
\label{ss:su3:sc}
The main result for smoothly-closing~\sunitary{3}-invariant solitons is the
following more technical version of Theorem \ref{mthm:smooth:closure}.
\begin{theorem}
\label{thm:SU3:smooth:closure}
Fix any $\lambda \in \R$ and $b>0$ and $c\in \R$. Then there exists a unique local \sunitary{3}-invariant \gtwo-soliton which closes smoothly on 
the singular orbit $\CP^2_1 \subset \Lambda^2_-\CP^2$ and satisfies
\begin{equation}
\label{eq:smooth:close}
f_1 = t + t^3 \hat{f}_1, \ f_2 = b+ t\hat{f}_2, \ f_3 = b + t \hat{f}_3, \ \tau_1 = t^3 \hat{\tau}_1, \ \tau_2 = c+ t \hat{\tau}_2, \ \tau_3 = -c + t \hat{\tau}_3,
\end{equation}
where the leading-order terms of the $\hat{f}_i$ and $\hat{\tau}_i$ satisfy
\[
\hat{f}_1(0) = -\frac{1}{6b^2} - \frac{2\lambda}{27} + \frac{c^2}{18 b^4}, \quad \hat{f}_2(0) = -\hat{f}_3(0) = \frac{c}{6b},
\]
and 
\[
\hat{\tau}_1(0) = \frac{2(2\lambda b^4-c^2)}{9b^4}, \quad 
\hat{\tau}_2(0)=\hat{\tau}_3(0) = \frac{2}{9} \left(\lambda b^2 + \frac{c^2}{b^2}\right)
\]
respectively. The solution is  real analytic on $[0,\epsilon) \times \sunitary{3}/\T^2$ for some $\epsilon >0$.

Moreover, up to the $\Z_2 \times \Z_2$-action, any \sunitary{3}-invariant \gtwo-soliton which closes smoothly on the singular orbit $\CP^2_1 \subset \Lambda^2_-\CP^2$ belongs to this $2$-parameter family of solutions.
\end{theorem}
We have performed a computer-assisted symbolic computation of the terms of the power series solutions of the $2$-parameter family of smoothly-closing solitons constructed in Theorem~\ref{thm:SU3:smooth:closure} in terms of the real parameters $b$, $c$ and $\lambda$. The first several terms of these solutions are listed in 
Appendix~\ref{app:sc:power:series}.

\begin{remark}
\label{rmk:non_isom}
Any isometry between two of these smoothly-closing solutions
must map the singular orbit to itself and induces an isometry of it. Because $\sunitary{3}$
surjects onto the orientation-preserving isometries of $\CP^2$
equipped with the Fubini--Study metric (the restriction of the metric of a smoothly-closing solution to its singular orbit is a multiple 
of the Fubini--Study metric), if there is an orientation-preserving
isometry between two smoothly-closing solutions, then we can assume
without loss of generality that it acts trivially on the singular orbit and on its
normal bundle and so it must be the identity map. 
Hence the elements of the two-parameter family are all distinct
up to orientation-preserving isometry. However, the involution
defined by multiplying the fibres of $\Lambda^2_-\CP^2$ by $-1$, which acts by fixing $f_1$ and exchanging $f_2$ and $f_3$ 
(recall Remarks \ref{rmk:weyl:asd} and  \ref{rmk:involution_asd})
defines an orientation-reversing isometry
between the solutions with parameters $(b, c)$ and $(b,-c)$.
\end{remark}

\begin{remark}
\label{rmk:scaling_sc}
For $\mu > 0$, the $\lambda$-soliton with parameters $b$ and $c$ is related
to the $\mu^{-2}\lambda$-soliton with parameters $\mu b$ and $\mu c$ by
the scaling symmetry from Remark \ref{rmk:scaling}.
\end{remark}

\begin{remark}
\label{rmk:MZ:singular}
The mean curvature of the singular orbit is zero. In terms of our parameters,
this can be seen as a consequence of the smooth closure conditions implying
that $\lim_{t \to 0}\left(\frac{f_2'}{f_2} + \frac{f_3'}{f_3}\right)=0$. 
The referee pointed out to the authors that a singular orbit in a smoothly-closing cohomogeneity-one space is in fact minimal more generally: 
the mean curvature vector $H$ is a normal vector to the singular orbit that is fixed under the normal isotropy representation, 
and since by standard facts about cohomogeneity-one actions the latter always acts transitively on the unit sphere we must have~$H=0$.
However, when $c \neq 0$ the singular orbit is not totally geodesic.
\end{remark}

The main technical tool for proving the previous theorem is the following general result about regular first-order singular initial value problems. 
(Here regular refers to the fact that the singular term is of order $t^{-1}$ rather than $t^{-k}$ for some $k >1$.)
It appears in B\"ohm's work on complete noncompact cohomogeneity-one Einstein metrics~\cite[Theorem 7.1]{Bohm:BSMF}.
\begin{theorem}
\label{thm:Singular:IVP}
Consider the singular initial value problem
\begin{equation}\label{eq:Singular:IVP}
y'=\frac{1}{t}M_{-1}(y)+M(t,y), \qquad y(0)=y_0,
\end{equation}
where $y$ takes values in $\R^k$, $M_{-1}\co \R^k\ra \R^k$ is a (real) analytic function of $y$ in a neighbourhood of $y_0$ and $M\co\R\times\R^k\ra\R^k$ is analytic in $t,y$ in a neighbourhood of $(0,y_0)$. Assume that
\begin{enumerate}
\item $M_{-1}(y_0)=0$;
\item $h\text{Id}-d_{y_0}M_{-1}$ is invertible for all $h \in \N$, $h \geq 1$.
\end{enumerate}
Then there exists a unique solution $y(t)$ of \eqref{eq:Singular:IVP} which moreover is analytic on $[0,\epsilon)$ for some $\epsilon>0$. Furthermore $y$ depends continuously on $y_0$ satisfying \textup{(i)} and \textup{(ii)}.
\end{theorem}
Condition (ii) guarantees the existence of a unique formal power series solution $y(t)$ to \eqref{eq:Singular:IVP}. 
(In the context of constructing smoothly-closing solutions to the cohomogeneity-one Einstein equations Eschenburg--Wang~\cite[\S 5]{eschenburg:wang}
also constructed formal power series solutions as a first step; in that context an additional difficulty is that such formal power series solutions are not necessarily unique). 
Once a formal power series solution has been shown to exist, one can then follow the arguments of \cite[Theorem 7.1]{malgrange1974}, \cite[\S 6]{eschenburg:wang} or \cite[\S 4]{Ferus:Karcher}: use a truncation of the power series of sufficiently high degree as an approximate solution to \eqref{eq:Singular:IVP} and deform it to a genuine solution by applying a contraction mapping fixed point argument. 
In our setting the fact that any formal power series solution to~\eqref{eq:Singular:IVP} actually converges is due to Malgrange~\cite{malgrange1989}.
As for the continuous dependence on the initial conditions (an issue discussed in~\cite[Theorem 7.1]{Bohm:BSMF} but not in~\cite{eschenburg:wang}): the coefficients of the unique formal power series solution $y(t)$ depend differentiably on $y_0$ satisfying (i) and (ii) and the operator used in the fixed point argument is uniformly contracting with respect to the initial conditions.

\begin{proof}
The structure of the proof is to rewrite the (autonomous) first-order ODE system~\eqref{eq:SU3:order1} for the pair of triples $(f,\tau)$ as a $t$-dependent ODE system for the pair of triples $(\hat{f},\hat{\tau})$
and to show that the hypotheses of Theorem~\ref{thm:Singular:IVP} apply to the latter ODE system. It is therefore particularly important to identify the leading-order 
term $M_{-1}$ in order to be able to understand condition (i) and verify that the hypothesis required in (ii) is satisfied.
\\[0.2em]
\emph{Warmup: the torsion-free case.}
For \sunitary{3}-invariant~torsion-free \gtstr s we already know from Cleyton--Swann that the only solutions (described in Example \ref{ex:BS}) 
that extend smoothly across the $\CP^2$ singular orbit have $\hat{f}_2(0)=\hat{f}_3(0)=0$ and so must have $f_2 \equiv f_3$. 
Cleyton--Swann proved this by finding an explicit parametrisation of any local \sunitary{3}-invariant~torsion-free \gtstr~\cite[p.~217]{Cleyton:Swann} and observing that the conditions on the singular orbit force $f_2^2 \equiv f_3^2$. 
As a warmup for the general \sunitary{3}-invariant soliton case we want to obtain this result in the 
torsion-free by an application of Theorem~\ref{thm:Singular:IVP},  since that method will apply 
even when we are not able to find explicit parametrisations of all the local \sunitary{3}-invariant~\gtwo-solitons.
Moreover, all the calculations presented here are also needed in the more general soliton analysis.
\\[0.2em]
If we assume that all the torsion coefficients $\tau_i$ vanish identically then the first-order ODE system~\eqref{eq:SU3:order1}
reduces to the following first-order ODE system for the triple $f$
\begin{equation}
\label{eq:SU3:tf}
(f_i^2)' = \frac{f_j^2+f_k^2-f_i^2}{f_i f_j f_k}, \quad (ijk) \text{\ a cyclic permutation of\ } (123).
\end{equation}
Now we rewrite~\eqref{eq:SU3:tf} as a (singular) ODE system for $\hat{f}_1$, $\hat{f}_2$, $\hat{f}_3$ as defined above in~\eqref{eq:fi:su3:singular:ic}.
It will also be useful to record the ODE system satisfied by the sums and differences of pairs 
of the components~$f_i$. A straightforward calculation shows that
\[
(f_j + f_k)' = \left(\frac{f_i^2 - (f_j-f_k)^2}{2f_1 f_2 f_3}\right) (f_j+f_k), 
\qquad
(f_j-f_k)' = \left( \frac{f_i^2-(f_j+f_k)^2}{2f_1f_2f_3}\right) (f_j-f_k),
\]
where as above $(ijk)$ is any cyclic permutation of  $(123)$.
In particular, if on some principal orbit we have $f_j=f_k$ then we will have $f_j \equiv f_k$ up to the maximal 
existence time of the solution. 

For later purposes we record some intermediate calculations used to calculate the recast ODE system:
\begin{align*}
f_2^2+f_3^2 - f_1^2 &= 2b^2 + 2b t (\hat{f}_2+\hat{f}_3) - t^2 (1 -\hat{f}_2^2 - \hat{f}_3^2 + 2t^2 \hat{f}_1 + t^4 \hat{f}_1^2),\\
f_1^2+f_3^2-f_2^2 &= \phantom{ 2b^2 +\ } 2b t(\hat{f}_3 - \hat{f}_2) + t^2( 1 + \hat{f}_3^2 - \hat{f}_2^2 + 2t^2 \hat{f}_1 + t^4 \hat{f}_1^2)
\end{align*}
while for $f_1^2+f_2^2-f_3^2$ we simply exchange $\hat{f}_2$ and $\hat{f}_3$ in the previous expression. For the mixed quadratic terms we find
\begin{align*}
2f_2f_3 &= 2\left(b^2 + t(\hat{f}_2+\hat{f_3}) + t^2 \hat{f}_2 \hat{f}_3\right), \\
2f_1 f_3 &= 2t(b + t\hat{f}_3)(1+t^2 \hat{f}_1), \\
2f_1 f_2 &= 2t(b + t\hat{f}_2)(1+t^2 \hat{f}_1).
\end{align*}
For $i=1$ rewriting~\eqref{eq:SU3:tf} yields
\begin{align*}
 t^3 \hat{f}_1' +1 + 3t^2 \hat{f}_1  &= \frac{ 2b^2 + 2b t (\hat{f}_2+\hat{f}_3) - t^2 (1 -\hat{f}_2^2 - \hat{f}_3^2 + 2t^2 \hat{f}_1 + t^4 \hat{f}_1^2)}{2\left(b^2 + t(\hat{f}_2+\hat{f_3}) + t^2 \hat{f}_2 \hat{f}_3\right)} \\ 
 &= 1 - t^2 \left( \frac{1-(\hat{f}_2-\hat{f}_3)^2 + 2t^2 \hat{f}_1 + t^4 \hat{f}_1^2)}{2\left(b^2 + t(\hat{f}_2+\hat{f_3}) + t^2 \hat{f}_2 \hat{f}_3\right)}\right).
\end{align*}
Hence we have
\begin{subequations}
\label{eq:fihat:tf}
\begin{gather}
\label{eq:f1hat:tf}
t \hat{f}_1' = -3 \hat{f}_1 + \frac{1}{2b^2} (\hat{f}_2-\hat{f}_3)^2 - \frac{1}{2b^2}
+ t M_1(t,\hat{\bf{f}}).\\
\intertext{Similarly for $i=2$ and $i=3$~\eqref{eq:SU3:tf} is equivalent to}
t\hat{f}_2' = -\hat{f}_2 + \frac{2b (\hat{f}_3 - \hat{f}_2) + t( 1 + \hat{f}_3^2 - \hat{f}_2^2 + 2t^2 \hat{f}_1 + t^4 \hat{f}_1^2)}
{2(b + t\hat{f}_3)(1+t^2 \hat{f}_1)} =  -2\hat{f}_2 + \hat{f}_3 + t M_2(t,\hat{\bf{f}}),\\
\intertext{and}
t\hat{f}_3'  =  \hat{f}_2 -2\hat{f}_3 + t M_3(t,\hat{\bf{f}}),\\
\intertext{respectively. Note also that the sum and difference of $\hat{f}_2$ and $\hat{f}_3$ satisfy the ODEs}
t (\hat{f}_2 + \hat{f}_3)' = - (\hat{f}_2 + \hat{f}_3) + O(t), \qquad t (\hat{f_2}- \hat{f}_3) = -3 (\hat{f}_2 - \hat{f}_3) + O(t).
\end{gather}
\end{subequations}
The ODE system~\eqref{eq:fihat:tf} is now in the form~\eqref{eq:Singular:IVP} and so 
Theorem~\eqref{thm:Singular:IVP} tells us that to have a real analytic solution around $t=0$ first we must impose that the singular term $M_{-1}$
given by
\begin{equation}
\label{eq:Msing:tf}
M_{-1}(\hat{f}_1,\hat{f}_2,\hat{f}_3) = \left( -3 \hat{f}_1 + \frac{1}{2b^2} (\hat{f}_2-\hat{f}_3)^2 - \frac{1}{2b^2}, \,
2\hat{f}_2 + \hat{f}_3 ,\, \hat{f}_2 -2\hat{f}_3  \right) 
\end{equation}
vanishes initially. In our case this means that the initial conditions for $\hat{f}_1$, $\hat{f}_2$ and 
$\hat{f}_3$ at $t=0$ must satisfy
\[
-3 \hat{f}_1 + \frac{(\hat{f}_2-\hat{f}_3)^2 - 1}{2b^2} = 0, \quad   2\hat{f}_2 + \hat{f}_3=0, \quad
\hat{f}_2 -2\hat{f}_3=0.
\]
These equations force the initial conditions to satisfy
\begin{equation}
\label{eq:fi:IC:tf}
\hat{f_1}(0) = -\frac{1}{6b^2}, \qquad \hat{f}_2(0) = \hat{f}_3(0)=0.
\end{equation}

To guarantee a unique formal series solution for all $b \neq 0$ it remains to verify that $n \tu{Id} - dM_{-1}$ evaluated at the initial conditions~\eqref{eq:fi:IC:tf}
is invertible for all positive integers $n$. 
Evaluating the differential of $M_{-1}$ at the initial conditions~\eqref{eq:fi:IC:tf} yields
\[
n\, \tu{Id} - dM_{-1} = 
\left(
\begin{matrix}
n+3 & 0 & 0\\
0 & n+2 & -1\\
0 & -1 & n+2
\end{matrix}
\right),
\]
which having determinant $(n+3)^2(n+1)$ is indeed invertible for all positive integers $n$.
Hence Theorem~\ref{thm:Singular:IVP} applies. 
However we already know that there are solutions of the ODE system with any such initial values that have $f_2 \equiv f_3$, namely 
the classical AC Bryant--Salamon solutions on $\Lambda^2_- \CP^2$
from Example \ref{ex:BS}.
So by local uniqueness these solutions must coincide.
\\[0.2em]
\emph{The general case.}
First we must determine what constraints the initial values of an invariant soliton must satisfy to extend smoothly over 
a singular orbit $\CP^2$ at $t=0$. The conditions that the $3$-form $\varphi$ must satisfy have already been described. 
The coefficient $u$ determining the vector field $X=u \,\partial_t$ must be odd in $t$, 
and in particular $u(0)=0$. We now use these constraints to understand the behaviour of the components 
of the torsion $\tau$ close to the singular orbit.

Since $u$ is odd, using the initial conditions for $f_i$ given in~\eqref{eq:fi:su3:singular:ic}
and the conservation law~\eqref{eq:ODEs:SU3:tausum} we find that $\tb$ must be odd 
and in particular $\tb(0)=0$.
Since $f_1$ is odd and $f_2 f_3$ is even, \eqref{eq:tau:i:SU3} implies that $\tau_1$ is odd
and therefore also $\tau_2+\tau_3$ is odd. 
In particular $\tau_1=0$ and $\tau_2+\tau_3=0$ at $t=0$. 

To proceed further we choose to single out the 5 variables $(f_1,f_2,f_3,\tau_2, \tau_3)$ and to recover $\tau_1$ by defining 
\begin{equation}
\label{eq:tau1}
\tau_1:= - f_1^2\left(\frac{\tau_2}{f_2^2}+\frac{\tau_3}{f_3^2}\right).
\end{equation}
Then we write
\[
\tau_2 = c + t \hat{\tau}_2, \quad \tau_3 = -c + t \hat{\tau}_3,
\]
where $\hat{\tau}_2+\hat{\tau}_3$ is even (since $\tau_2+\tau_3$ was odd).

The key point is to determine the potentially singular terms $M_{-1}$ in the first-order ODE system rewritten in terms of the variables $(\hat{f}_1,\hat{f}_2,\hat{f}_3,\hat{\tau}_2,\hat{\tau}_3)$.
For $\hat{f}_i'$ we can do this by understanding what corrections to the singular terms $M_{-1}$ that we calculated in the torsion-free case in~\eqref{eq:Msing:tf} will appear due to the presence of the torsion terms $\tfrac{\tau_i}{2f_i}$. 

For $i=2,\,3$ this is straightforward: the terms in $M_{-1}$ change only by the addition/subtraction respectively
of the term $\frac{c}{2b}$. In other words we have
\[
t \hat{f}_2' = -2\hat{f}_2 + \hat{f}_3 + \frac{c}{2b} + O(t), \qquad t \hat{f}_3' = \hat{f}_2 -2\hat{f}_3-\frac{c}{2b}+ O(t).
\]
The ODEs satisfied by the sum $\hat{f}_2+\hat{f}_3$ and the 
difference $\hat{f}_2-\hat{f}_3$ are
\begin{align*}
t(\hat{f}_2+\hat{f}_3)' &= - (\hat{f}_2+\hat{f}_3) + O(t), \\
t (\hat{f}_2 -\hat{f}_3)'&= -3(\hat{f}_2-\hat{f}_3) + \frac{c}{b} + O(t).
\end{align*}
The vanishing of these two components of $M_{-1}$ on the initial data therefore forces
\begin{equation}
\label{eq:hatf2:f3:ic}
\hat{f}_2(0) = -\hat{f}_3(0) = \frac{c}{6b}.
\end{equation}
The computation for $i=1$ requires a little more effort and we return to this at the end.

Instead next we consider the equations for $\tau_2$ and $\tau_3$, \ie \eqref{eq:SU3:taui:dot} rewritten
in terms of $\hat{\tau}_2$ and $\hat{\tau}_3$.
For $i=2, 3$ the coefficients $\fb f_j f_k f_i^{-1}$ appearing on the left-hand side of~\eqref{eq:SU3:taui:dot} 
are both equal to 
\[2tb^2 + O(t^2).\]
Hence to determine the terms that will contribute to $M_{-1}$  we need only look at terms 
on the right-hand side of~\eqref{eq:SU3:taui:dot} that give rise to terms linear in $t$. Two of the five terms, namely $-2\lambda f_j^2f_k^2 \tau_i$ and $\tb f_j f_k f_i^{-1}\tau_i$, contain only terms quadratic in $t$ and higher: 
for the first term this is immediate from the fact that it contains $f_1^2$ and for the second recall that $\tb$ is odd in $t$.
So we have to consider the remaining three terms: one not involving the coefficients of $\tau$, one linear and one quadratic in those coefficients. The first term $\tfrac{4}{3}\lambda f_1 f_2 f_3 (2f_i^2-f_j^2-f_k^2)$ is easily seen to contribute
\[
\frac{4\lambda b^4}{3}t.
\]
To see the $t$ contribution made by the second term first note that $\tb = t(\hat{\tau}_2+\hat{\tau}_3) + O(t^2)$. 
This is because $\tb$ and $\tau_1$ are both odd and \eqref{eq:tau1}
combined with the leading-order behaviour of $f_2$, $f_3$, $\tau_2$ and $\tau_3$ force that the leading-order behaviour of $\tau_1$ is $t^3$. (Below we will need to determine this term).
It is therefore clear that the term linear in the coefficients of $\tau$, namely $-2\tb f_i^2$, contributes
\[
-2t b^2 (\hat{\tau}_2+\hat{\tau}_3).
\]
It remains to compute the $t$ contribution made by the term $\frac{1}{3} \fb f_1 f_2 f_3 \abs{\tau}^2$. Using~\eqref{eq:tau1} again, we can replace the term $\tau_1^2 f_1^{-4}$
and obtain 
\[\abs{\tau}^2 = \left(\sum{\tau_l^2 f_l^{-4}} \right) = 2(\tau_2^2f_2^{-4}+\tau_3^2f_3^{-4} + \tau_2\tau_3 f_2^{-2} f_3^{-2}).\]
Using the leading-order behaviour of those coefficients we find that  $\left(\sum{\tau_l^2 f_l^{-4}} \right) = 2c^2 b^{-4} + O(t)$. Therefore the $t$ contribution of this term is
\[
\frac{4}{3}t c^2.
\]
Combining all this we find
\begin{align*}
t \hat{\tau}_2' &= -2 \hat{\tau}_2 - \hat{\tau}_3 + \frac{2}{3}\lambda b^2 + \frac{2c^2}{3b^2} + O(t), \\
t \hat{\tau}_3' &= -\hat{\tau}_2 - 2\hat{\tau}_3 + \frac{2}{3}\lambda b^2 + \frac{2c^2}{3b^2} + O(t).
\end{align*}
We also note that the ODEs satisfied by the sum $\hat{\tau}_2+\hat{\tau}_3$ and the 
difference $\hat{\tau}_2-\hat{\tau}_3$ are
\begin{align*}
t(\hat{\tau}_2+\hat{\tau}_3)' &= -3 (\hat{\tau}_2+\hat{\tau}_3) + \frac{4}{3}\left(\lambda b^2 + \frac{c^2}{b^2} \right) + O(t), \\ t (\hat{\tau}_2 -\hat{\tau}_3)' &= -(\hat{\tau}_2-\hat{\tau}_3) + O(t).
\end{align*}
Requiring these two components of $M_{-1}$ to vanish on the initial data therefore forces
\begin{equation}
\label{eq:hat:tau2:3:ic}
\hat{\tau}_2(0)=\hat{\tau}_3(0) = \frac{2}{9} \left(\lambda b^2 + \frac{c^2}{b^2}\right).
\end{equation}

Finally we return to the change in the contribution to~\eqref{eq:f1hat:tf},  the equation for $\hat{f}_1'$ in the torsion-free case, due to the presence 
of the additional torsion term $\tfrac{\tau_1}{2f_1}$. To evaluate this we need to find the term linear in $t$ 
in the expansion of $(\tau_2 f_2^{-2} + \tau_3 f_3^{-2})$. A short calculation shows
\begin{equation}
\label{eqn:tau1:su3:expand}
\frac{\tau_1}{f_1^2}:= -
\left(\frac{\tau_2}{f_2^2}+ \frac{\tau_3}{f_3^2}\right) = -\left( \frac{2c}{b^3} (\hat{f}_3 - \hat{f}_2) +  \frac{1}{b^2} (\hat{\tau}_2+\hat{\tau}_3)\right) t + O(t^2).
\end{equation}
Substituting for the initial values given in \eqref{eq:hatf2:f3:ic} and \eqref{eq:hat:tau2:3:ic} we find 
\[
\frac{\tau_1}{f_1^2} = \frac{2}{9} \left( -2\lambda + \frac{c^2}{b^4} \right) t + O(t^2).
\]
Combining~\eqref{eq:f1hat:tf}~and~\eqref{eqn:tau1:su3:expand} we conclude that
\[
t \hat{f}_1' = \left(-3 \hat{f}_1 + \frac{1}{2b^2} (\hat{f}_2-\hat{f}_3)^2 - \frac{1}{2b^2}\right) + 
\frac{c}{b^3} (\hat{f}_2-\hat{f}_3) - \frac{1}{2b^2} (\hat{\tau}_2+\hat{\tau}_3) + O(t).
\]
Hence the constraint that the corresponding component of $M_{-1}$ evaluated at the initial condition must vanish also yields
\begin{equation}
\label{eq:hatf1:ic}
\hat{f}_1(0) = -\frac{1}{6b^2} - \frac{2\lambda}{27} + \frac{c^2}{18 b^4},
\end{equation}
where to obtain this we also used the previously determined constraints on the initial conditions given by \eqref{eq:hatf2:f3:ic} and \eqref{eq:hat:tau2:3:ic}.
Evaluating the differential of $M_{-1}$ at the initial conditions determined by~\eqref{eq:hatf2:f3:ic} at~\eqref{eq:hatf1:ic} yields
\[
n\, \tu{Id} - dM_{-1} = 
\left(
\begin{matrix}
n+3 & -\frac{4c}{3b^3} & \frac{4c}{3b^3} & \frac{1}{2b^2} & \frac{1}{2b^2}\\
0 & n+2 & -1 & 0 & 0\\
0 & -1 & n+2 & 0 & 0\\
0 & 0 & 0 & n+2 & 1\\
0 & 0 & 0 & 1 & n+2
\end{matrix}
\right),
\]
which having determinant $(n+3)^3(n+1)^2$ is indeed invertible for all positive integers $n$.
Hence Theorem~\eqref{thm:Singular:IVP} applies and our result follows.
\end{proof}

\subsection{\texorpdfSp{2}-invariant \texorpdfgtwo-solitons extending smoothly over \texorpdfstring{$\Sph^4$}{S\textasciicircum 4}}
\label{ss:sp2:sc}

We can also deduce results about \Sp{2}-invariant \gtwo-solitons that close smoothly on the singular orbit $\Sph^4 \subset \Lambda^2_-\Sph^4$ by 
specialising to the case where $f_2=f_3$ and $\tau_2=\tau_3$. We now give further details of the results obtained this way.

For a closed \Sp{2}-invariant \gtstr~$\varphi_f$ expressed in the form
\[
\varphi_f = \omega_f \wedge dt + \Real{\Omega_f} = (f_1^2 \omega_1 + f_2^2 \omega_2) \wedge dt +  f_1 f_2^2 \alpha
\]
the conditions that $\varphi_f$ extend smoothly over the singular orbit $\Sph^4 = \Sp{2}/\Sp{1}\times \Sp{1}$ at $t=0$ are:
$f_1$ is odd in $t$ and $\abs{f_1'(0)}=1$; 
$f_2$ is even in $t$. 

Recall also that by acting with Klein four-group we may assume that $(f_1,f_2)$ is a positive pair for all $t>0$. 
In particular, we must have that $f_1'(0)=+1$ and $f_2(0)>0$.
Hence we can write the coefficients $f_1$ and $f_2$ as
\begin{equation}
\label{eq:fi:sp2:singular:ic}
f_1 = t + t^3 \hat{f}_1, \quad f_2 = b+ t^2 \hat{f}_2, 
\end{equation}
for some $b>0$ and 
new functions $\hat{f}_1$, $\hat{f}_2$ that are both even in $t$. 
Similarly, we write the torsion coefficients $\tau_1$ and $\tau_2$ as 
\[
\tau_1 = t^3 \hat{\tau}_i, \quad \tau_2 = t \hat{\tau}_2,
\]
for functions $\hat{\tau}_1$ and $\hat{\tau}_2$ that are both even in $t$.

The main theorem for smoothly-closing solitons in the \Sp{2}-invariant setting
can be obtained by setting $c = 0$ in Theorem \ref{thm:SU3:smooth:closure}.

\begin{theorem}
\label{thm:Sp2:smooth:closure}
Fix any $\lambda \in \R$ and $b>0$. Then there exists a unique local \Sp{2}-invariant \gtwo-soliton which closes smoothly on 
the singular orbit $\Sph^4 \subset \Lambda^2_-\Sph^4$ and which satisfies
\[
f_1= t + t^3 \hat{f}_1, \quad f_2 = b + t^2 \hat{f}_2, \quad \tau_1 = t^3 \hat{\tau}_1, \quad \tau_2 = t \hat{\tau}_2,
\]
where the terms $\hat{f}_i$ and $\hat{\tau}_i$ satisfy
\[
\hat{f}_1(0) = -\frac{1}{6b^2} - \frac{2\lambda}{27},
\quad \hat{f}_2(0) = \frac{\left( 2\lambda b^2+9 \right)} {36b}, \quad
\hat{\tau}_1(0) = -\frac{4\lambda}{9}, \quad 
\hat{\tau}_2(0) = \frac{2}{9} \lambda b^2.
\]
 The solution is real analytic on $[0,\epsilon) \times \CP^3$ for some $\epsilon >0$ (depending on $b$ and $\lambda$).
Moreover, up to the $\Z_2 \times \Z_2$-action, any \Sp{2}-invariant \gtwo-soliton which closes smoothly on the singular orbit $\Sph^4 \subset \Lambda^2_-\Sph^4$ belongs to this $1$-parameter family of solutions.
\end{theorem}
Taking into account the scaling behaviour of steady solitons immediately gives us the following:

\begin{corollary}
\label{Cor:Sp2:steady}
Any local \Sp{2}-invariant steady~\gtwo-soliton which closes smoothly on the singular orbit $\Sph^4 \subset \Lambda^2_-\Sph^4$ is torsion-free and has vector field $X \equiv 0$. 
Hence the underlying~\gtstr~is the (unique up to scale) standard asymptotically conical torsion-free
~\gtstr~on $\Lambda^2_-\Sph^4$ constructed by Bryant--Salamon. In particular, its asymptotic cone is the cone over the unique \Sp{2}-invariant 
nearly K\"ahler structure on $\CP^3$.
\end{corollary}

\subsection{Explicit complete asymptotically conical shrinkers}
\label{ss:shrinker:explicit}
In this subsection we prove Theorem~\ref{mthm:shrinkers}, observing that
properties of the power series expansions of \Sp{2}-invariant smoothly-closing
Laplacian solitons close to the singular orbit
lead naturally to a $1$-parameter family of explicit complete asymptotically conical shrinking solitons. Up to the action of scaling this yields a unique explicit asymptotically conical shrinker on $\Lambda^2_-\Sph^4$ and also on $\Lambda^2_-\CP^2$. For notational simplicity it will be convenient to set $x=f_1$ and $y=f_2=f_3$ in the discussion below.

The starting point is to notice the following property of the power series expansions 
close to the singular orbit for the smoothly-closing shrinkers constructed in Theorem~\ref{thm:Sp2:smooth:closure} and described in 
Appendix~\ref{app:sc:power:series}: when we write
\[
x = t \left(1+ \sum_{i=0}{x_i t^{2i}}\right), \quad  \tau_2 = t\left(\sum_{i=0}{T_i t^{2i}}\right)
\]
then the coefficients $x_1$, $x_2$, $T_1$ and $T_2$ all contain $x_0$ as a factor. Moreover, $x_0$ vanishes if and only if  $\lambda b^2 = -\tfrac{9}{4}$. 
In other words, when the latter equality holds we have 
\[
x=t + O(t^9), \quad \tau_2 = -\tfrac{1}{2}t + O(t^9).
\] 
This suggests  that $x+2\tau_2 \equiv 0$ and $x=t$ when $\lambda b^2 = -\tfrac{9}{4}$.

We will see that this is indeed the case and then find explicit solutions also for $y$ and $u$. 
First recall from \eqref{eq:ODE:tau2'} that the ODE for $x^2$ can be written as
$
(x^2)' = 2x - \frac{x^2}{y^2}(x+2\tau_2).
$
So if we impose the condition $x+2\tau_2\equiv 0$ then we are forced to have
$x'=1$ everywhere and hence $x = t$ by
the initial conditions. Then for $y$ we have
$
(y^2)' = x+ \tau_2 = \frac{1}{2}x = \frac{1}{2}t, 
$
and therefore $y^2 = b^2 + \frac{1}{4}t^2$.
The quantity $S= y^2 - x^2 - \tfrac{3}{2}x \tau_2$ also has a particularly simple behaviour:
$S = b^2 + \tfrac{1}{4}t^2 - t^2 + \tfrac{3}{4}t^2  \equiv b^2.$
Now we compute that $4R_1 S = 4(\lambda xy^2 - 3\tau_2)S = -\frac{3t}{4}(4b^2+3t^2)$ and $3x(x^2+2y^2) = \frac{3t}{2}(4b^2+3t^2)$
and therefore the right-hand side of the ODE for $\tau_2'$ is identically equal to $-\frac{1}{2}$. 
Hence we do have a consistent solution to the ODE system \eqref{eq:ODE:tau2'}.
Finally using the second equation in~\eqref{eq:tau1:u:tau2} we calculate that 
\[
u = \frac{3t}{4b^2} + \frac{4t}{4b^2+t^2}.
\]

In other words,  the unique solution of the initial value problem for the \Sp{2}-invariant shrinker equation with 
initial conditions as in the statement of Theorem~\ref{thm:Sp2:smooth:closure} and  $\lambda = - \frac{9}{4b^2}<0$ is
\[
x = t, \quad y^2 = b^2 + \tfrac{1}{4}t^2, \quad \tau_2 = -\tfrac{1}{2}t, \quad u = \frac{3t}{4b^2} + \frac{4t}{4b^2+t^2}.
\]
Note that $u + \tfrac{\lambda}{3}t = \tfrac{4t}{4b^2+t^2} \ge 0 $ and that the latter tends to zero as $t \to \infty$.
For $t \gg 1$ this shrinker satisfies $y \simeq \frac{1}{2}t$, so that it is asymptotic to the closed non-torsion-free \Sp{2}-invariant \gtwo--cone 
$x=c_1 t,\, y=c_2 t$ with $c_1=1$, $c_2=\frac{1}{2}$. 

\begin{remark}
\label{rmk:x:eq:y}
Note also that in this case asymptotically we have $x>y$, whereas initially $y>x$, and moreover
there is a unique time $t_0 = \sqrt{\frac{4b^2}{3}}$ at which $x=y$.
\end{remark}

We summarise the situation in the following more precise version of
Theorem \ref{mthm:shrinkers}.
 
\begin{theorem}
\label{thm:explicit:AC:shrinker}
For any $b>0$ let $\lambda = - \frac{9}{4b^2}<0$ and define the following functions 
\[
x = t, \quad y^2 = b^2 + \tfrac{1}{4}t^2, \quad \tau_2 = -\tfrac{1}{2}t, \quad u = \frac{3t}{4b^2} + \frac{4t}{4b^2+t^2}, \quad\  t \ge 0.
\]
\begin{enumerate}[left=0em]
\item
The closed \Sp{2}-invariant \gtstr~$\varphi  =  xy^2 \alpha  + (x^2 \omega_1 + y^2 \omega_2)\wedge dt$, together with the vector field $X=u \,\partial_t$
determines a complete asymptotically conical Laplacian shrinker on $\Lambda^2_-\Sph^4$ asymptotic to the 
closed but non-torsion-free \Sp{2}-invariant \gtwo--cone defined by $x=t$, $y= \frac{1}{2}t$.
\item
The closed $\sunitary{3}$-invariant \gtstr~$\varphi  =  xy^2 \alpha  + (x^2 \omega_1 + y^2 \omega_2 + y^2 \omega_3)\wedge dt$
together with the vector field $X=u \,\partial_t$ determines
a complete asymptotically conical Laplacian shrinker on $\Lambda^2_-\CP^2$ 
that closes smoothly on $\CP^2_1$ and is asymptotic to the closed but non-torsion-free $\sunitary{3}$-invariant 
\gtwo--cone defined by $f_1=t, f_2=f_3 = \frac{1}{2}t$. 
\end{enumerate}
\end{theorem}
\begin{remark}
\label{rmk:SU3:shrinker:variants}
By Remark \ref{rmk:involution_asd} the metric induced by the AC shrinking $\sunitary{3}$-invariant soliton on~$\Lambda^2_-\CP^2$ is invariant under the $\Z_2$-action
given by multiplying the fibres by $-1$ (which acts on the
\mbox{\gtstr{}} as multiplication by $-1$).
By acting with the nontrivial elements of ${A_3 \subset W = S_3}$
we also get two other variants of this AC shrinking soliton; 
these variants have different free~$\Z_2$ actions, different asymptotic
cones and close smoothly not on the singular orbit $\CP^2_1$ but
on ${\CP^2_2 = \sunitary{3}/\unitary{2}_2}$
or~$\CP^2_3 = \sunitary{3}/\unitary{2}_3$ instead.
\end{remark}

\begin{remark*}
Fowdar~\cite[Section 5.2]{Fowdar:S1:invariant:LF}  found explicit
complete (usually inhomogeneous) gradient shrinking Laplacian solitons on generalised cylinders $\R \times P^6$ where $P^6$ 
is the total space of certain
$T^2$-bundles over~\hk~4-manifolds $B^4$. 
These examples arise from a $1$-parameter family of special~\sunitary{3}-structures $(\omega_t,\Omega_t)$ on $P^6$ as described in~\S\ref{ss:G2:evolve:closed}
and also come from solving a fixed ODE system that is independent of the choice of $B$. 
As $t$ varies along the cylinder the induced metric $g_t$ on the base~$B$
and the fibre $T^2$ both change by homotheties, but the base and the fibre scale differently. 
In the simplest case where $B=T^4$, $P^6$ is a compact nilmanifold $N^6$, the complex $3$-dimensional Iwasawa manifold, \ie the compact quotient 
of the complex $3$-dimensional Heisenberg group by the lattice generated by the Gaussian integers. 

Apostolov--Salamon~\cite[Example 1 p55]{Apostolov:Salamon} and Gibbons et al \cite{CLPS:domain:walls} 
explain how the analogous construction of an incomplete
torsion-free~\gtstr~on $(b,\infty) \times N^6$, where $N^6$ is the Iwasawa manifold, 
arises from the Bryant--Salamon \Sp{2}-invariant complete AC torsion-free~\gtstr{} on~$\Lambda^2_-\Sph^4$ by a `contraction' of its isometry group~\Sp{2}.
In a similar spirit, the functions describing Fowdar's shrinker can also
be viewed as defining an approximate end solution to the
\Sp{2}-invariant soliton system \eqref{eq:ODE:tau2'} on $(b,\infty) \times \CP^3$. 
This approximate end solution turns out to serve as a model for a genuine forward-complete
shrinking solution of \eqref{eq:ODE:tau2'} that is not asymptotically conical \cite[Proposition 8.6]{HJN}. 
\end{remark*}

\subsection{Finite extinction solutions}
\label{ss:finite:extinction}
In order to understand which of the smoothly-closing solitons constructed in Sections~\ref{ss:su3:sc} and~\ref{ss:sp2:sc}
give rise to complete solitons we need to be able to determine which of them exist on an infinite $t$-interval 
and to be able to distinguish these from the solutions with finite extinction time. 
Proposition \ref{prop:steady_extinct} below suggests the existence of steady
soliton solutions with finite future lifetime, such that close to the
extinction time $t_*$ we have the following asymptotics
\[ f_1, f_2 = O\left((t_*-t)^{-\frac14}\right),
\quad f_3 = O\left(\sqrt{t_*-t}\right) . \]
Motivated by these asymptotics 
we find that we can make an ansatz in terms of $v := \sqrt{t_*-t}$ that solves
the ODE system~\eqref{eq:SU3:order1} to leading order. Moreover, the terms that
involve $\lambda$ are actually of lower order, so the ansatz works equally well
for non-steady solitons.
We can then go ahead and solve by power series, and
for each fixed $\lambda$ we find a 4-parameter family of finite-extinction time solutions (so their
flow lines fill an open region in the phase space).
We include the result in this section because (a) the method of proof 
is similar to that used in the construction of the smoothly-closing solitons
and (b) like the construction of smoothly-closing solitons the construction is insensitive to the value (or sign) of the dilation constant. 

One of the free parameters appears in terms of higher order than the other
three parameters.
That there is no obstruction to solving for the
last free parameter is
related to the ansatz forcing $\frac{\tau_1}{f_1^2}$ and
$\frac{\tau_2}{f_2^2}$, which both have leading term
$\frac{1}{2(t_*-t)} = \frac{1}{2v^2}$, to have difference $O(v)$.
This fact is due to the derivative of
$\frac{\tau_1}{f_1^2} - \frac{\tau_2}{f_2^2}$ (expanded from
\eqref{eq:SU3:order1}) being a sum of terms containing a factor of
$\frac{\tau_1}{f_1^2} - \frac{\tau_2}{f_2^2}$ itself and lower
order terms.
Incorporating this fact into the ansatz avoids the need to
iteratively identify several terms in the
power series in order to prove that there is no obstruction.
 
\begin{theorem}
\label{thm:fwd:incomplete}
Given any $a, b, c > 0$ and $m\not = 0$, there is a unique solution
to \eqref{eq:SU3:order1} of the form
\begin{equation}
\label{eq:extinction_ansatz}
\begin{aligned}
f_1^2 &= \frac{2a^2}{v} (1 + v g_1), \quad
f_2^2 = \frac{2b^2}{v} (1 + v g_2), \quad
f_3^2 = c^2 v^2(1+v g_3), \\
\tau_1 &= \frac{f_1^2}{2v^2}(1 + v h + v^3 \delta), \quad
\tau_2 = \frac{f_2^2}{2v^2}(1 + v h - v^3 \delta)
\end{aligned}
\end{equation}
where $g_i, h, \delta : [0,\epsilon) \to \R$ are real analytic,
and $\delta(0) = m$.

If $a = b$ and $m = 0$ then $f_1 = f_2$ and $\tau_1 = \tau_2$.
\end{theorem}

\begin{proof}
Taking the ansatz \eqref{eq:extinction_ansatz}, we can rewrite
\eqref{eq:SU3:order1} as an ODE for $x = (g_1, g_2, g_3, h, \delta)$.
It can be written as
\[ x' = \frac{x_0 + Lx}{v} + H(x,v) \]
for an analytic function $H$, constant $x_0$ and linear map $L$ that
a computer-aided calculation shows to be represented by
\[
\setstretch{1.2}
x_0 = \frac{1}{abc}\left(\begin{array}{c}
2 a^2-2 b^2 \\
-2 a^2+2 b^2 \\
-2 a^2-2 b^2 \\
2 a^2+2 b^2 \\
0 %
\end{array}\right) \quad
\textrm{and} \quad
L = \left(\begin{array}{ccccc}
-1 & 0 & 0 & -1 & 0
\\
 0 & -1 & 0 & -1 & 0 
\\
0 & 0 & -1 & 2 & 0 
\\
0 & 0 & 0 & -3 & 0 
\\
0 & 0 & 0 & 0 & 0
\end{array}\right) . \]
Once one fixes $\delta(0) = m$, there is clearly a unique solution for the
other initial terms $g_i(0)$ and~$h(0)$.
$L$ is diagonalisable with eigenvalues $-1$ (of multiplicity 3,
``caused'' by the fact that changing the parameters $a$, $b$, and $c$
still gives a valid ansatz), 0 and $-3$
(the latter with eigenvector $(1,1,-2,2,0)$). 
Because there are no positive eigenvalues there is a unique formal power
series solution for each $m$, which corresponds to a unique,
analytic solution.

\pagebreak[2]
If $a = b$ then setting $g_1 = g_2$ and $\delta = 0$ defines an ODE in three
variables that has a unique solution for the same reasons, so the solutions
with $\delta(0) = 0$ stay in that subset.
\end{proof}
Theorem \ref{mthm:extinction} is a less detailed version of the previous theorem.

\section{Scale decoupling and AC steady ends}
\label{sec:scale_eqs}
In this section we introduce a scale-normalised version of the first-order soliton ODE system
that is particularly convenient for analysing the  behaviour of asymptotically conical (AC) ~\sunitary{3}-invariant Laplacian solitons.
In the steady case this scale-normalised system exhibits a special property: a decoupling between the scale 
and the remaining scale-normalised quantities. 
We use this system in our proof of Theorem~\ref{mthm:steady_ac_end} in \S\ref{ss:AC:ends}: in particular, we prove
that the property that an~\sunitary{3}-invariant steady soliton has an AC
end is an open condition and moreover that any~\sunitary{3}-invariant steady soliton for which all the ratios $f_i/f_j$ remain
bounded is necessarily AC, generically with rate $-1$, with asymptotic cone the unique invariant torsion-free cone.

\subsection{A scale-normalised version of the first-order soliton system}
\label{ss:si:ODEs}
To analyse the behaviour of asymptotically conical ~\sunitary{3}-invariant Laplacian solitons 
it will prove useful to consider the equations satisfied by the following set of variables:
\begin{equation}
\label{eq:g:si:var}
g^3:= f_1 f_2 f_3, \qquad \sif_i: = \frac{f_i}{g}, \qquad \sit_i:= \frac{\tau_i}{g} \qquad i=1,2,3.
\end{equation}
Note that the geometric mean $g$ of the positive triple $f=(f_1,f_2,f_3)$ has scaling weight one under dilations of $f$ whereas all the variables $\sif_i$ and $\sit_i$ are scale invariant. Notice also that by definition the $\sif_i$ satisfy $\sif_1 \sif_2 \sif_3=1$.
It is also convenient to introduce scale-invariant versions of the quantities $S_i$ introduced in~\eqref{eq:Si:def}. 
More specifically we define
\begin{equation}
\label{eq:sis:def}
\sis_i:= \frac{S_i}{g^2} = (2\sif^2_i-\sif^2_j-\sif^2_k) -\frac{3\sit_i }{2\sif_i^2} =  \sie_i - \frac{3\sit_i}{2\sif_i^2},
\end{equation}
where we define the \emph{excesses} $E_i$  and their scale-invariant analogues $\sie_i$ by
\begin{subequations}
\begin{align}
E_i:&= 3 f_i^2 - \fb,\\
\label{eq:sie:def}	
\sie_i:&= \frac{E_i}{g^2} = 3\sif_i^2 - \overline{\sif^2}.
\end{align}
\end{subequations}
The excess $\sie_i$ gives a scale-invariant measure of the deviation (or excess) of $f_i^2$ from the average $\tfrac{1}{3}\sum_j{f_j^2}$. 
We note some elementary properties of various sums related to these excesses:
\[
\sum_i{E_i}=0, \quad \sum_i{E_i^2}=3D(f),  \quad \sum_i{E_i f_i^2}=D(f),
\]
where $D(f):= \sum_{i<j}{(f_i^2-f_j^2)^2}$ as in \eqref{eq:D}.
Obvious analogues also hold for sums of their scale-invariant versions.
The scale-invariant version of the final equality also yields  
\begin{equation}
\label{eq:sif2:sis}
\sum{\sif_i^2 \sis_i} = D(\sif) - \tfrac{3}{2} \overline{\sit}.
\end{equation}
Since 
\begin{equation}
\label{eq:sif}
3\sif_i^2  = \sie_i + \overline{\sif^2}>0
\end{equation} 
and $\sif_1 \sif_2 \sif_3=1$ we also have 
\begin{equation}
\label{eq:cubic:sie}
\prod_{i=1}^3 (\sie_i + \overline{\sif^2}) = 27.
\end{equation}
Rearranging~\eqref{eq:sis:def} gives
\[
\sit_i = \frac{2}{3} \sif_i^2 ( 3\sif_i^2 - \overline{\sif^2} - \sis_i) = \frac{2}{3} \sif_i^2 (\sie_i - \sis_i).
\]
Hence we can also express $\abs{\tau}^2$ in terms of $g$ and scale-invariant variables as
\begin{equation}
\label{eq:abs:tau:si}
\abs{\tau}^2 = \sum{\frac{\tau_i^2}{f_i^4}} = \frac{1}{g^2} \sum{\left(\tfrac{\sit_i}{\sif_i^2}\right)^2} = \frac{4}{9g^2} \sum{(\sie_i - \sis_i)^2} = \frac{4}{9g^2}\left(3D(\sif) + \sum{(\sis_i^2 - 2 \sie_i \sis_i)} \right).
\end{equation} 
Similarly using the conservation law~\eqref{eq:SU3:conserve} yields the following expression for $u$ in terms of $g$ and scale-invariant variables
\begin{equation}
\label{eq:u:si:g}
u\, \overline{\sif^2} = -2 \lambda g + \overline{\sit}g^{-1}.
\end{equation}

A direct calculation using the previous first-order system~\eqref{eq:SU3:order1} shows that if $(f,\tau)$ is a solution of the first-order 
version of the invariant Laplacian soliton system then the ODE system satisfied by the variables $(g,\sif_i,\sit_i)$ defined in~\eqref{eq:g:si:var} is
\begin{subequations}
\label{eq:SU3:normal}
\begin{align}
\label{eq:g:dot}
g' &= \frac{1}{6}\overline{\sif^2},\\
\label{eq:sif:dot}
g \sif_i' & = \frac{1}{2\sif_i}\left(\sit_i  - 2\sif_i^4 + \frac{2}{3} \overline{\sif^2} \sif^2_i \right) = -\frac{1}{3} \sif_i \sis_i,\\
\label{eq:sit:dot}
g \sit_i' & = \frac{4\lambda g^2}{3}\frac{\sif_i^2\sis_i}{ \overline{\sif^2}} + \frac{\overline{\sit}}{\overline{\sif^2}} (\sit_i - 2\sif_i^4) +  \frac{1}{3}\left(\sum{\frac{\sit_i^2}{\sif_i^4}}\right)\sif^2_i - \frac{1}{6} \sit_i \overline{\sif^2}.
\end{align}
\end{subequations}
In particular~\eqref{eq:sif:dot} implies that $\sif_i$ is monotone in $t$ if and only if $\sis_i$ has a definite sign (since $\sum{\sis_i}=0$ the $\sis_i$ cannot all have the same sign).

\begin{remark}
\label{rmk:siti:dot:sis0}
Using~\eqref{eq:sis:def} and~\eqref{eq:sif2:sis} the right-hand side of~\eqref{eq:sit:dot} can be rewritten as
\[
\frac{4\sif_i^2 \sis_i}{9\overline{\sif^2}} \left( 3 \lambda g^2 - D(\sif) + (\sum{\sif_i^2 \sis_i}) \right) 
+ \frac{4}{9} \sif_i^2 \sum{\sif_i^2 \sis_i}
+ \frac{4}{9} \left(\sum{(\sis_i^2 - 2 \sie_i \sis_i)}\right) - \frac{1}{6} \sit_i \overline{\sif^2}.
\]
Note that when all $\sis_i$ vanish the only (potentially) nonvanishing term in the expression above is~$ -\frac{1}{6} \sit_i \overline{\sif^2}$.
\end{remark}

\begin{remark}
\label{rmk:si:other}
The evolution of other scale-invariant variables can be computed from~\eqref{eq:SU3:normal}. For instance we have
\begin{equation}
\label{eq:si:fb}
g(\overline{\sif^2})' = -\frac{2}{3} \sum_i{\sif_i^2 \sis_i} = \overline{\sit} - \frac{2}{3} D(\sif),
\end{equation}
\begin{equation}
\label{eq:si:tb}
g \overline{\sit}' = \frac{2\lambda g^2}{3\overline{\sif^2}} \left( 2 D(\sif) - 3 \overline{\sit} \right) + \frac{\overline{\sit}}{\overline{\sif^2}} \left( \overline{\sit} - 2 \overline{\sif^4} - \frac{1}{6} (\overline{\sif^2})^2 \right) + \frac{1}{3} \left(\sum{\frac{\sit_i^2}{\sif_i^4}} \right) \overline{\sif^2},
\end{equation}
\begin{equation}
\label{eq:si:excess}
g \sie_i'  = -2 \sif_i^2 \sis_i - \overline{\sit} + \frac{2}{3} D(\sif) = (3 \sit_i - \overline{\sit}) - \frac{2}{3} \left( 3\sif_i^2 \sie_i - (\sum{\sie_j\sif_j^2}) \right).
\end{equation} 
\end{remark}

\subsubsection*{Characterisations of the Gaussian solitons}Using the scale-normalised version of the Laplacian soliton ODE 
system~\eqref{eq:SU3:normal} we now prove two simple results that provide characterisations of the 
Gaussian solitons over the unique torsion-free $\sunitary{3}$-invariant~\gtstr~on the cone over $\sunitary{3}/\T^2$. The first result says
that these Gaussian solitons are the only conical $\sunitary{3}$-invariant Laplacian solitons; the second says that 
they are the only nontrivial invariant solitons for which the underlying~\gtstr~is torsion-free. 
\begin{lemma}
\label{lem:conical:implies:tf}
For any value of $\lambda \in \R$ let $(\varphi_f,X,\lambda) $ be an $\sunitary{3}$-invariant closed Laplacian soliton for which the induced metric $g_f$ is conical, then $\varphi_f$ coincides with a portion of the 
Gaussian soliton over the unique torsion-free $\sunitary{3}$-invariant~\gtstr~on the cone over $\sunitary{3}/\T^2$, \ie $f_1 = f_2 = f_3 = \tfrac{1}{2}t$ and $X= -\tfrac{1}{3} \lambda t \,\partial_t$. 
\end{lemma}

\begin{proof}
Since $g_f$ is conical the scale-invariant variables $\sif_i$ must all be constant and therefore $\sis_i=0$ for all $i$
by~\eqref{eq:sif:dot}. Similarly the scale-invariant variables $\sit_i$ must also all be constant and therefore  $\dot{\sit}_i=0$ 
for all $i$. Since $\sis_i=0$ for all $i$ then by Remark~\ref{rmk:siti:dot:sis0} the vanishing of the right-hand side of~\eqref{eq:sit:dot}
implies that $\sit_i=0$ for all $i$. Therefore 
$\overline{\sit}$ vanishes and hence by~\eqref{eq:sif2:sis} so does $D(\sif)$ and this implies that $\sif_i=1$ for all $i$.
Therefore by~\eqref{eq:g:dot} $g'= \tfrac{1}{2}$ and hence (up to a time-translation) $f_i=\tfrac{1}{2}t$ for all $i$. 
\end{proof}

Another easy consequence of the equations derived in Remark~\ref{rmk:si:other} is the following
\begin{lemma}
\label{lem:torsion:free:solitons}
An \sunitary{3}-invariant Laplacian soliton $(\varphi,X=u\, \partial_t,\lambda)$ whose underlying closed~\gtstr~$\varphi$ is torsion free is either a trivial soliton, \ie $\lambda=0$, $u \equiv 0$, 
or a Gaussian soliton over the torsion-free cone over $\flag$ for any $\lambda \in \R$.
\end{lemma}
\begin{proof}
If $\lambda=0$ and the torsion $2$-form $\tau$ vanishes, then~\eqref{eq:u:si:g} implies that $u \equiv 0$. 
Now suppose that $\lambda \neq 0$ and the torsion $2$-form $\tau$ vanishes.  Then the only potentially nonvanishing term on the right-hand side of~\eqref{eq:si:tb}
is 
\[
\frac{4\lambda g^2 D(\sif)}{3\overline{\sif^2}}.
\]
Since we must have $\overline{\sit}' \equiv 0$ this implies that $D(\sif)\equiv 0$ and hence that $\sif_i \equiv 1$ for all $i$. 
\end{proof}

\subsection{The purely scale-normalised steady ODE system} 
\label{ss:si:steady}

In this subsection we observe that in the steady case,  after a suitable reparametrisation of $t$, the equations for $\sif_i'$ and $\sit_i'$ can be rewritten entirely in terms of the scale-invariant variables $\sif_i$ and $\sit_i$, 
\ie with no appearance of the scale-factor~$g$. More specifically, using the fact 
that when $\lambda=0$ the terms on the right-hand side of the ODE system~\eqref{eq:SU3:normal} involving the scale-factor $g$ vanish, the ODE system reduces to
\begin{subequations}
\label{eq:steady:normal}
\begin{align}
\label{eq:sifi:dot}
\dot{\sif}_i & = \frac{1}{2\sif_i}\left(\sit_i  - 2\sif_i^4 + \frac{2}{3} \overline{\sif^2} \sif^2_i\right)=-\frac{1}{3} \sif_i \sis_i,\\
\label{eq:siti:dot}
\dot{\sit}_i & =  \frac{\overline{\sit}}{\overline{\sif^2}} (\sit_i - 2\sif_i^4) +  \frac{1}{3}\left(\sum{\frac{\sit_i^2}{\sif_i^4}}\right)\sif^2_i - \frac{1}{6} \sit_i \overline{\sif^2},
\end{align}
\end{subequations}
where $\dot{}$ denotes differentiation with respect to a new variable $s$ defined (up to a constant) by
\begin{equation}
\label{eq:s:t:var:change}
\frac{dt}{ds} = g.
\end{equation}
The phase space of this system %
is the smooth $4$-manifold $\mathcal{P}\subset 
\R^3_> \times \R^3$ defined by 
\[
\mathcal{P}:= \left\{ (\sif_1,\sif_2,\sif_3,\sit_1,\sit_2, \sit_3) \, | \, \sif_1 \sif_2 \sif_3 = 1, \ \tfrac{\sit_1}{\sif_2^2} +  \tfrac{\sit_2}{\sif_2^2} +  \tfrac{\sit_3}{\sif_3^2} = 0\right\}. 
\]
$\mathcal{P}$ has the structure of a real $2$-plane bundle over the smooth connected surface of $\R^3_>$ 
cut out by the equation $\sif_1 \sif_2 \sif_3=1$. Radial projection shows that the latter surface is diffeomorphic 
to~$\Sph^2_>:=\Sph^2 \cap \R^3_>$.

Note also that an immediate consequence of~\eqref{eq:g:dot} and~\eqref{eq:s:t:var:change} is that $g(s)$ 
satisfies
\begin{equation}
\label{eq:dg:ds}
\frac{d}{ds} \ln{g} = \frac{1}{6} \overline{\sif^2}.
\end{equation}
Equation~\eqref{eq:dg:ds} clearly implies that given a solution $(\sif,\sit)$ to the purely scale-invariant steady soliton 
system~\eqref{eq:steady:normal} we can recover the scale-factor $g$ 
via quadrature. 
\begin{remark*}
The fact that in the steady case the scale-invariant variables $\sif_i$ and $\sit_i$
satisfy a self-contained ODE system and that the scale-factor $g$ is then determined by the scale-invariant variables 
is strongly reminiscent of an analogous decoupling that occurs in B\"ohm's work on noncompact cohomogeneity-one Ricci-flat metrics~\cite[Section 3]{Bohm:BSMF}.
\end{remark*}

\begin{remark}
\label{rmk:scale_invt_tf}
The 2-dimensional submanifold $\mathcal{P}_{\tu{tf}} \subset \mathcal{P}$ where
all $\sit_i = 0$ is invariant.
By Lemma~\ref{lem:torsion:free:solitons}, the solitons corresponding to flow
lines in that submanifold are static solitons on torsion-free \gtstr s.
The restriction $\dot{\sif}_i  = \frac{1}{2\sif_i}
\left(- 2\sif_i^4 + \frac{2}{3} \overline{\sif^2} \sif^2_i\right)$ of
\eqref{eq:steady:normal} to $\mathcal{P}_{\tu{tf}}$ is simply a scale-invariant
reformulation of the torsion-free system \eqref{eq:SU3:tf}.
\end{remark}
 
A simple, but key, observation about the ODE system~\eqref{eq:steady:normal}
is that it has a unique critical point~$c_{\tu{tf}}$; this critical point corresponds to the torsion-free~\gtstr~on the cone 
over the flag variety $F_{1,2}=\sunitary{3}/\T^2$; it will also be crucial that $c_{\tu{tf}}$ is a stable critical point of~\eqref{eq:steady:normal}.
\pagebreak[2]
\begin{lemma}\hfill
\label{lem:steady:crit:pt}
\begin{enumerate}[left=0pt]
\item
The purely scale-invariant steady soliton ODE system~\eqref{eq:steady:normal}
has a unique critical point given by $\sif_1=\sif_2=\sif_3=1$, $\sit_1=\sit_2=\sit_3=0$; this critical point corresponds to the 
scale-invariant description of the~\sunitary{3}-invariant torsion-free \gtstr~on the cone 
over the flag variety $F_{1,2}=\sunitary{3}/\T^2$. We therefore we denote this critical point by $c_{\tu{tf}}$. 
\item
The linearisation of~\eqref{eq:steady:normal} at $c_{\tu{tf}}$ has eigenvalues $-\tfrac{1}{2}$ and $-2$, each with multiplicity two. Therefore $c_{\tu{tf}}$ is an exponentially stable critical point of~\eqref{eq:steady:normal}
with (exponential) rate of convergence at worst $\frac{1}{2}$. The $-2$ eigenspace is the tangent space to $\mathcal{P}_{\tu{tf}}$.
\end{enumerate}
\end{lemma}
\begin{proof}Let $(\sif_1, \sif_2,\sif_3,\sit_1,\sit_2,\sit_3)$ be any critical point of~\eqref{eq:steady:normal}. Equation~\eqref{eq:sifi:dot} implies that $\sis_i=0$ for all $i$. 
Hence by Remark~\ref{rmk:siti:dot:sis0} the condition $\dot{\sit}_i=0$  for all $i$ forces $\sit_i=0$ for all $i$. 
As we observed previously in the proof of Lemma~\ref{lem:conical:implies:tf}, the vanishing 
of all $\sis_i$ and all $\sit_i$ implies that $\sif_i=1$ for all~$i$. 
Denote this unique critical point $(1,1,1,0,0,0)$ by $c_{\tu{tf}}$.

The tangent space to the phase space $\mathcal{P}$ at $c_{\tu{tf}}$ 
is $\{(\zeta, \eta) \in \R^6 :  \overline{\zeta}=0=\overline{\eta}\}$,
where $\overline{\zeta}$ means the sum of its three coefficients.
The linearisation of~\eqref{eq:steady:normal} about its unique critical point $c_{\tu{tf}}$ is given by %
\begin{subequations}
\begin{align}
\frac{d \zeta_i}{ds} &= -2 \zeta_i + \frac{1}{2} \eta_i,\\
\frac{d \eta_i}{ds} & = \phantom{ -2 \zeta_i } - \frac{1}{2} \eta_i,
\end{align}
\end{subequations} 
This vector field has eigenvalues $-2$ and $-\tfrac{1}{2}$ each with multiplicity two and
\begin{align*}
E_{-2} &= \{ (\zeta,\eta) \in \R^3 \times \R^3\, |\, \overline{\zeta}=0, \,\eta=0\}, \\
 E_{-\tfrac{1}{2}} & = \{ (\zeta,\eta) \in \R^3 \times \R^3 \, | \, 3\zeta=\eta, \,\overline{\zeta} = \overline{\eta}=0\},
\end{align*}
where $E_\mu$ denotes the eigenspace of eigenvalue $\mu$.
\end{proof}

\subsection{AC steady ends}
\label{ss:AC:ends}

To conclude the section, we explain how Lemma \ref{lem:steady:crit:pt} 
readily implies Theorem~\ref{mthm:steady_ac_end}, and moreover helps us improve
Proposition \ref{prop:lifetime} in the steady
case to say that boundedness of the~$\frac{f_i}{f_j}$ implies not just
forward-completeness but also that the resulting end must be AC. 

Firstly, the uniqueness part of Theorem \ref{mthm:steady_ac_end}(i) is
immediate from the uniqueness of $c_{\tu{tf}}$, while the rest of
Theorem \ref{mthm:steady_ac_end}(i--ii) is captured by the following
consequence of Lemma \ref{lem:steady:crit:pt}. 

\begin{corollary}
\label{cor:exist_AC_end}
There exists a 3-dimensional family (up to time translation and scale)
of steady ends that are asymptotic to the torsion-free cone.
All such AC steady ends have asymptotic rate $-1$, except for a 1-parameter
subfamily of static solutions on torsion-free ends with rate $-4$.
\end{corollary}

\begin{proof}
It is immediate from Lemma \ref{lem:steady:crit:pt} that there exists a
3-parameter family of flow lines approaching $c_{\tu{tf}}$ at least
as fast as $e^{-\frac{s}{2}}$. The only ones with faster rate of approach
are those within the invariant submanifold $\mathcal{P}_{\tu{tf}}$,
which approach like $e^{-2s}$. There is a 1-parameter subfamily of such flow
lines, and as noted in Remark \ref{rmk:scale_invt_tf} those correspond
to static solutions on torsion-free ends.

To translate the conclusion into the rates in terms of $t$, note that
asymptotically $t$ and $s$ are related by $ t^{-2} \simeq e^{-s}$.
\end{proof}

\begin{remark}
Among the 1-parameter family of torsion-free \sunitary{3}-invariant ends (up to
scale), there are precisely three where two of the variables $f_i$ agree.
Those are the ends of the (three variants of the) Bryant--Salamon AC manifold from Example \ref{ex:BS}.
\end{remark}

The proof of Theorem \ref{mthm:steady_ac_end} is completed by noting the
following consequence of the stability of $c_{\tu{tf}}$.

\begin{corollary}
\label{cor:AC_stab}
AC steady ends are stable in the sense that if a non-singular initial
condition for~$(f_i, \tau_i)$ (\ie $f_i > 0$) yields a forward-complete
AC solution, then so does any sufficiently close initial condition.
\end{corollary}

We moreover obtain the following criterion for solutions to
approach $c_{\tu{tf}}$.

\begin{theorem}
\label{thm:AC_basin}
A steady soliton with all $\frac{f_i}{f_j}$ bounded is AC, generically with rate $-1$,
to the torsion-free cone.
\end{theorem}

\begin{proof}
The solution is forward-complete by Proposition \ref{prop:lifetime}.

It remains to prove that in fact the solution must be asymptotic to the unique~\sunitary{3}-invariant torsion-free \gtwo-cone.
To prove this it suffices to show that the corresponding solution of the purely scale-invariant ODE system~\eqref{eq:steady:normal} must eventually enter the basin of attraction of the stable fixed point $c_{\tu{tf}}$. To see this, 
first we prove that all the scale-invariant torsion components $\sit_i$ have vanishing limits as $t \to \infty$. 

The hypothesis implies that there is a constant $C \ge 1$ such that
each $f_i \leq Cg$. Then $\frac{1}{2} \leq g' \leq \frac{1}{2}C^2$,
so $g$ grows linearly (and hence so does each $f_i$).
An immediate consequence of~\eqref{eq:ODEs:SU3:taui} is that the three quantities $\tilde \tau_i:= \tau_i - uf_i^2$ are each constant in $t$. 
We denote those constant values by $c_i$. 
Using the type 14 condition \eqref{eq:tau_14}
(see Section~\ref{ss:steady:special} for further details, in particular equation~\eqref{eq:taui:fi:steady}).
shows that $\sit_i$
can be written in terms of $g$, the constants $c_i$ and the $\sif_i$ as
\[
\sit_i = \frac{1}{3g} \left (2c_i - c_j \frac{\sif_i^2}{\sif_j^2} - c_k \frac{\sif_i^2}{\sif_k^2} \right).
\]
The assumed boundedness of all the ratios $f_i/f_j$ is of course equivalent to the boundedness of all the ratios $\sif_i/\sif_j$ 
and hence the linear growth of $g$ implies that $\sit_i \to 0$ as $t \to \infty$.

Note also that since $g \ge \tfrac{1}{2}t$ it follows from~\eqref{eq:s:t:var:change} and the infinite-$t$ lifespan of the solutions that their $s$-lifetimes 
are also infinite. In terms of the variable $s$ the equation satisfied by the scale-invariant average $\overline{\sif^2}$ is
\[
\frac{d}{ds} \overline{\sif^2} = - \frac{2}{3} D(\sif)  +\overline{\sit}.
\]
Suppose that for a contradiction that $\overline{\sif^2}$ remains bounded away from $3$ as $s \to \infty$. This is equivalent to supposing that there exists a constant $\wt C>0$ so that $D(\sif) > \wt C$ holds for all $s$ sufficiently large. The term $\overline{\sit}$ 
on the right-hand side of the previous equation tends to $0$ as $s \to \infty$ 
(because  $\lim_{t \to \infty}\sit_i=0$ for all $i$). Hence for $s$ sufficiently large we have $\frac{d}{ds} \overline{\sif^2} < -\tfrac{1}{2}\wt C$, but since the solution has infinite $s$-lifetime this contradicts the fact that (by its definition) $\overline{\sif^2}$ is bounded below by $3$.
The fact that  as $s \to \infty$, $\overline{\sif^2}$ has $3$ as a limit point and that $\sit_i \to 0$ for all $i$ implies that the solution 
must eventually enter the basin of attraction of the stable fixed point $c_{\tu{tf}}$. 
\end{proof}

\section{Forward-time behaviour of steady solitons}
\label{s:forward:steady}

Above (in Sections~\ref{ss:finite:extinction} and~\ref{ss:AC:ends})
we have encountered the following two types of behaviour for the evolution of an
\sunitary{3}-invariant soliton forward in $t$:
\begin{itemize}[left=0pt]
\item Forward-complete and asymptotically conical;
\item Forward-incomplete, with its smallest variable $f_k= O(\sqrt{t_*-t})$
near the extinction time $t_*$.
\end{itemize}
In fact, both of these types of end solutions exist for any dilation constant
$\lambda$. The incomplete solutions come in a 4-parameter family up to time translation. 
In the steady case this amounts to a 3-parameter family up to scale and time translation,
as does the family of AC end solutions (Corollary \ref{cor:exist_AC_end}).

\begin{samepage}
In this section we establish the existence of a 2-parameter family (up to scale and time translation) of steady solutions of a third type:
\begin{itemize}[left=0pt]
\item Forward-complete, with one variable converging to a finite positive  limit and the
other two variables growing exponentially (and becoming close to each other) as $t \to \infty$.
\end{itemize}
\end{samepage}

\begin{remark*}
There are non-AC forward-complete end solutions, but with qualitatively
different asymptotics, also for expanders and shrinkers \cite[\S7.3 \& \S9.2]{HJN}.
\end{remark*}

The main result of this section is Theorem~\ref{thm:trichotomy}, a more
detailed version of Theorem \ref{mthm:trichotomy},
establishing that for any~\sunitary{3}-invariant steady soliton, the
forward-evolution falls into one of these three classes.

To show the existence of the exponentially-growing complete solutions we find a coordinate change that
transforms the steady soliton equations into a polynomial ODE system of degree 3. These alternative
coordinates will also be a key ingredient in the next section: there
we will use them to identify to which of the three classes a given
smoothly-closing steady soliton belongs.

\subsection{Exploiting the conserved quantities of the steady ODE system}
\label{ss:steady:special}
We have already seen in the previous section that the ODE system for steady solitons has special features. 
However, so far we have not exploited fully what is arguably the single most important feature: the existence of three conserved quantities
that are not present for shrinkers or expanders. 

To this end, observe from~\eqref{eq:ODEs:SU3:taui} that when $\lambda=0$ the three
quantities $\tau_i - uf_i^2$ are conserved, say 
\begin{equation}
\label{eq:taui:steady}
\tau_i=uf_i^2 + c_i
\end{equation}
for some triple $(c_1,c_2,c_3) \in \R^3$ satisfying $c_1+c_2+c_3 = 0$.
Using~\eqref{eq:taui:steady} one can verify that the type $14$ condition \eqref{eq:tau_14} on $\tau$ is equivalent to 
\begin{equation}
\label{eq:u:steady}
u = -\frac{1}{3} \left( \frac{c_1}{f_1^2} + \frac{c_2}{f_2^2} + \frac{c_3}{f_3^2}\right).
\end{equation}
Substituting the expression for $u$ from~\eqref{eq:u:steady} into~\eqref{eq:taui:steady} we obtain the following expressions for the components of $\tau$ in terms of the triple $f$
\begin{equation}
\label{eq:taui:fi:steady}
\tau_i = \frac{2c_i}{3} - \frac{c_j f_i^2}{3f_j^2} - \frac{c_kf_i^2}{3f_k^2} .
\end{equation}
Similarly,
\begin{equation}
\label{eq:normtau:steady}
\begin{aligned}
\abs{\tau}^2 = \sum{\tau_i^2 f_i^{-4}} &= \frac{1}{9}\sum_{cyc} \left(\frac{2c_i}{f_i^2}
- \frac{c_j}{f_j^2} - \frac{c_k}{f_k^2}\right)^2 \\
&= \frac{2}{3}\left(\frac{c_1^2}{f_1^4}+\frac{c_2^2}{f_2^4} + \frac{c_3^2}{f_3^4}\right)
- \frac{1}{3}\left(\frac{c_1c_2}{f_1^2f_2^2}+\frac{c_2c_3}{f_2^2f_3^2} + \frac{c_3c_1}{f_3^2f_1^2}\right)
\end{aligned}
\end{equation}
where the final equality uses~\eqref{eq:u:steady}.

An immediate consequence of formulae~\eqref{eq:taui:steady} and~\eqref{eq:u:steady}
is that in the steady case the first-order ODE system~\eqref{eq:SU3:order1} for $(f_i,\tau_i)$ can be rewritten as a self-contained first-order system involving only the triple $(f_i)$ and the zero-sum triple $(c_1, c_2, c_3)$ 
that specifies the values of the three conserved quantities $\tau_i - u f_i^2$ 
(for any steady soliton that closes smoothly on the singular orbit $\CP^2_1$ 
it follows from~\eqref{eq:smooth:close} that this triple must have the form $(0,c,-c)$ for some $c \in \R$).

\begin{lemma}
A steady soliton %
satisfies the first-order ODE system
\begin{equation}
\label{eq:fi:steady}
\begin{aligned}
\frac{d}{dt} \ln{f_i^2} &=
-\frac{1}{3} \left( \frac{c_1}{f_1^2} + \frac{c_2}{f_2^2} +
\frac{c_3}{f_3^2}\right) + \frac{c_i}{f_i^2} +
\frac{1}{f_1f_2f_3} (\fb - 2f_i^2) 
\\ &=
\frac{2c_i}{3f_i^2} - \frac{c_j}{3f_j^2} - \frac{c_k}{3f_k^2} +
\frac{1}{f_1f_2f_3} (-f_i^2+f_j^2+f_k^2) 
\end{aligned}
\end{equation}
where $c_i$ are the values (with sum 0) of the three conserved quantities
$\tau_i - uf_i^2$ and $(ijk)$ is any permutation of  $(123)$.

Conversely, for any zero-sum triple $(c_1,c_2,c_3) \in \R^3$, given any solution $f=(f_1,f_2,f_3)$ to the ODE system~\eqref{eq:fi:steady}, then
defining $u$ via~\eqref{eq:u:steady} and $\tau_i$ via~\eqref{eq:taui:steady} makes 
$(f_1,f_2,f_3,u)$ a solution of the mixed-order soliton ODE system~\eqref{eq:ODEs:SU3} and $(f_1,f_2,f_3,\tau_1,\tau_2,\tau_3)$ a solution of the first-order soliton ODE system~\eqref{eq:SU3:order1}. 
\end{lemma}
\begin{remark*}
Note that when $c_1 = c_2 = c_3 = 0$ the system~\eqref{eq:fi:steady} reduces to
the torsion-free ODE system~\eqref{eq:SU3:tf}.
\end{remark*}

We note that~\eqref{eq:fi:steady} implies that the ratio $f_i/f_j$ satisfies the differential equation
\begin{equation}
\label{eq:fi:fj:steady}
\frac{d}{dt} \log{\left(\frac{f_i^2}{f_j^2}\right)} = \frac{c_i}{f_i^2} - \frac{c_j}{f_j^2} + \frac{2(f_j^2-f_i^2)}{f_1 f_2 f_3}.
\end{equation}
These differential equations for the ratios $f_i/f_j$ will play an important role in controlling the 
behaviour of steady solitons.

\subsection{Cubic system and exponential growth ends}
\label{ss:cubic_exp}

The ODE system~\eqref{eq:fi:steady} that the triple~$f$ satisfies is a rational ODE system. The next result shows that a suitable change 
of variables transforms it into a polynomial ODE system. 
This polynomial reformulation gives us some important insights into the
possible asymptotic behaviours of steady solitons not readily apparent
from~\eqref{eq:fi:steady}.

\begin{lemma}
For any~\sunitary{3}-invariant soliton $(f_1,f_2,f_3)$
the triple $F=(F_i)$ defined by \begin{equation}
\label{def:Fi}
F_i:= \frac{f_i}{f_j f_k} = \frac{f_i^2}{g^3},
\end{equation}
where as usual $(ijk)$ is a cyclic permutation of $(123)$, satisfies the polynomial ODE system
\begin{equation}
\label{eq:steady:poly}
F_i' = \frac{c_i}{3} F_1 F_2 F_3 -\frac{F_i^2}{6}(c_jF_k+c_kF_j) + \frac{1}{4} F_i(F_j + F_k - 3F_i)
\end{equation}
for the zero-sum triple $(c_1,c_2,c_3)$ determined by~\eqref{eq:taui:steady}.
\end{lemma}

\begin{proof}
Noting that $F_i^2 = \frac{f_i^2}{g^3}$ and $F_i F_j= f_k^{-2}$, we deduce
from~\eqref{eq:fi:steady}
that
\begin{align*}
\label{eq:log:der:Fi}
\frac{d}{dt} \log{F_i^2} &= \frac{d}{dt} \log{f_i^2} - \frac{d}{dt} \log{g^3} \\
&=
\frac{2}{3} c_i F_j F_k - \frac{1}{3} c_j F_i F_k - \frac{1}{3} c_k F_i F_j + (F_j+F_k-F_i) - \frac{1}{2} (F_1+F_2+F_3),
\end{align*}
from which the claim follows directly.
\end{proof}
\noindent
Note that~\eqref{eq:steady:poly} is an inhomogeneous (for $c \neq 0$) cubic system with no constant or linear terms. 
The next result determines the critical points of this cubic ODE system.
\begin{lemma} 
\label{lem:poly:fixed}
Fix any any values $c_i$, with $c_1 +c_2+ c_3 = 0$, for the constants
appearing in~\eqref{eq:steady:poly}.
\begin{enumerate}[left=0em]
\item
Any solution $(F_1,F_2,F_3)$ of~\eqref{eq:steady:poly} defined for all positive $t$ satisfies
\begin{equation}
\label{eq:F123:large:t}
\lim_{t \to \infty} F_1 F_2 F_3=0.
\end{equation}

\item
There are no critical points of~\eqref{eq:steady:poly} within the positive octant.
\item
The origin is a totally degenerate critical point, \ie the linearisation of~\eqref{eq:steady:poly} at the origin is the zero map.
\item
If $c_i < 0$ then setting $F_i = 0$ and the other two variables to
$F_j = F_k = -3c_i^{-1}$ defines a critical point $p_i$ of the
system~\eqref{eq:steady:poly}.
These and the origin are the only critical points within the closure of
the positive octant. 
\item
The critical point $p_i$ %
is a hyperbolic fixed point 
of the system: it has two negative eigenvalues $-k_i$ and $-2k_i$ (where $k_i:= -\tfrac{c_i}{3}>0$) and a single positive one. It therefore has a $2$-dimensional local stable manifold which at the fixed point has normal vector $(1,1,1)$.
\end{enumerate}
\end{lemma}

\begin{proof}
(i) The system of equations~\eqref{eq:steady:poly} implies that 
\begin{equation}
\label{eq:F123:dot}
\frac{d}{dt} \log |F_1F_2F_3| = -\tfrac{1}{2} (F_1+F_2+F_3),
\end{equation}
and hence
\[
\frac{d}{dt} \left|F_1 F_2 F_3\right|^{-1/3} \ge \frac{1}{2}.
\]
(These are essentially equivalent to the closure condition
$2(f_1f_2f_3)'=\fb$ and the inequality~\eqref{eq:g:dot:lower:bound} respectively.)
In particular, if $F_1F_2F_3 \not= 0$ then $|F_1F_2 F_3|$ is a
strictly decreasing function of $t$ and
any solution of~\eqref{eq:steady:poly} defined for all positive $t$ satisfies
\begin{equation*}
\lim_{t \to \infty} F_1 F_2 F_3=0.
\end{equation*}
(ii) A critical point of~\eqref{eq:steady:poly} within the positive octant would have
$\tfrac{d}{dt} \log{F_1F_2F_3}=0$ and hence $F_1+F_2+F_3=0$ by~\eqref{eq:F123:dot}.
\\[0.5em]
(iii) This is immediate from the absence of linear terms on the right-hand side of \eqref{eq:steady:poly}.
\\[0.5em]
(iv)
By the permutation symmetry of the equations, without loss of generality, we can consider 
only critical points of the form $(F_1,F_2,0)$ within the nonnegative octant. 
The condition for $(F_1, F_2,0)$ to be a critical point is given by the pair of algebraic equations
\begin{align*}
-\frac{c_3}{3}F_1^2F_2 + \frac{1}{2}F_1(F_2-3F_1) &= 0, \\
-\frac{c_3}{3}F_2^2F_1 + \frac{1}{2}F_2(F_1-3F_2) &= 0,
\end{align*}
whose unique non-zero solution is $F_1 = F_2 = -\frac{3}{c_3}$.
Apart from the origin, since $\sum{c_i}=0$, we thus get one or two further fixed points within the closure of the positive octant, depending
on whether one or two of the $c_i$ are negative. 
\\[0.5em]
(v)
Without loss of generality suppose $c_3 < 0$. Let $k := -\frac{3}{c_3} > 0$ and consider the fixed point
$p_3 = (k, k, 0)$.
The derivative of the right-hand side of
\eqref{eq:steady:poly} at $p_3$ is represented by
\[ k \begin{pmatrix}
-\tfrac{1}{2} & \tfrac{3}{2} & kc_1 - \tfrac{1}{2} \\
\tfrac32 & -\tfrac12 & kc_2 - \tfrac{1}{2} \\
0 & 0 & -1
\end{pmatrix} . \]
(Note that to write the linearisation in this form we have used the facts that $\sum{c_i}=0$ and that $-k c_3 =3$.)
Its eigenvalues are $k$, $-k$ and $-2k$, with corresponding eigenvectors
$(1,1,0)$, $(-c_1 {+}2c_2, 2c_1{-}c_2, c_3)$ and $(1,-1,0)$ respectively.
\end{proof}

Moreover, if $F$ is a solution of~\eqref{eq:steady:poly} asymptotic to a critical point $p_i$ as $t \to \infty$, then 
by~\eqref{eq:F123:dot} $F_1 F_2 F_3$ decays exponentially to $0$ with rate at least $k_i=3c_i^{-1}$. 
This is equivalent to the fact that $f_1 f_2 f_3$ grows exponentially in $t$ with rate at least $k_i$, \ie that the corresponding steady soliton has a forward-complete end with 
exponential volume growth. 

As an immediate consequence of the existence of a $2$-dimensional stable manifold 
for the fixed point~$p_i$ we have the following:

\begin{corollary}
\label{cor:exist_exp_end}
For each zero-sum triple $(c_1, c_2, c_3)$ with $c_i<0$
there is a $1$-parameter family of distinct forward-complete steady soliton
ends all with exponential volume growth and 
with $f_i$ converging to a finite limit and $f_j/f_k$ converging to $1$.
More specifically, the generic solution converging to $p_i$ has $g^3 \simeq \exp{k_i t}$ where $k_i=-\tfrac{3}{c_i}>0$.
\end{corollary}

Up to scale, the zero-sum triple $(c_1, c_2, c_3)$ involves a single parameter,
so all in all there is a 2-parameter family of such exponential-growth ends.

\subsection{Trichotomy for forward-evolution of steady solitons}
\label{ss:trichotomy}

Here is a detailed statement of the trichotomy for the possible 
forward-evolution behaviours of any steady soliton claimed at the start of
this section. This is a (slightly) more detailed version of Theorem~\ref{mthm:trichotomy}.

\begin{theorem}
\label{thm:trichotomy}
Any~\sunitary{3}-invariant steady soliton satisfies exactly one of the following:
\begin{enumerate}[left=0em]
\item 
It has infinite forward existence time (and therefore gives rise to a metrically-complete end),  and all ratios $f_i/f_j$ remain uniformly bounded as $t \to \infty$. 
Then the steady soliton is AC (generically with rate $-1$) with asymptotic cone the torsion-free cone over the standard~\sunitary{3}-invariant nearly K\"ahler structure on $\sunitary{3}/\T^2$.
\item
It has infinite forward existence time (and therefore it gives rise to metrically-complete end), and as $t \to \infty$ the ratio of the two largest variables $f_i/f_j$ 
tends to $1$ while
the ratio of the largest to the smallest variable $f_i/f_k$ is unbounded. Then
$f_k \to -\tfrac{c_k}{3}>0$ as $t \to \infty$ (in particular, $c_k$ must be
negative) while $f_i$ and $f_j$ both grow exponentially in $t$. 
The volume of large geodesic balls of radius $r$ grows exponentially in $r$; the scalar curvature decays exponentially fast to a negative constant as $r \to \infty$. 
\item
It has a finite forward maximal existence time $t_*$ (and therefore it is metrically incomplete). 
As $t \to t_*$, its smallest variable $f_k \to 0$ with $(f_k^2)' \to \tfrac{2 c_k}{3}<0$ (in particular, $c_k$ must be negative)
while the ratio of its two largest variables $f_i/f_j$ remains bounded. 
\end{enumerate}
\end{theorem}

\begin{remark}
\label{rmk:trichotomy:v2}
It follows immediately from~\eqref{eq:u:steady} and~\eqref{eq:taui:fi:steady} that the previous theorem implies that only three eventual forward-time behaviours
for the soliton vector field coefficient $u$ or the torsion coefficients $\tau_i$ are possible.
\begin{itemize}[left=0em]
\item 
In case (i) of Theorem~\ref{thm:trichotomy}, $u \to 0$ and $\tau_i \to c_i$ for all $i$ as $t \to \infty$.
\item
In case (ii) of Theorem~\ref{thm:trichotomy}, $u \to -\tfrac{3}{c_3}>0$, $\tau_3 \to \tfrac{2c_3}{3}<0$, while $\tau_1 \simeq \tau_2  \simeq - \frac{c_3 f_1^2}{3f_3^2} \to + \infty$ as~$t \to \infty$.
\item
In case (iii) of Theorem~\ref{thm:trichotomy}, $u \simeq \frac{1}{2(t_*-t)} \to + \infty$, $\tau_3 \to \tfrac{2c_3}{3} <0$, 
with $\tau_1 \simeq - \frac{c_3 f_1^2}{3 f_3^2} \to + \infty$ and~$\tau_2 \simeq  \frac{c_3 f_2^2}{3 f_3^2} \to + \infty$, as $t \to t_*$.
\end{itemize}
\end{remark}

The proof of Theorem~\ref{thm:trichotomy} will be established in the rest of this section. 
To make the proof more digestible we have chosen to 
present it in terms of several separate lemmas and propositions, that together immediately imply the theorem.

First, we recall that according to Corollary~\ref{cor:sign_changes} for any non-shrinking soliton and distinct pair $i \neq j$ the quantity $T_{ij} = (\tilde\tau_i - \tilde\tau_j)(f_i^2-f_j^2)$
can go from negative to positive, but not vice versa. Moreover, for a steady soliton the value of $(\tilde\tau_i - \tilde\tau_j)= c_i - c_j$ is constant. 
Thus any given pair of the three variables $f_i$ will swap
order at most once during the lifetime of a solution. In particular,  when analysing end behaviour of solutions without loss of generality we can assume that
$f_1 \geq f_2 \geq f_3$ for all future time. 

For the rest of this section, unless otherwise stated, we will therefore 
make this assumption about the eventual ordering of the coefficients $f_i$. 

Our first result shows that the ratio of the two largest variables $f_1/f_2$ always remains bounded forward in time. 
\begin{lemma}
\label{lem:big_pair}
For any steady soliton with regular initial data, \ie all the $f_i$ are initially positive, the ratio of the two largest variables
$\frac{f_1}{f_2}$ is bounded for all future time.
\end{lemma}

\begin{proof} 
For any given steady soliton we have the fixed zero-sum triple $(c_1,c_2,c_3)$ determined by~\eqref{eq:taui:steady} 
and by assumption initially $g^3=f_1 f_2 f_3 >0$. Hence we can choose 
$\alpha > 0$ large enough so that 
\[
2(1-\alpha^{-2}) \alpha^{2/3} g >  \left(\abs{c_1}\alpha^{-2} - c_2\right)
\]
holds initially and then since $g$ is increasing this remains true for all future $t$.

At any time where $f_1 \ge  \alpha f_2$, \eqref{eq:fi:fj:steady} gives that
\[
\left(\log \frac{f_1^2}{f_2^2}\right)' \le \frac{|c_1|}{\alpha^2 f_2^2} - \frac{c_2}{f_2^2} + 2 \left( \frac{f_2^2}{f_1^2}-1 \right) \frac{f_1}{f_2 f_3}
\le  \frac{|c_1|}{\alpha^2 f_2^2} - \frac{c_2}{f_2^2} - 2(1-\alpha^{-2}) \frac{f_1}{f_2 f_3}
\]
and this is negative whenever additionally 
\begin{equation}
\label{eq:squeeze}
2(1-\alpha^{-2}) \frac{f_1 f_2}{f_3} > \left( \abs{c_1}\alpha^{-2} - c_2\right).
\end{equation}
But because $\frac{\alpha f_3^2}{f_1 f_2} \le \frac{f_3^2}{f_2^2} \le 1$ 
\[
\left(\frac{f_1f_2}{f_3}\right)^3 > \alpha^2 \left(\frac{f_1f_2}{f_3}\right)^3
\left(\frac{f_3^2}{f_1f_2}\right)^2 = \alpha^2 f_1 f_2 f_3,
\]
so \eqref{eq:squeeze} holds by our choice of $\alpha$. 
Thus
if $f_1/f_2 \ge \alpha$ for all time then $f_1/f_2$ is decreasing,
while if $f_1/f_2 \le \alpha$ at some time then that condition is
maintained at all future times. Either way, $f_1/f_2$ is bounded.
\end{proof}

The next result describes the asymptotic behaviour of a forward-incomplete steady soliton close to its extinction time. 
\begin{prop}
\label{prop:steady_extinct}
If a steady soliton is forward-incomplete with extinction time $t_*$ then its smallest variable $f_3 \to 0$
as~$t \to t_*$, while $(f_3^2)' \to \frac{2c_3}{3} < 0$. (In particular, this can happen only if $c_3<0$).
\end{prop}

\begin{proof}
First we show that $f_3$ cannot be bounded away from 0
(where as previously we have arranged that $f_1 \geq f_2 \geq f_3$).
If it were, then \eqref{eq:fi:steady} implies that 
\[\frac{d}{dt}\log (f_1^2f_2^2) = 
- \frac{2c_3}{3f_3^2} + \frac{c_1}{3f_1^2} + \frac{c_2}{3f_2^2} +
\frac{2f_3}{f_1f_2} \]
is bounded above.
Thus $f_1f_2$ remains bounded in finite time. Since $f_2 \ge f_3$ is bounded
below, this implies that $f_1$ is bounded above. Therefore $\frac{f_1}{f_3}$ 
remains bounded for finite time (and hence so do the other ratios
$\frac{f_i}{f_j}$), contradicting Proposition \ref{prop:lifetime}.

Next we argue that there exists an $\epsilon > 0$ such that once $f_3 < \epsilon$,
then $f_3' < 0$.
When $f_3$ is small, then (since $f_1f_2f_3$ is increasing and
$\frac{f_1}{f_2}$ is bounded by Lemma \ref{lem:big_pair}) $f_1$ and $f_2$ must
both be large, and so the right-hand side of
\begin{equation}
\label{eq:f3'}
(f_3^2)' =
\frac{2c_3}{3} + f_3\left(\frac{f_1}{f_2}+\frac{f_2}{f_1}\right)
+ f_3^2\left(-\frac{c_1}{3f_1^2} - \frac{c_2}{3f_2^2} -
\frac{f_3}{f_1f_2}\right)
\end{equation}
is dominated by the first two terms. Therefore if $c_3 \geq 0$ then, using again the fact that $f_1/f_2$ is bounded, we see that
$f_3' > 0$ whenever $f_3$ is small enough, implying that $f_3$ is bounded
away from 0 in contradiction of the above.
So we must have $c_3 < 0$, and hence $f_3' < 0$ when $f_3$ is small enough. Thus $f_3 \to 0$, and
\eqref{eq:f3'} gives $(f_3^2)' \to \frac{2c_3}{3}$.
\end{proof}

The next result will be used in the proof of Proposition~\ref{prop:fc:steady:exp:vol} where we 
determine the asymptotic behaviour of any forward-complete solution where the ratio of the largest variable
to the smallest variable is unbounded as $t\to \infty$. 
\begin{lemma}
\label{lem:min_force_AC}
If at any instant a steady soliton satisfies $\min_i f_i > 7\max_j |c_j|$ then it is forward-complete and all $\frac{f_i}{f_j}$ are bounded.
\end{lemma}

\begin{proof}
For the smallest component $f_i$ we clearly have
\[
\frac{f_j^2+f_k^2-f_i^2}{f_1 f_2 f_3} = \frac{f_j^2+f_k^2-2f_i^2}{2 f_i f_j f_k} + \frac{f_j^2+f_k^2}{2 f_j f_k} \ge  \frac{f_j^2+f_k^2}{2 f_j f_k}. 
\]
Hence once the condition assumed in the statement happens, we can bound the derivative of the smallest $f_i$ by
\[ f_i' = \frac{f_i}{2}(\log f_i^2)' > f_i\left(
-\sum_j\frac{|c_j|}{f_j^2}\right) + \frac{f_j^2+f_k^2}{4f_jf_k}
> -\frac{3\max|c_j|}{f_i} + \frac{1}{2} > \frac{1}{14}\]
to see that $\min f_i$ is increasing thenceforth.
Thus by Proposition~\ref{prop:steady_extinct} the lifetime is infinite, and the derivative lower bound on  $\min f_i$ implies that all $f_i \to \infty$.

To simplify the notation, assume now with loss of generality that we have passed the last time
when the ordering of the $f_i$ changes, and that $f_1 \geq f_2 \geq f_3$. 
It now suffices to prove that $f_1/f_3$ is bounded above. 
If $f_1 > 2f_3$ then by~\eqref{eq:fi:fj:steady}
\[ \left(\log\frac{f_1^2}{f_3^2}\right)'
= \frac{c_1}{f_1^2} - \frac{c_3}{f_3^2} + \frac{2(f_3^2 -f_1^2)}{f_1f_2f_3}
<  \frac{|c_1| + |c_3|}{f_3^2} + \frac{2f_1^2}{f_1 f_2f_3} \left( \frac{f_3^2}{f_1^2}-1 \right) 
< \frac{|c_1| + |c_3|}{f_3^2}  - \frac{3}{2f_3} \]
is negative for $f_3$ large enough.
Thus $\frac{f_1}{f_3}$ is bounded above.
\end{proof}
\begin{remark*} 
The importance of the previous result is not the particular constant $7$ (which is not sharp) that appears in the statement, but rather that there is a threshold (depending on the initial conditions, \ie on the values of the constants $c_i$)
such that if the smallest $f_i$ ever exceeds that threshold then the solution must be forward-complete with all ratios $f_i/f_j$ bounded. 
\end{remark*}

\begin{prop}
\label{prop:fc:steady:exp:vol}
Any forward-complete steady soliton with $\frac{f_1}{f_3}$ unbounded has
$f_3 \to -\tfrac{c_3}{3} > 0$ while~$\frac{f_1}{f_2} \to 1$ as $t \to \infty$. 
In fact, $f_1$ and $f_2$ both grow exponentially as $t \to \infty$. 
(Again, this can happen only if $c_3<0$).
\end{prop}

\begin{proof}
By the previous lemma, if $f_1/f_3$ is unbounded then the smallest variable $f_3$ must remain bounded above.
Then the fact that $g \to \infty$ as $t \to \infty$ together with
Lemma \ref{lem:big_pair} implies that both $f_1$ and $f_2 \to \infty$.

As previously we can assume without loss of generality that $f_1 \geq f_2$. Now given any $C > 1$, if $f_1 > Cf_2$ then
\eqref{eq:fi:fj:steady} implies that
\[ f_3\left(\log\frac{f_1^2}{f_2^2}\right)'
< \frac{|c_1|}{f_1} + \frac{|c_2|}{f_2} -2\left(\frac{C^2-1}{C}\right) < 
 \frac{|c_1|}{f_1} + \frac{|c_2|}{f_2} -2(C-1)< -(C-1) \]
for $t$ large enough, since $f_1$ and $f_2$ both tend to $\infty$. 
Because $f_3$ is bounded above,
$\frac{f_1}{f_2}$ will thus eventually decrease past $C$ and then remain below $C$.
Hence $\frac{f_1}{f_2} \to 1$.

Then \eqref{eq:f3'} gives
\[ (f_3^2)' - 2\left(\frac{c_3}{3} + f_3\right)  \to 0 . \]
We must have $c_3 \not= 0$, since otherwise $f_3' \to 1$, contradicting the fact
that $f_3$ is bounded.

Now for any $\epsilon > 0$ there is $t_1$ such that if
$|f_3(t) + \frac{c_3}{3}| > \epsilon$ for any $t > t_1$
then $f_3'$ has the same sign as $f_3+\frac{c_3}{3}$. In other words,
$\frac{d}{dt}|f_3+\frac{c_3}{3}| > 0$, and hence the condition $|f_3 + \frac{c_3}{3}| > \epsilon$ persists from then on. If $f_3$ is moreover bounded then
$\frac{d}{dt}|f_3+\frac{c_3}{3}|$ is bounded away from 0, so $f_3$ is either unbounded
above or reaches 0 in finite time, either of which is a contradiction.

Thus $|f_3(t) + \frac{c_3}{3}| < \epsilon$ for all $t > t_1$.
Since $\epsilon$ was arbitrary, that means that $f_3 \to -\frac{c_3}{3}$.
\end{proof}

\pagebreak[2]

\begin{corollary}
Viewed in the $(F_1, F_2, F_3)$ system, a forward-complete steady soliton with
$\frac{f_1}{f_3}$ unbounded converges to the limit point $p_3$.
\end{corollary}
\nopagebreak
This completes the proof of Theorem \ref{thm:trichotomy}, and hence of Theorem \ref{mthm:trichotomy}.

\pagebreak[2]
\section{Complete \texorpdfsunitary{3}-invariant steady solitons}
\label{S:steady}

To find complete cohomogeneity-one solitons %
we need solutions that both close smoothly
on a singular orbit and are forward-complete.
In this section in the steady case we are able to resolve the completeness question definitively.  
In particular, Theorem~\ref{thm:steady:complete} establishes our final main Theorem~\ref{mthm:complete:steady}.
This result determines precisely which of the smoothly-closing steady solitons are forward-complete 
and also what is the geometric behaviour of the resulting end. 

Recall that by Corollary~\ref{Cor:Sp2:steady}, any 
smoothly-closing \Sp{2}-invariant steady soliton is necessarily a trivial
soliton, \ie has vanishing vector field and the underlying \gtstr~is torsion
free. 
Throughout this section we will therefore \emph{only} consider
\sunitary{3}-invariant steady solitons. 

In Theorem \ref{thm:SU3:smooth:closure} we identified, for each fixed
$\lambda$, all the local \sunitary{3}-invariant solitons that close smoothly
on $\CP^2$ (\ie are defined on a neighbourhood of the zero section in
$\Lambda^2_- \CP^2$).
In the case~$\lambda = 0$, let us denote the 2-parameter family of
locally-defined steady solitons by $\mathcal{S}_{b,c}$.
Recall from Remark \ref{rmk:scaling_sc} that $\mathcal{S}_{\mu b, \mu c}$ is
a rescaling of $\mathcal{S}_{b,c}$, so we have a 1-parameter family up
to scale.

Theorem \ref{thm:trichotomy} identified the three possible types of
behaviour of a steady soliton under forward evolution: AC, forward-complete with exponential volume growth, 
or finite extinction. We now wish to identify to which of the
three classes each element of the family $\mathcal{S}_{b,c}$ belongs. 

\begin{remark}
\label{rmk:AC:steady:stability}
The stability of AC steady ends established in Corollary~\ref{cor:AC_stab},
together with the standard continuous dependence of $\mathcal{S}_{b,c}$ on the initial data
$(b,c)$, immediately implies that for fixed~$b>0$, the set of $c \in \R$ such
that $\mathcal{S}_{b,c}$ is AC is open. Since $\mathcal{S}_{b,0}$ is the
static solution given by the AC Bryant--Salamon torsion-free~\gtstr~on $\Lambda^2_- \CP^2$, 
$\mathcal{S}_{b,c}$ defines an AC steady soliton on~$\Lambda^2_- \CP^2$
for sufficiently small $c$ (which is a non-trivial steady soliton for $c \not= 0$).
\end{remark}

However, the line of reasoning in the previous remark does not help much with establishing whether
there is any threshold for the parameter where the long-time behaviour
transitions away from being AC.
Numerical investigation suggested that such a threshold does in fact occur
at $\frac{c^2}{b^2} = \frac{9}{2}$; this parameter value turns out to
correspond to an explicit complete solution with exponential volume growth
which we describe in Section \ref{ss:explicit}. By comparing the solutions
for other parameter values with this explicit solution we obtain the following decisive result.

\begin{samepage}
\begin{theorem}\hfill
\label{thm:steady:complete}
\begin{enumerate}[left=0em]
\item
For any $b > \sqrt{2}$, %
$\mathcal{S}_{b,3}$  is asymptotic with rate $-1$ to the torsion-free \gtwo-cone on $\sunitary{3}/\T^2$ (case \textup{(i)} in Theorem \ref{thm:trichotomy}).
\item $\mathcal{S}_{\sqrt{2}, 3}$ is complete with exponential volume growth
(case \textup{(ii)} in Theorem \ref{thm:trichotomy}).
\item
For any $b < \sqrt{2}$ the smoothly-closing steady soliton $\mathcal{S}_{b,3}$
is incomplete  (case \textup{(iii)} in Theorem \ref{thm:trichotomy}).
\end{enumerate}
\end{theorem}
\end{samepage}

In this statement we have used the scale freedom to normalise $c = 3$, rather
than $b = 1$ as in the statement of Theorem~\ref{mthm:complete:steady} in the introduction. Keeping $c$ fixed is convenient in the
proof because this normalisation allows us to treat all the solitons considered as solutions
of the same cubic system from Section \ref{ss:cubic_exp}.
The proof heavily exploits the fact that the explicit solution
$\mathcal{S}_{\sqrt{2}, 3}$ has a simple description in this system.
The complete cases (i)
and (ii) are proved in Section \ref{ss:explicit}, and the incomplete case
(iii) in Section \ref{ss:incomplete}. 

\subsection{Smoothly-closing steady solitons}

As above, we denote the smoothly-closing steady solitons by $\mathcal{S}_{b,c}$ where $b>0$ 
and $c \in \R$ and recall that this gives rise to a $1$-parameter family
of smoothly-closing steady solitons distinct up to scale.
By acting with the element of the Weyl group $W$ that exchanges $f_2$ and $f_3$, recall Lemma~\ref{lem:Weyl:SU3}, we can 
also transform a solution with $c<0$ into one with~$c>0$. Therefore, for simplicity of exposition, in most of this section we will make the assumption that $c>0$;
results about solutions with $c<0$ follow easily by minor modifications of those obtained for solutions with $c>0$.

We have noted before that \eqref{eq:ODEs:SU3:taui} implies that for a steady
soliton, the quantities $\tilde \tau_i = \tau_i - uf_i^2$ take constant values
$c_i$.
The small-$t$ power series for the smoothly-closing $\lambda$-soliton with initial conditions $b$ and $c$ given in
Appendix~\ref{app:sc:power:series} imply that $u(0) = 0$, and
$\tilde\tau_1(0) = 0$, $\tilde\tau_2(0) = c$ and $\tilde\tau_3(0) = -c$.
In particular, on the smoothly-closing steady soliton $\mathcal{S}_{b,c}$ the values of the conserved quantities are
\begin{equation}
c_1 = 0, \qquad c_2 = c > 0, \quad \textrm{ and } \quad c_3 = -c < 0.
\end{equation}

\begin{lemma}
\label{lem:steady:ordering}
For any smoothly-closing steady soliton $\mathcal{S}_{b,c}$ with $c>0$ the component $f_2$ is dominant, 
\ie for all positive times within the lifetime of the solution the triple $f$ satisfies the ordering properties
\[
f_2 > f_3, \qquad f_2 >f_1.
\]
Hence 
\[
u>0, \quad f_2 > \frac{1}{2}t
\]
also hold for all positive times within the lifetime of the solution.
\end{lemma}
\begin{proof}
The small-$t$ power series for the smoothly-closing solutions given in Appendix~\ref{app:sc:power:series} imply that
\[
f_2 - f_1 = b + \frac{c-6b}{6b}t + \hot, \quad f_2 - f_3 = \frac{c}{3b}t + \hot
\]
Hence $f_2-f_1$ and $f_2-f_3$ are both positive (since we assumed $c>0$) for $t>0$ sufficiently small. 
Since $\tilde \tau_2 - \tilde \tau_1 = c>0$,
Corollary~\ref{cor:sign_changes} implies that the condition
$f_2 > f_1$ is preserved.
Similarly $\tilde \tau_2 - \tilde \tau_3=2c>0$ implies that $f_2 > f_3$
is preserved.

Positivity of $u$ now follows immediately from the dominance of $f_2$ and~\eqref{eq:u:steady}.
\end{proof}

\begin{remark}
On the other hand, any smoothly-closing steady soliton satisfies $f_3 > f_1$ initially, but if the soliton is not AC,
then $f_3$ must eventually be the smallest variable, because $c_3 = -c$
is the only negative $c_i$.
This is because in both the other cases of the steady trichotomy (Theorem~\ref{thm:trichotomy}) exactly
one $f_i$ remains bounded, and the corresponding $c_i$ is negative.
Thus in those cases we would be forced to have $f_1 = f_3$ at some $t_0$, such a $t_0$ is unique by Corollary
\ref{cor:sign_changes}, so then $f_1 > f_3$ holds for $t > t_0$. In fact, 
one can prove that any smoothly-closing steady soliton (including the AC ones) with $c>0$ has this property, but because we do not need this fact we omit its proof.
\end{remark}

\begin{remark}
In this section we have chosen to keep the labelling of the variables $f_i$
consistent with Section \ref{s:smooth:closure:solns}. 
As a consequence, the ``eventual'' ordering $f_2 > f_1 > f_3$ is not consistent
with the ordering $f_1 > f_2 > f_3$ we assumed in Section \ref{ss:trichotomy},
but at least the smallest variable (which has distinct behaviour from the other
two) has the same label $f_3$ in both cases.
\end{remark}

\subsection{Complete asymptotically conical and exponentially growing steady
solitons}
\label{ss:explicit}

Observing numerically a transition in the long-time behaviour of smoothly-closing steady solutions led us to investigate solitons at the critical ratio
$\tfrac{c^2}{b^2}=\tfrac{9}{2}$.
By scaling and use of the discrete symmetries we can suppose that $b= \sqrt{2}$ and $c=3$.
In this case the small-$t$ power series expansion for the coefficient $f_1$ given in Appendix~\ref{app:sc:power:series} specialises to
\[
f_1 =  t+\frac{t^3}{24}+\frac{t^5}{1920}+\frac{t^7}{322560}+ \frac{t^9}{92897280} + \frac{t^{13}}{25505877196800} + \cdots = \sum_{n=0}{ \frac{t^{2n+1}}{4^n (2n+1)!}}.
\]
We recognise this as the beginning of the Taylor series for the function $2 \sinh{\tfrac{1}{2}t}$ centred at $t=0$. Inspection of the expansions 
for $f_2^2$ and $f_3^2$ also reveal expansions that are consistent with being the Taylor series for other hyperbolic trigonometric functions. 
This leads us to the following result.

\begin{theorem}
\label{thm:explicit:steady}
The smoothly-closing steady soliton $\mathcal{S}_{\sqrt{2},3}$ is given explicitly by 
\begin{equation}
\label{eq:fi:explicit:steady}
f_1^2 = 2 (\cosh{t}-1) = 4 \sinh^2{\tfrac{t}{2}}, \qquad  f_2^2 = 1 + e^t, \qquad f_3^2  = 1 + e^{-t}.
\end{equation}
It then follows from~\eqref{eq:u:steady},~\eqref{eq:taui:fi:steady} and~\eqref{eq:normtau:steady} that 
$u$ and $f_1 f_2 f_3$ are given by 
\[
\ f_1 f_2 f_3 = 2 \sinh{t}, \qquad u = \frac{e^t -1}{e^t+1}  = \tanh{\tfrac{t}{2}},
\]
that the $\tau_i$ are given by 
\[
\tau_1 = \frac{(e^t -1)^3}{e^t(1+e^t)} = 4 \tanh{\tfrac{t}{2}} \sinh^2{\tfrac{t}{2}}, \qquad 
\tau_2 = 2 + e^t, \qquad 
\tau_3 = - (2+ e^{-t}),
\]
and that
\[
\abs{\tau}^2 =  \frac{6(e^{2t}+e^{t}+1)}{(e^t+1)^2} = 6 - \tfrac{3}{2}\sech^2{\tfrac{t}{2}}.
\]
In particular, as a special case of Theorem \ref{thm:trichotomy} (ii), $\mathcal{S}_{\sqrt{2},3}$ is complete and has exponential volume growth. 
Its scalar curvature decays exponentially fast and is asymptotic to $-3$.
The vector field $X= u \,\partial_t$ decays exponentially fast to the constant vector field $\partial_t$.
\end{theorem}

\begin{proof}
By uniqueness of solutions to the initial value problem it suffices to 
verify by direct computation that the explicit triple $f$ given in~\eqref{eq:fi:explicit:steady} satisfies both the ODE system~\eqref{eq:fi:steady} 
and the initial conditions needed for $\mathcal{S}_{\sqrt{2},3}$.  The claim about scalar curvature follows immediately from~\eqref{scalar:curv:closed}.
\end{proof}

\begin{remark}
The solution of the polynomial version of the ODE system~\eqref{eq:steady:poly} that corresponds to the explicit steady soliton given in~\eqref{eq:fi:explicit:steady} is 
\begin{equation}
\label{eq:Fi:explicit:steady}
F_1 = \tanh{\tfrac{t}{2}}, \quad F_2 = \frac{e^t}{e^t-1}, \quad F_3 = \frac{1}{e^t - 1}.
\end{equation}
Note that $(F_1,F_2,F_3) \to (1,1,0)$ as $t \to + \infty$ and that  the point $(1,1,0)$ is precisely the hyperbolic fixed point $p_3$ of~\eqref{eq:steady:poly} that we identified in Lemma~\ref{lem:poly:fixed} when $(c_1,c_2,c_3)=(0,3,-3)$. 
Also, since only $c_3$ is negative, in this case there is a unique such hyperbolic fixed point within the closure of the positive octant. 
In other words, the positive triple $F$ given in~\eqref{eq:Fi:explicit:steady} that corresponds to the solution 
$\mathcal{S}_{\sqrt{2},3}$ yields a curve that belongs to the $2$-dimensional local stable manifold of the (unique) hyperbolic fixed point $p_3$.

Observe also that the functions $F_i$ defined in~\eqref{eq:Fi:explicit:steady} satisfy the following algebraic relations
\begin{subequations}
\label{eq:explicit:steady:Fi}
\begin{align}
\label{eq:explicit:steady:D}
F_2 - F_3 &= 1, \\ 
\label{eq:explicit:steady:L}
F_1(F_2 +F_3) &= 1.
\end{align}
\end{subequations}
We note also that imposing the condition~\eqref{eq:explicit:steady:L} on a nonnegative triple $F$ defines a smooth surface within the nonnegative octant; this surface contains the hyperbolic fixed point $p_3=(1,1,0)$ and its tangent space at $p_3$ coincides with the stable eigenspace of $p_3$.
\end{remark}

Case (ii) of Theorem \ref{thm:steady:complete} is immediate from 
Theorem~\ref{thm:explicit:steady}.
The rest of this subsection is devoted to the proof of case (i), \ie that $\mathcal{S}_{b,3}$ is AC for $b > \sqrt{2}$.

Since in our setup for the smoothly-closing case $c_3=-3$ is the only negative conserved quantity, 
it follows from Theorem~\ref{thm:trichotomy} that the asymptotic behaviour of $f_3$ completely determines
which alternative within that trichotomy a solution belongs to.
Moreover, since we have normalised $c_3=-3$, the only possible asymptotic behaviours for $f_3$ are:
$f_3 \to \infty$ as $t \to \infty$ in the AC case;
$f_3 \to 1$ as~$t \to \infty$ in the exponentially-growing end case; or,
$f_3 \to 0$ as $t \to t_*$ in the incomplete case.

Motivated by the relation~\eqref{eq:explicit:steady:L}, for any  smoothly-closing steady soliton we define the following positive function of $t$
\begin{equation}
\label{eq:L:defn}
\Lambda:= F_1 (F_2 + F_3) = \frac{1}{f_2^2} + \frac{1}{f_3^2},
\end{equation}
and observe that $\Lambda$ can also be used to distinguish the cases of the
trichotomy in Theorem \ref{thm:trichotomy}. 

\begin{lemma}
\label{lem:Lambda_trich}
A steady soliton with $c_1 = 0$, $c_2 = 3$ and $c_3 = -3$ is
\begin{enumerate}
\item AC if and only if $\Lambda \to 0$ as $t \to \infty$;
\item forward-complete with exponential volume growth if and only if $\Lambda \to 1$
as $t \to \infty$;
\item incomplete with extinction time $t_*$ if and only if $\Lambda \to \infty$
as $t \to t_*$.
\end{enumerate}
\end{lemma}

\begin{proof}[Proof of Theorem~\ref{thm:steady:complete}(i)]
Using the polynomial form of the steady ODE system~\eqref{eq:steady:poly} we compute (using the normalisation $c=3$) that $\Lambda$ 
satisfies the differential equation
\begin{equation}
\label{eq:Lambda:dot}
\Lambda' = -F_1 \left( \Lambda (1-\Lambda) + (F_2-F_3 - \Lambda)^2 \right).
\end{equation}
In particular,~\eqref{eq:Lambda:dot} implies that $\Lambda' < 0$ whenever
$\Lambda < 1$ (indeed, $\Lambda' \leq 0$ whenever $\Lambda \leq 1$, with
equality if and only if the equations \eqref{eq:explicit:steady:Fi} defining
the path traced by $\mathcal{S}_{\sqrt{2},3}$ hold).

Note that $f_2(0) = f_3(0) = b$, so the smoothly-closing steady soliton
$\mathcal{S}_{b,3}$ satisfies
\begin{equation*}
\Lambda(0) = \frac{2}{b^2} .
\end{equation*}
Hence for $b > \sqrt{2}$ the inequality $\Lambda < 1$ holds at $t=0$, and is
then preserved for all time. 
Therefore the only possible case in Lemma \ref{lem:Lambda_trich} is (i).
\end{proof}

\subsection{Incomplete steady solitons}
\label{ss:incomplete}
It now remains to prove case (iii) of Theorem \ref{thm:steady:complete},
\ie that~$\mathcal{S}_{b,3}$ is forward-incomplete for $b < \sqrt{2}$.
By Lemma \ref{lem:Lambda_trich} it suffices to prove that $\Lambda \ge \frac{2}{b^2} >1$ 
holds throughout the lifetime of any solution with $b < \sqrt{2}$.

Proving that $\Lambda$ is increasing in this situation is a little more
involved than the argument that $\Lambda$ was decreasing in the proof of
Theorem \ref{thm:steady:complete}(i).
First we rewrite the polynomial ODE system~\eqref{eq:steady:poly} entirely in terms of the quantities $\Lambda$, $F_2-F_3$ and $F_1$.
A calculation shows that~\eqref{eq:steady:poly} is equivalent to the system
\begin{subequations}
\label{eq:steady:poly:v2}
\begin{align}
D' &= -\frac{1}{2} F_1 D(D-1) + \frac{3\Lambda}{2F_1} (\Lambda - D),\\
\label{eq:L'}
\Lambda ' &= F_1 \left( (\Lambda-1) D^2 -\Lambda (D-1)^2 \right),\\
\label{eq:F1'}
F_1' & = F_1^2 \left(D- \frac{3}{2}\right) + \frac{1}{2}\Lambda,
\end{align}
\end{subequations}
where to achieve a more compact presentation we have introduced the notation $D:=F_2 - F_3$.
(The first two equations of~\eqref{eq:steady:poly:v2} are satisfied
automatically when $D=F_2-F_3=1 = \Lambda$; integrating the third equation then gives an alternative
derivation of the explicit steady solution $\mathcal{S}_{\sqrt{2},3}$.)

\begin{lemma}
\label{lem:L>D>1}
For a steady soliton with $c_1 =0$, $c_2 = 3$ and $c_3 = -3$ the condition
\begin{equation}
\label{eq:Lambda:D:pos}
\Lambda>D>1
\end{equation}
\begin{enumerate}
\item implies that $\Lambda$ is increasing, and
\item %
is preserved forward in time.
\end{enumerate}
\end{lemma}

\begin{proof}
(i)
First rewrite~\eqref{eq:L'} as
\[ \Lambda' = \left((\Lambda-1)D^2 - \Lambda(D-1)^2\right) F_1 = (-D^2 + 2D\Lambda - \Lambda) F_1. 
\]
Since the quadratic $q(D)=-D^2 + 2D\Lambda - \Lambda$ is an increasing function of $D \in [1,\Lambda]$,
\eqref{eq:Lambda:D:pos} implies 
\begin{equation}
\label{eq:L':ineq}
\Lambda' > q(1) F_1 = (-1 + 2\Lambda - \Lambda) F_1 = (\Lambda - 1) F_1.
\end{equation}
In particular, this implies that $\Lambda$ is strictly increasing on any
connected interval $I$ on which~\eqref{eq:Lambda:D:pos} holds.

(ii) Suppose for a contradiction that~\eqref{eq:Lambda:D:pos} eventually fails
and that $t>t_0$ is the first instant at which it fails. By (i),
$\Lambda$ must be increasing on the interval $[t_0,t)$ and hence
$\Lambda(t) > 1$. 
So $t$ must satisfy one of the following two conditions:
\begin{enumerate}
\item
$D(t)=1$, or
\item
$\Lambda(t) = D(t)$.
\end{enumerate}
In each case we will derive a contradiction. In the first case we have $D > 1$ on $[t_0,t)$ and $D(t)=1$, 
so $D'(t) \le 0$. 
On the other hand, it follows from~\eqref{eq:steady:poly:v2} that at any point where $D=1$ we have
\[
D' = \frac{3}{2F_1} \Lambda (\Lambda -1).
\]
Hence $D'(t)>0$ since $\Lambda(t)>1$. 

In the second case we have $\Lambda-D>0$ on $[t_0,t)$ and $(\Lambda-D)(t)=0$, so $(\Lambda-D)'(t) \le 0$. 
Again using~\eqref{eq:steady:poly:v2} we find that at any point where $D=\Lambda$
\[
(\Lambda-D)' = \frac{3}{2}  \Lambda (\Lambda-1) F_1.
\]
Hence $(\Lambda-D)'(t)>0$ because $\Lambda(t)>1$. 
\end{proof}

\begin{proof}[Proof of Theorem \ref{thm:steady:complete}(iii)]
The small-$t$ expansions for $\mathcal{S}_{b,3}$ given in Appendix~\ref{app:sc:power:series}
imply that 
\[
\Lambda(0) = D (0) = \frac{2}{b^2}, \quad \Lambda - D   = \frac{3(2-b^2)}{5b^6} t^2 + \hot
\]
Hence for any $b^2<2$, the condition \eqref{eq:Lambda:D:pos} holds for $t>0$
sufficiently small. By Lemma \ref{lem:L>D>1} that condition is preserved
and $\Lambda$ is increasing throughout the lifetime of the solution.
Thus $\Lambda \geq \frac{2}{b^2}>1$ for all time, and Lemma \ref{lem:Lambda_trich} implies that $\mathcal{S}_{b,3}$ is incomplete. 
\end{proof}

\begin{remark}
The same type of argument proves the following analogue of Lemma~\ref{lem:L>D>1}(ii): 
\\[0.3em]
\emph{If a solution of~\eqref{eq:steady:poly:v2} satisfies $\Lambda<D<1$ at some instant $t_0$ then 
it continues to satisfy the same inequalities for the remainder of its lifetime. In particular, any solution of~\eqref{eq:steady:poly:v2} 
that arises from a smoothly-closing steady soliton $\mathcal{S}_{b,3}$ 
with $b>\sqrt{2}$ satisfies $\Lambda<D<1$ for all $t>0$.}
\\[0.2em]
Our proof that smoothly-closing steady solitons $\mathcal{S}_{b,3}$ with $b>\sqrt{2}$ are all asymptotically conical already used the 
fact that the condition $\Lambda<1$ is preserved;  we did not need to use that the remaining inequalities $D<1$ and $\Lambda<D$
hold for all $t$.
\end{remark}

\section{Comparisons with 7-dimensional Ricci solitons}
\label{S:comparison}

\noindent
\textbf{Cohomogeneity-one steady Ricci solitons on $\Lambda^2_-\Sph^4$ and $\Lambda^2_-\CP^2$.}
In order to compare the behaviour of Laplacian flow with  Ricci flow in dimension $7$ it is natural
to ask what one can say about cohomogeneity-one Ricci solitons in dimension $7$. 
Given the examples we have studied in this paper, it is particularly natural to compare 
$G$-invariant Ricci solitons on $\Lambda^2_-\Sph^4$ and $\Lambda^2_-\CP^2$ 
for~$G=\Sp{2}$ and $G=\sunitary{3}$ respectively to the complete Laplacian solitons we have found on those spaces. 
We describe known results in this direction.

We begin with results in the \Sp{2}-invariant setting. We are not aware of any results or numerical evidence either for or against the existence of 
complete $\Sp{2}$-invariant gradient AC Ricci shrinkers on $\Lambda^2_-\Sph^4$. However, for complete~\Sp{2}-invariant steady solitons, Wink proved 
the following result.
\begin{theorem}[{\cite[Theorem 3.1]{Wink:IMRN}}]
\label{thm:wink}
There exists a $1$-parameter family of complete \Sp{2}-invariant gradient steady Ricci solitons on $\Lambda^2_-\Sph^4$. 
The metric coefficients $f_1$ and $f_2$ both grow asymptotically like $\sqrt{t}$ where $t$ is the 
arclength parameter along a unit-speed geodesic normal to every principal orbit and the potential for the soliton 
vector field has linear growth with asymptotic slope $-1$. 
\end{theorem}
For comparison, recall that any~\Sp{2}-invariant complete steady Laplacian soliton is trivial.

Theorem \ref{thm:wink} was conjectured by Buzano--Dancer--Wang~\cite{BDGW:steady},  based on numerical investigations they conducted.
In fact, they considered a higher-dimensional generalisation, namely to cohomogeneity-one manifolds of dimension $4m+3$
that are $3$-dimensional disc bundles over $\mathbb{HP}^m$ for $m \ge 1$. In this case the defining  triple of groups $K \subset H \subset G$
is $G=\Sp{m+1}$, $H=\Sp{m} \times \Sp{1}$ and $K=\Sp{m} \times \unitary{1}$. The principal orbit is therefore diffeomorphic to $\CP^{2m+1}$ and the singular orbit is $\mathbb{HP}^m$. 
They conjectured that the same results as stated above hold for any $m\ge 1$.
Wink proved this conjecture for $m \ge 3$~\cite[Theorem A]{Wink:cohom1:Ricci:solitons}, 
and subsequently he proved the conjecture for all $m\ge 1$~\cite[Theorem 3.1]{Wink:IMRN} by a different method.

\bigskip
In the \sunitary{3}-invariant case we are  aware only of the following recent result of H. Chi for Ricci-flat metrics 
(with generic holonomy). 
In particular, nontrivial~\sunitary{3}-invariant steady Ricci solitons do not yet seem to have been studied.
\begin{theorem}\cite[Theorems 1.2 \& 1.5]{Chi:AGAG}
\label{thm:Chi:AC:RF:SU3}
For any $c \in \R$ there is a unique (up to scale) smoothly-closing \sunitary{3}-invariant Ricci-flat metric
$\mathcal{RF}_c$ defined on a neighbourhood of the zero-section of $\Lambda^2_-\CP^2$ satisfying the initial conditions
\[ 
f_1 = t + O(t^3), \quad f_2 = 1 + c t + O(t^2), \quad f_3 = 1 - c t + O(t)^2.
\]
There exists $\epsilon>0$ so that for $\abs{c} < \epsilon$ the smoothly-closing metric $\mathcal{RF}_c$
extends to a complete~\sunitary{3}-invariant AC Ricci-flat metric on $\Lambda^2_-\CP^2$ 
asymptotic to the unique \sunitary{3}-invariant torsion-free~\gtwo-cone.
The metric $\mathcal{RF}_0$ is the Bryant--Salamon~\gtmetric~on  $\Lambda^2_-\CP^2$, while for any
$c\neq 0$ the Ricci-flat metric $\mathcal{RF}_c$ has generic holonomy. 
\end{theorem}
We note that Chi's proof of Theorem~\ref{thm:Chi:AC:RF:SU3} does not give any information 
about the completeness (or otherwise) 
of $\mathcal{RF}_c$ for $\abs{c}$ large; nor are any numerical results presented in \cite{Chi:AGAG} 
that shed any light on this question.
In this sense his result is closely analogous to the result we stated in Remark~\ref{rmk:AC:steady:stability}
for nontrivial steady solitons close to the Bryant--Salamon torsion-free AC\mbox{~\gtstr}.
Recall, however, that in our case Theorem~\ref{thm:steady:complete}
provided a complete understanding of which parameter values of our~\sunitary{3}-invariant smoothly-closing steady solitons lead to
AC solitons and also what type of degeneration occurs at the boundary of the space of smoothly-closing AC solutions. 
Since complete Ricci-flat metrics have at most Euclidean volume growth, any degeneration that could occur in Chi's family 
of smoothly-closing~\sunitary{3}-invariant Ricci-flat metrics would necessarily be quite different from what we encountered in the steady soliton case.

\appendix
\section{Formal power series solutions for smoothly-closing invariant solitons}
\label{app:sc:power:series}

\noindent
\textbf{\texorpdfsunitary{3}-invariant solitons.}
A computer-assisted symbolic computation of the power series expansions around $t=0$ of the $2$-parameter family of smoothly-closing solitons constructed in Theorem~\ref{thm:SU3:smooth:closure} in terms of the real parameters $b$, $c$ and $\lambda$ was performed up to high order. The first several terms are listed below
{\allowdisplaybreaks
\begin{align*}
\small
f_1 &= t - \frac{t^3}{54b^4} \left(4\lambda b^4+9b^2-3c^2\right) + \frac{t^5}{48600 b^8}\left( 464b^8 \lambda^2 + 2844 b^6 \lambda -972 b^4c^2 \lambda + 4050b^4 -2322b^2c^2 +321c^4\right) + \cdots \\
f_2 &= b + \frac{c}{6b}t + \frac{t^2}{72b^3} \left( 4\lambda b^4+18b^2-c^2 \right) - \frac{ct^3}{6480b^5} \left( 152\lambda b^4 + 126b^2-63c^2 
\right) + \cdots \\
f_3 &= b - \frac{c}{6b}t + \frac{t^2}{72b^3} \left( 4\lambda b^4+18b^2-c^2 \right) + \frac{ct^3}{6480b^5} \left( 152\lambda b^4 + 126b^2-63c^2 \right) + \cdots \\
\tau_1 & = -\frac{2(2\lambda b^4-c^2)}{9b^4} t^3 + \frac{2t^5}{405b^8} \left(26b^8\lambda^2 + 81b^6 \lambda - 40b^2c^2\lambda - 54b^2c^2 + 12c^4\right) + \cdots\\
\tau_2 &= c + \frac{2(\lambda b^4+c^2)}{9b^2} t - \frac{ct^2}{54b^4} \left(5\lambda b^4-4c^2\right) - \frac{t^3}{1215b^6} \left( 8b^8\lambda^2+18b^6\lambda + 92b^4c^2\lambda + 99b^2c^2 - 42c^4 \right) + \cdots\\
\tau_3 &= -c + \frac{2(\lambda b^4+c^2)}{9b^2} t +  \frac{ct^2}{54b^4} \left(5\lambda b^4-4c^2\right) 
- \frac{t^3}{1215b^6} \left( 8b^8\lambda^2+18b^6\lambda + 92b^4c^2\lambda + 99b^2c^2 - 42c^4 \right) + \cdots\\
\tb & = \left(\frac{4(\lambda b^4+c^2)}{9b^2}\right)t  - \frac{4t^3}{1215 b^6} \left( 4b^8 \lambda^2 + 144 b^6 \lambda + 46b^4c^2\lambda - 18b^2c^2 -21c^4 \right) + \cdots \\
u & = - \left(\frac{7\lambda b^4 - 2c^2}{9b^4}\right)t + \frac{2t^3}{1215b^8} \left(26b^8 \lambda^2 + 126 b^6 \lambda - 61b^4 c^2 \lambda - 117 b^2 c^2 + 21c^4\right) + \cdots 
\end{align*}}

\noindent
Note the following hold for $t>0$ sufficiently small
\begin{itemize}[left=0pt]
\item
If $b$ and $c$ are assumed to be positive then we have the ordering $f_1<f_3<f_2$.
\item
If $\lambda\le 0$ then $u$ and $\tau_1$ are positive. 
\item
If $\lambda \ge 0$ then $\tb$ is positive.
\end{itemize}

\subsubsection*{Steady solitons}
In the steady case $\lambda=0$ these power series specialise to 

{\allowdisplaybreaks
\begin{align*}
\small
f_1 &= t - \frac{t^3}{18b^4} \left(3b^2-c^2\right) + \frac{\left(  1350b^4 -774b^2c^2 +107c^4\right)t^5}{16200 b^8} + 
\cdots \\
f_2 &= b + \frac{c}{6b}t + \frac{t^2}{72b^3} \left(18b^2-c^2 \right) -\frac{7ct^3}{720b^5} \left( 2b^2-c^2 \right) + \frac{(- 2700b^4 + 636b^2c^2 + 7c^4)t^4}{51840b^7} + \cdots \\
f_3 &= b - \frac{c}{6b}t + \frac{t^2}{72b^3} \left( 18b^2-c^2 \right) + \frac{7ct^3}{720b^5} \left( 2b^2-c^2 \right) + \frac{(- 2700b^4 + 636b^2c^2 + 7c^4)t^4}{51840b^7} + \cdots \\
\tau_1 & = \frac{2c^2}{9b^4} t^3 + \frac{4c^2t^5}{135b^8} \left(- 9b^2 + 2c^2\right) + \cdots\\
\tau_2 &=\phantom{-} c + \frac{2c^2}{9b^2} t + \frac{2c^3t^2}{27b^4} + \frac{c^2 t^3}{405b^6} \left( -33b^2 +14c^2 \right) - \frac{2c^3t^4}{405b^8}(11b^2-3c^2) 
+ \cdots\\
\tau_3 &= -c + \frac{2c^2}{9b^2} t -  \frac{2c^3t^2}{27b^4} + \frac{c^2 t^3}{405b^6} \left( -33b^2 + 14c^2 \right) + \frac{2c^3t^4}{405b^8}(11b^2-3c^2) 
+ \cdots\\
\tb & = \left(\frac{4c^2}{9b^2}\right)t  + \frac{4c^2t^3}{405 b^6} \left( 6b^2 +7c^4 \right) + \cdots \\
u & = \left(\frac{2c^2}{9b^4}\right)t + \frac{2c^2t^3}{405b^8} \left( - 39 b^2 + 7c^2\right) + \cdots .
\end{align*}}
\noindent
Note that the coefficient of $t^k$ in any of the $f_i$ or $\tau_i$ has the general structure
$
t^k b^{1-k} P_k(b/c)
$
for some 1-variable polynomial $P_k$  in the variable $b/c$ with rational coefficients. This is consistent with the scaling behaviour of solitons and the fact that rescaling a steady soliton yields another steady soliton.

\bibliography{g2soliton}

\end{document}